\documentclass{amsart}

\usepackage[english]{babel}

\usepackage{amstext,amsmath,amsthm,amssymb,mathrsfs}
\usepackage{amscd}
\usepackage{verbatim}
\usepackage{float}

\usepackage{mathtools}
\usepackage{esint}

\usepackage{tikz-cd}

\usepackage{faktor}

\usepackage{enumerate}

\usepackage{appendix}

\usepackage{bm}

\usepackage[colorinlistoftodos,prependcaption,textsize=tiny]{todonotes}

\allowdisplaybreaks

\DeclareMathAlphabet{\mathpzc}{OT1}{pzc}{m}{it}

\newcommand{\C}{\mathbb{C}}
\newcommand{\R}{\mathbb{R}}
\newcommand{\N}{\mathbb{N}}
\newcommand{\F}{\mathbb{F}}
\newcommand{\Z}{\mathbb{Z}}
\newcommand{\Q}{\mathbb{Q}}
\newcommand{\E}{\mathbb{E}}

\newcommand{\Prob}{\mathbb{P}}

\newcommand{\ra}{\ensuremath{\Rightarrow}}

\newcommand{\lra}{\ensuremath{\Leftrightarrow}}
\newcommand{\Longra}{\ensuremath{\Longrightarrow}}

\newcommand{\longra}{\ensuremath{\longrightarrow}}

\DeclareMathOperator*{\foo}{\scalerel*{+}{\sum}}

\usepackage{scalerel}

\renewcommand{\vec}[1]{\bm{#1}}

\renewcommand{\Re}{\hbox{\rm Re}\,}

\newcommand{\supp}{\text{\rm supp\,}}

\renewcommand{\emptyset}{\varnothing}

\newcommand{\norm}[1]{||#1||}

\newcommand{\normB}[1]{\Big|\Big|#1\Big|\Big|}
\newcommand{\normb}[1]{\big|\big|#1\big|\big|}
\newcommand{\normm}[1]{\left|\left|\left|#1\right|\right|\right|}
\newcommand{\ip}[2]{\langle #1, #2 \rangle}

\newcommand{\bcdot}{\boldsymbol{\cdot}}

\newcommand {\Schw}{\mathcal{S}}

\newcommand{\ud}{\,\mathrm{d}}

\newcommand{\loc}{{\rm loc}}

\theoremstyle{plain}
\newtheorem{thm}{Theorem}[section]
\newtheorem{theorem}[thm]{Theorem}
\newtheorem{lemma}[thm]{Lemma}
\newtheorem{prop}[thm]{Proposition}
\newtheorem{proposition}[thm]{Proposition}
\newtheorem{cor}[thm]{Corollary}
\newtheorem{corollary}[thm]{Corollary}

\theoremstyle{definition}
\newtheorem{definitie}[thm]{Definition}
\newtheorem{definition}[thm]{Definition}

\newtheorem{example}[thm]{Example}

\theoremstyle{remark}
\newtheorem{remark}[thm]{Remark}

\usepackage{tikz}                           
\usetikzlibrary{shapes.geometric, arrows}

\begin{document}

\title[]{An Intersection Representation for a Class of Anisotropic Vector-valued Function Spaces}

\author{Nick Lindemulder}
\address{Delft Institute of Applied Mathematics\\
Delft University of Technology \\ P.O. Box 5031\\ 2600 GA Delft\\The
Netherlands}
\address{Institute of Analysis \\
Karlsruhe Institute of Technology \\
Englerstra\ss e 2 \\
76131 Karlsruhe\\
Germany}
\email{nick.lindemulder@kit.edu}

\subjclass[2010]{Primary: 46E35, 46E40; Secondary: 46E30}

\keywords{anisotropic, axiomatic approach, Banach space-valued functions and distributions, difference norm, Fubini property, intersection representation, maximal function, quasi-Banach function space}
\date{\today}

\thanks{The author was supported by the Vidi subsidy 639.032.427 of the Netherlands Organisation for Scientific Research (NWO) during his doctorate at Delft University of Technology.}

\begin{abstract}
The main result of this paper is an intersection representation for a class of anisotropic vector-valued function
spaces in an axiomatic setting \`a la Hedberg$\&$Netrusov~\cite{Hedberg&Netrusov2007}, which includes
weighted anisotropic mixed-norm Besov and Lizorkin-Triebel spaces.
In the special case of the classical Lizorkin-Triebel spaces, the intersection
representation gives an improvement of the well-known Fubini property.
The main result has applications in the weighted $L_{q}$-$L_{p}$-maximal regularity problem for parabolic boundary value problems, where weighted anisotropic mixed-norm Lizorkin-Triebel spaces occur as spaces of boundary data.
\end{abstract}

\maketitle

\section{Introduction}\label{IR:sec:intro}

The motivation for this paper comes from the $L_{q}$-$L_{p}$-maximal regularity problem for fully inhomogeneous parabolic boundary value problems, see \cite{Denk&Hieber&Pruess2007,lindemulder2017maximal,Lindemulder&Veraar2018}. In such problems, Lizorkin-Triebel spaces have turned out to naturally occur in the description of the sharp regularity of the boundary data. This goes back to \cite{Weidemaier2002} in the special case that $1 < p \leq q < \infty$ for second order problems with special boundary conditions and was later extended in  \cite{Denk&Hieber&Pruess2007} to the full range $q,p \in (1,\infty)$ for the more general setting of vector-valued parabolic boundary value problems with boundary conditions of Lopatinskii-Shapiro type.
The inevitability of Lizorkin-Triebel spaces for a correct description of the boundary data was reafirmed in \cite{Johnsen&Mucn_Hansen&Sickel2015,JS_traces}, but in a different form on the function space theoretic side. 

On the one hand, in \cite{Denk&Hieber&Pruess2007,Weidemaier2002} the parabolic anisotropic regularity of the boundary data is described by means of an intersection of two function space-valued function spaces, in which the Lizorkin-Triebel space appears as an isotropic vector-valued Lizorkin-Triebel space describing the sharp temporal regularity. 
On the other hand, in \cite{Johnsen&Mucn_Hansen&Sickel2015,JS_traces} the anisotropic structure is dealt with more directly through a Fourier analytic approach, leading to anisotropic mixed-norm Lizorkin-Triebel spaces.
A link between the two approaches was obtained in \cite[Proposition~3.23]{Denk&Kaip2013}, by comparing the trace result \cite[Theorem~2.2]{JS_traces} with a trace result from \cite{Berkolaiko1985,Berkolaiko1987}: for every $q,p \in (1,\infty)$, $a,b \in (0,\infty)$ and $s \in (0,\infty)$, there is the intersection representation
\begin{equation}\label{IR:intro:eq:intersection_rep;Denk&Kaip}
F^{s,(a,b)}_{(p,q),p}(\R^{n} \times \R) =
F^{s/b}_{q,p}(\R;L_{p}(\R^{n})) \cap L_{q}(\R;B^{s/a}_{p,p}(\R^{n})).
\end{equation}

The anisotropic mixed-norm Lizorkin-Triebel space $F^{s,(a,b)}_{(p,q),r}(\R^{n} \times \R)$ for $s \in \R$, 
$r \in [1,\infty]$, is defined analogously to the classical isotropic Lizorkin-Triebel space $F^{s}_{p,r}(\R^{d})$,
but with an underlying Littlewood-Paley decomposition of $\R^{n} \times \R$ that is adapted to the $(a,b)$-anisotropic scalings $\{ \delta_{\lambda}^{(a,b)} : \lambda \in (0,\infty) \}$ given by
\begin{equation}\label{intro:eq:anisotropic_dilation;parabolic}
\delta_{\lambda}^{(a,b)}(\xi,\tau) = (\lambda^{a}\xi,\lambda^{b} \tau), \qquad (\xi,\tau) \in \R^{n} \times \R.
\end{equation}
Intuitively the dilation structure \eqref{intro:eq:anisotropic_dilation;parabolic} causes a decay behaviour on the Fourier side at different rates in the two components of $\R^{n} \times \R$ in such a way that smoothness $s \in (0,\infty)$ with respect to the anisotropy $(a,b)$ corresponds to smoothness $s/a$ in the spatial direction and smoothness $s/b$ in the time direction.
One way to look at the intersection representation \eqref{IR:intro:eq:intersection_rep;Denk&Kaip} is as a way to make this intuition precise for $F^{s,(a,b)}_{(p,q),r}(\R^{n} \times \R)$ in the special case that $r=p$.

It is the goal of this paper to provide a more systematic approach to the intersection representation \eqref{IR:intro:eq:intersection_rep;Denk&Kaip} and obtain more general versions of it, covering the weighted Banach space-valued setting.
In order to do so, we introduce a new class of anisotropic vector-valued function spaces in an axiomatic setting \`a la Hedberg$\&$Netrusov~\cite{Hedberg&Netrusov2007}, which includes Banach space-valued weighted anisotropic mixed-norm Besov and Lizorkin-Triebel spaces (see Section~\ref{IR:sec:def&basic_properties}).

The main result of this paper is an intersection representation for this new class of anisotropic function spaces (see Section~\ref{IR:sec:IR}), from which the following theorem can be obtained as a special case (see Example~\ref{IR:ex:ex:cor:thm:IR;comparY&YL;concrete_examples;scalar_anisotropies}):
\begin{theorem}\label{IR:intro:thm}
Let $a,b \in (0,\infty)$, $p,q \in (1,\infty)$, $r \in [1,\infty]$ and $s \in (0,\infty)$. Then
\begin{equation}\label{IR:eq:intro:thm}
F^{s,(a,b)}_{(p,q),r}(\R^{n} \times \R^{m}) =
\F^{s/b}_{q,r}(\R^{m};L_{p}(\R^{n})) \cap L_{q}(\R^{m};F^{s/a}_{p,r}(\R^{n})),
\end{equation}
where, for $E=L_{p}(\R^{n})$ and $\sigma \in \R$,
\[
\F^{\sigma}_{q,r}(\R^{m};E) = \left\{ f \in \mathcal{S}'(\R^{m};E) : (2^{k\sigma}S_{k}f)_{k} \in L_{q}(\R^{n};E[\ell_{r}(\N)]) \right\}
\]
with $(S_{k})_{k \in \N}$ the Fourier multiplier operators induced by a Littlewood-Paley decomposition of $\R^{m}$ and where 
\[
E[\ell_{r}(\N)] = \left\{ (f_n)_{n} \in E^\N : \norm{(f_n)_n}_{\ell_{r}(\N)} \in E \right\}.
\]
\end{theorem}

In the case $p=r$, Fubini's theorem yields $\F^{s/b}_{q,p}(\R^{m};L_{p}(\R^{n})) = F^{s/b}_{q,p}(\R^{m};L_{p}(\R^{n}))$ and $F^{s/a}_{p,p}(\R^{n})=B^{s/a}_{p,p}(\R^{n})$, and from \eqref{IR:eq:intro:thm} we obtain an extension of the intersection representation~\eqref{IR:intro:eq:intersection_rep;Denk&Kaip} to decompositions $\R^{d}=\R^{n} \times \R^{m}$:
\begin{equation*}\label{IR:intro:eq:intersection_rep;Denk&Kaip;nm}
F^{s,(a,b)}_{(p,q),p}(\R^{n} \times \R^{m}) =
F^{s/b}_{q,p}(\R^{m};L_{p}(\R^{n})) \cap L_{q}(\R^{m};B^{s/a}_{p,p}(\R^{n})).
\end{equation*}

In the special case that $a=b$ and $p=q$, the latter can be viewed as a special instance of the Fubini property. In fact, the main result of this paper, Theorem~\ref{IR:thm:IR}/\ref{IR:thm:IR_Ap-case}, extends the well-known Fubini property for the classical Lizorkin-Triebel spaces $F^{s}_{p,q}(\R^{d})$ (see \cite[Section~4]{Triebel2001_SF} and the references given therein), see Remark~\ref{IR:ex:cor:thm:IR;comparY&YL}.
However, as seen in Theorem~\ref{IR:intro:thm}, the availability of Fubini's theorem is not required
for intersection representations, it should just be thought of as a powerful tool to simplify the function spaces that one has to deal with in the special case that some of the parameters coincide.

As a special case of the general intersection representation from Section~\ref{IR:sec:IR} we also obtain intersection representations for anisotropic mixed-norm Besov spaces (see Example~\ref{IR:ex:ex:cor:thm:IR;comparY&YL;concrete_examples;Fubini;scalar_anisotropies}).
An intersection representation for anisotropic Besov spaces for which the integrability parameter coincides with the microscopic parameter can be found in \cite[Theorem~3.6.3]{Amann09}. 

Let us now give an alternative viewpoint of \eqref{IR:eq:intro:thm} in order to motivate and provide some intuition for the function space theoretic framework of this paper.
First of all, the isotropic $\F^{s/b}_{q,r}$ and $F^{s/a}_{p,r}$ on the right-hand side of \eqref{IR:eq:intro:thm} could be viewed as the anisotropic $\F^{s,b}_{q,r}$ and $F^{s,a}_{p,r}$:
\begin{equation}\label{IR:eq:intro:thm;alternative}
F^{s,(a,b)}_{(p,q),r}(\R^{n} \times \R^{m}) =
\F^{s,b}_{q,r}(\R^{m};L_{p}(\R^{n})) \cap L_{q}(\R^{m};F^{s,a}_{p,r}(\R^{n})).
\end{equation}
As already mentioned above, in this paper we will introduce a new class of anisotropic vector-valued function spaces in an axiomatic setting \`a la Hedberg$\&$Netrusov~\cite{Hedberg&Netrusov2007}. This class of function spaces will be defined in such a way that each of the three spaces in \eqref{IR:eq:intro:thm;alternative} is naturally contained in it. In particular, this will allow us to treat the three function spaces in \eqref{IR:eq:intro:thm;alternative} in the same way 
from a conceptual point of view. 

In order to elaborate a bit on the latter, let us write $d=n+m$, $A=aI_{n}$, $B=bI_{m}$, $\vec{A}=(A,B)$ and let $\R^{d}$ be $(n,m)$-decomposed, i.e.\ $\R^{d}=\R^{n} \times \R^{m}$, with $(n,m)$-anisotropy $\vec{A}$ and the induced one-parameter group of expansive dilations $(\vec{A}_{t})_{t \in \R_{+}}$:
\[
\vec{A}_{t}(x,y) = (A_{t}x,B_{t}y) = (t^{a}x,t^{b}y), \quad (x,y) \in \R^{n} \times \R^{m}=\R^{d}.
\]
Let 
\begin{align*}
E &=L_{(p,q)}(\R^n \times \R^m)[\ell^s_r(\N)] \\
 &= \left\{ (f_n)_{n \in \N} \in L_0(\R^d \times \N) \ : \Big(\sum_{n=0}^{\infty}2^{ns}|f_n|^r \Big)^{1/r} \in  L_{(p,q)}(\R^n \times \R^m) \right\},
\end{align*}
where we use the natural identification $L_0(\R^d)^\N \simeq L_0(\R^d \times \N)$; here, given a measure space $(S,\mathscr{A},\mu)$, $L_0(S)$ stands for the space of equivalence classes of measurable functions from $S$ to $\C$.
Let $E_{(n,m);1}$ and $E_{(n,m);2}$ denote $E$ viewed as Banach function space on $\R^m \times \N \times \R^n$ and $\R^n \times \N \times \R^m$, respectively.
Let $(S^{\vec{A}}_n)_{n \in \N}$ be a Littlewood-Paley decomposition of $\R^d$ with respect to the dilation structure $(\vec{A}_{t})_{t \in \R_{+}}$ induced by the anisotropy $\vec{A}$, let $(S^A_n)_{n \in \N}$ be a Littlewood-Paley decomposition of $\R^m$ with respect to the dilation structure $(A_{t})_{t \in \R_{+}}$ induced by the anisotropy $A$ and let $(S^B_n)_{n \in \N}$ be a Littlewood-Paley decomposition of $\R^n$ with respect to the dilation structure $(B_{t})_{t \in \R_{+}}$ induced by the anisotropy $B$; see Definition~\ref{IR:def:LP-sequences} in the main text.

With the just introduced notation, $F^{s,(a,b)}_{(p,q),r}(\R^{n} \times \R^{m})$ coincides with the space
\begin{equation*}
Y^{\vec{A}}(E) = \{ f \in \mathcal{S}'(\R^d) : (S^{\vec{A}}_n f)_n \in E \},    
\end{equation*}
$\F^{s,b}_{q,r}(\R^{m};L_{p}(\R^{n}))$ can be naturally identified with the space
\begin{equation*}
Y^{B}(E_{(n,m);2}) = \{ f \in L_0(\R^n;\mathcal{S}'(\R^m)) : (S^B_n f)_n \in E_{(n,m);2} \},
\end{equation*}
and $L_{q}(\R^{m};F^{s,a}_{p,r}(\R^{n}))$ can be naturally identified with the space
\begin{equation*}
Y^{A}(E_{(n,m);1}) = \{ f \in L_0(\R^m;\mathcal{S}'(\R^n)) : (S^A_n f)_n \in E_{(n,m);1} \},
\end{equation*}
so that \eqref{IR:eq:intro:thm;alternative} takes the form
\begin{equation}\label{IR:eq:intro:gen_IR_ex}
    Y^{\vec{A}}(E) = Y^{B}(E_{(n,m);2}) \cap  Y^{A}(E_{(n,m);1}).
\end{equation}

Each of the spaces $Y^{\vec{A}}(E)$, $Y^{B}(E_{(n,m);2})$ and $Y^{A}(E_{(n,m);1})$ is defined as a subspace of $L_0(S;\mathcal{S}'(\R^N))$ for some $\sigma$-finite measure space $(S,\mathscr{A},\mu)$, in terms of an anisotropy on $\R^N$ and a Banach function space on $\R^N \times \N \times S$, where we take the trivial measure space $(S,\mathscr{A},\mu) = (\{0\},\{\emptyset,\{0\}\},\#)$ in case of $Y^{\vec{A}}(E)$ above.
Furthermore, we view the Euclidean space $\R^N$ as being decomposed as
$\R^N = \R^{\mathpzc{d}_1} \times \ldots \times \R^{\mathpzc{d}_\ell}$ with $\ell \in \N_{1}$ and $\mathpzc{d}=(\mathpzc{d}_1,\ldots,\mathpzc{d}_\ell) \in (\N_{1})^\ell$, $\mathpzc{d}_1+\ldots+\mathpzc{d}_\ell = N$, where we take $\ell=2$ in case of $Y^{\vec{A}}(E)$ above and take $\ell=1$ in cases of $Y^{B}(E_{(n,m);2})$ and $Y^{A}(E_{(n,m);1})$ above.
This viewpoint naturally leads us to extend the axiomatic approach to function spaces by Hedberg$\&$Netrusov~\cite{Hedberg&Netrusov2007} to the anisotropic mixed-norm setting in which there additionally is some extra underlying measure space $(S,\mathscr{A},\mu)$.
This will give us a general framework that is well suited for a systematic treatment of intersection representations as in Theorem~\ref{IR:intro:thm} as well as extensions to a Banach space-valued setting with Muckenhoupt weights.

One of the main ingredients in the proof of Theorem~\ref{IR:intro:thm} (and the more general intersection representations) is a characterization by differences.
For a function $f \in \R^d \to \C$, $h \in \R^d$ and an integer $M \geq 1$, we write
$$
\Delta_h f\,(x) = f(x+h)-f(x), \qquad x \in \R^d,
$$
and
$$
\Delta^M_h f\,(x) = \underbrace{\Delta_h \cdots \Delta_h}_{M \:\text{times}}f\,(x)
= \sum_{j=0}^{M}(-1)^j{M \choose j}f(x+(M-j)h),  \quad x \in \R^d.
$$

For the special case of the anisotropic mixed-norm Lizorkin-Triebel space $F^{s,\vec{a}}_{\vec{p},q}(\R^{\mathpzc{d}_1} \times \ldots \times \R^{\mathpzc{d}_\ell})$, the difference norm characterization takes the form of Theorem~\ref{IR:thm:intro:difference_norm} below.

Before we state it, let us introduce some notation. 
Let $\ell \in \N_{1}$, $\mathpzc{d} \in (\N_{1})^{\ell}$ with $\mathpzc{d}_1+\ldots+\mathpzc{d}_\ell=d$, $\vec{a} \in (0,\infty)^{\ell}$, $\vec{p} \in (0,\infty)^\ell$, $q \in [1,\infty]$ and $s \in \R$. We put  
$$
F^{s,(\vec{a};\mathpzc{d})}_{\vec{p},q}(\R^{d}) := F^{s,\vec{a}}_{\vec{p},q}(\R^{\mathpzc{d}_1} \times \ldots \times \R^{\mathpzc{d}_\ell})
$$
and
\begin{align*}
&L_{(\vec{p};\mathpzc{d})}(\R^d) 
:= L_{p_\ell}(\R^{\mathpzc{d}_\ell})[\ldots [L_{p_1}(\R^{\mathpzc{d}_1})]\ldots] \\
&\quad= \left\{ f \in L_0(\R^d) : \left(\int_{\R^{\mathpzc{d}_\ell}}\left(\ldots\left(\int_{\R^{\mathpzc{d}_1}}|f(x)|^{p_1}\ud x_1\right)^{p_2/p_1}\ldots\right)^{p_\ell/p_{\ell-1}}\ud x_\ell\right)^{1/p_\ell} < \infty \right\}.
\end{align*}

\begin{theorem}\label{IR:thm:intro:difference_norm}
Let $\ell \in \N_{1}$, $\mathpzc{d} \in (\N_{1})^{\ell}$ with $\mathpzc{d}_1+\ldots+\mathpzc{d}_\ell=d$, $\vec{a} \in (0,\infty)^{\ell}$, $\vec{p} \in (1,\infty)^\ell$, $q \in [1,\infty]$ and $s \in (0,\infty)$. 
Let $\vec{\varphi} \in [1,\infty)^{\ell}$ and $M \in \N$ satisfy $s> \sum_{j=1}^{\ell}a_j\mathpzc{d}_j(1-\varphi_j^{-1})$ and $M\min\{a_1\mathpzc{d}_1,\ldots,a_\ell\mathpzc{d}_\ell\} > s$.
For all $f \in L_{(\vec{p};\mathpzc{d})}(\R^d)$ there is the two sided estimate
$$
\norm{f}_{F^{s,(\vec{a};\mathpzc{d})}_{\vec{p},q}(\R^{d})} \eqsim 
\norm{f}_{L_{(\vec{p};\mathpzc{d})}(\R^d)} + \norm{(2^{ns}d^{(\vec{a};\mathpzc{d}),\vec{\varphi}}_{M,n}(f))_{n \geq 1}}_{L_{(\vec{p};\mathpzc{d})}(\R^d)[\ell_q(\N_{1})])},
$$
where 
\begin{align*}
& d^{(\vec{a};\mathpzc{d}),\vec{\varphi}}_{M,n}(f)(x) := 2^{n\sum_{j=1}^{\ell}a_j\mathpzc{d}_j\varphi_j^{-1}}
\normB{ h \mapsto  1_{B_{\R^{\mathpzc{d}_1}}(0,2^{-na_1}) \times \ldots \times B_{\R^{\mathpzc{d}_\ell}}(0,2^{-na_\ell})}(h)\Delta^M_h f(x)}_{L_{(\vec{\varphi};\mathpzc{d})}(\R^{d})}.
\end{align*}
The implicit constants in this two-sided estimate, which is in (modified) Vinogradov notation for estimates (see the end of this introduction on notation and conventions), only depends on $\mathpzc{d}$, $\vec{a}$, $\vec{p}$, $q$ and $s$.
\end{theorem}

As a special case of the general difference norm results in this paper (see Section~\ref{IR:sec:Diff_norms}), we also have a corresponding version of Theorem~\ref{IR:thm:intro:difference_norm} for $\F^{s,(\vec{a};\mathpzc{d})}_{\vec{p},q}(\R^{d};E)$.
In connection to (the proof of) Theorem~\ref{IR:intro:thm}, this  especially includes $\F^{s/b}_{q,r}(\R^{m};L_{p}(\R^{n}))$. 

Theorem~\ref{IR:thm:intro:difference_norm} is an extension of the difference norm characterization contained in \cite[Theorem~1.1.14]{Hedberg&Netrusov2007} to the anisotropic mixed norm setting, restricted to the special case of Lizorkin-Triebel spaces in the parameter range $\vec{p} \in (1,\infty)^\ell$, $q \in [1,\infty]$. 
However, the range $\vec{p} \in (0,\infty)^\ell$, $q \in (0,\infty]$ are also covered by the general difference norm results in Section~\ref{IR:sec:Diff_norms} for the axiomatic setting considered in this paper. In fact, we cover weighted anisotropic mixed-norm Banach space-valued Besov and Lizorkin-Triebel spaces (both in the normed and quasi-normed parameter ranges). Related estimates involving differences in the isotropic case can be found in e.g.\ \cite{Strichartz1967,Triebel1983_TFS_I,Triebel1992_TFS_II}. 

The following duality result is a special case of a more general duality result from this paper for our abstract class of anisotropic vector-valued function spaces (see Theorem~\ref{IR:thm:duality} and Example~\ref{IR:ex:thm:duality}).
 
\begin{theorem}\label{IR:thm:intro:duality}
Let $X$ be a Banach space, $\mathpzc{d} \in (\N_{1})^{\ell}$ with $\mathpzc{d}_1+\ldots+\mathpzc{d}_\ell=d$, $\vec{a} \in (0,\infty)^{\ell}$, $\vec{p} \in (1,\infty)^{\ell}$, $q \in (1,\infty)$ and $s \in \R$.
Viewing 
$$
[F^{s,(\vec{a},\mathpzc{d})}_{\vec{p},q}(\R^{d};X)]^{*} \hookrightarrow \mathcal{S}'(\R^d;X^{*})
$$
under the natural pairing (induced by $\mathcal{S}'(\R^d;X^{*})=[\mathcal{S}(\R^d;X)]'$, see \cite[Theorem~1.3.1]{Amann2003_Vector-valued_distributions}), there is the identity
$$
[F^{s,(\vec{a},\mathpzc{d})}_{\vec{p},q}(\R^{d};X)]^{*} = F^{-s,(\vec{a},\mathpzc{d})}_{\vec{p}',q'}(\R^{d};X^{*})
$$
with an equivalence of norms.
\end{theorem}

Duality results for the classical isotropic Besov and Lizorkin-Triebel spaces can be found in \cite[Section~2.11.2]{Triebel1983_TFS_I}.
In the Banach space-valued setting, \cite[Theorem~2.3.1]{Amann2003_Vector-valued_distributions} is a duality result for Besov spaces. There the underlying Banach space is assumed to be reflexive or to have a separable dual space, except for the case $p=\infty$ (see \cite[Remark
~2.3.2]{Amann2003_Vector-valued_distributions}). In this paper we obtain a partial extension of \cite[Theorem~2.3.1]{Amann2003_Vector-valued_distributions} to the weighted mixed-norm setting with no assumptions on the Banach space (see Example~\ref{IR:ex:thm:duality}).

The following result is a sum representation for anisotropic mixed-norm Lizorkin-Triebel spaces of negative smoothness, which is a dual version to the intersection representation Theorem~\ref{IR:intro:thm}. 

\begin{corollary}\label{IR:cor:intro:sum}
Let $a,b \in (0,\infty)$, $p,q \in (1,\infty)$, $r \in (1,\infty)$ and $s \in (-\infty,0)$. Then
\begin{equation}\label{IR:eq:cor:intro:sum}
F^{s,(a,b)}_{(p,q),r}(\R^{n} \times \R^{m}) =
\F^{s/b}_{q,r}(\R^{m};L_{p}(\R^{n})) + L_{q}(\R^{m};F^{s/a}_{p,r}(\R^{n})),
\end{equation}
where $\F^{s/b}_{q,r}(\R^{m};L_{p}(\R^{n}))$ is as defined in Theorem~\ref{IR:intro:thm}.
\end{corollary}

The above sum representation is an easy consequence of the intersection representation, the duality results and some basic functional analysis on duals of intersections. 
A sum representation for anisotropic Besov spaces for which the integrability parameter coincides with the microscopic parameter can be found in \cite[Theorem~3.6.6]{Amann09}.

Note that in the special case $r=p$, \eqref{IR:eq:cor:intro:sum} reduces to
$$
F^{s,(a,b)}_{(p,q),p}(\R^{n} \times \R^{m}) =
F^{s/b}_{q,p}(\R^{m};L_{p}(\R^{n})) + L_{q}(\R^{m};F^{s/a}_{p,p}(\R^{n}))
$$
by Fubini's theorem.

\subsection*{Overview.} This paper is organized as follows.
\begin{itemize}
    \item \emph{Section~\ref{IR:sec:prelim}}: We discuss the necessary preliminaries on anisotropy and decomposition, quasi-Banach function spaces, vector-valued functions and distributions, and UMD Banach spaces.
    \item \emph{Section~\ref{IR:sec:def&basic_properties}}: We introduce a new class of anisotropic vector-valued function spaces in an axiomatic setting \`a la Hedberg$\&$Netrusov~\cite{Hedberg&Netrusov2007} and discuss some basic properties of these function spaces. In particular, in Definition~\ref{IR:def:Y} we define the spaces $Y^{\vec{A}}(E;X) \subset L_{0}(S;\mathcal{S}'(\R^{d};X))$ for 'admissable' quasi-Banach function spaces $E$ on $\R^d \times \N \times S$ in the sense of Definition~\ref{IR:def:S}. Proposition~\ref{IR:prop:LP-decomp_characterization} gives a characterization of $Y^{\vec{A}}(E;X)$ in terms of Littlewood-Paley decompositions, which is how Besov and Lizorkin-Triebel spaces are usually defined to begin with. 
    Example~\ref{IR:ex:prop:LP-decomp_characterization} then subsequently gives some concrete examples of $Y^{\vec{A}}(E;X)$, including Besov and Lizorkin-Triebel spaces in different generalities.
    \item \emph{Section~\ref{IR:sec:Diff_norms}}: We derive several estimates for the spaces of measurable functions $YL^{\vec{A}}(E;X)$ and $\widetilde{YL}^{\vec{A}}(E;X)$, including estimates involving differences. The spaces $YL^{\vec{A}}(E;X)$ and $\widetilde{YL}^{\vec{A}}(E;X)$ are defined in Definitions \ref{IR:def:YL} and \ref{IR:def:YL_widetilde}, but coincide with $Y^{\vec{A}}(E;X)$ under the conditions of Theorem~\ref{IR:thm:incl_comparY&YL}. In particular, we obtain difference norm characterizations for $Y^{\vec{A}}(E;X)$ in Corollary~\ref{IR:cor:thm:HN_T1.1.14_YL(_widetilde);Y} and Theorem~\ref{IR:thm:difference_norm_'Ap-case'}. The latter covers Theorem~\ref{IR:thm:intro:difference_norm} as a special case.
    \item \emph{Section~\ref{IR:sec:IR}}: Using the difference norm estimates from Section~\ref{IR:sec:Diff_norms}, we obtain intersection representations for $Y^{\vec{A}}(E;X)$ in the spirit of \eqref{IR:eq:intro:gen_IR_ex} in Corollary~\ref{IR:cor:thm:IR;comparY&YL} and Theorem~\ref{IR:thm:IR_Ap-case} (as well as intersection representations for $YL^{\vec{A}}(E;X)$ and $\widetilde{YL}^{\vec{A}}(E;X)$). In Examples \ref{IR:ex:ex:cor:thm:IR;comparY&YL;concrete_examples} and \ref{IR:ex:ex:cor:thm:IR;comparY&YL;concrete_examples;Fubini} we formulate the intersection representations for concrete choices of $E$, which in particular include the Besov and Lizorkin-Triebel cases.
    Examples~\ref{IR:ex:ex:cor:thm:IR;comparY&YL;concrete_examples} covers Theorem~\ref{IR:intro:thm} as a special case.
    \item \emph{Section~\ref{IR:sec:duality}}: We present a duality result for $Y^{\vec{A}}(E;X)$ in Theorem~\ref{IR:thm:duality}, for which we give concrete examples in Example~\ref{IR:ex:thm:duality}. The latter includes Theorem~\ref{IR:thm:intro:duality}.
    \item \emph{Section~\ref{IR:sec:sum_repr}}: Combining the intersection representation from Section~\ref{IR:sec:IR} with the duality result from Section~\ref{IR:sec:duality}, we obtain a sum representation for $Y^{\vec{A}}(E;X)$ in Corollary~\ref{IR:cor:thm:IR_Ap-case;SR}.  
\end{itemize}

\subsection*{Notation and convention.}
We write: $\N = \{0,1,2,3,\ldots\}$, $\N_k = \{k,k+1,k+2,k+3,\ldots\}$ for $k \in \N$,
$\hat{f}=\mathcal{F}f$, $\Z_{<0}=\{\ldots,-3,-2,-1\}$, $\check{f}=\mathcal{F}^{-1}f$,
$\R_{+}=(0,\infty)$, $\C_{+} =  \{z \in \C : \Re(z)>0\}$, $\ell_{p}^{s}(\N) = \{(a_{n})_{n \in \N} \in \C^{\N} : \sum_{n=0}^{\infty}2^{ns}|a_{n}|^{p} < \infty \}$. Furthermore, $\lfloor x \rfloor \in \N$ denotes the least integer part of $x \in [0,\infty)$.
Given a quasi-Banach space $Y$, we denote by $\mathcal{B}(Y)$ the space of bounded linear operators on $Y$ and we write $B_Y=\{y \in Y : \norm{y} \leq 1 \}$ for the closed unit ball in $Y$.
Throughout the paper, we work over the field of complex scalars and fix a Banach space $X$ and a $\sigma$-finite measure space $(S,\mathscr{A},\mu)$. Given two topological vector spaces $X$ and $Y$, we write $X \hookrightarrow Y$ if $X \subset Y$ and the linear inclusion mapping of $X$ into $Y$ is continuous and we write $X \stackrel{d}{\hookrightarrow} Y$ if $X \hookrightarrow Y$ and $X$ is dense in $Y$.

We use (modified) Vinogradov notation for estimates: $a \lesssim b$ means that there exists a constant $C \in (0,\infty)$ such that $a \leq C b$; $a \lesssim_{p,P} b$ means that there exists a constant $C \in (0,\infty)$, only depending on $p$ and $P$, such that $a \leq C b$; $a \eqsim b$ means $a \lesssim b$ and $b \lesssim a$; $a \eqsim_{p,P} b$ means $a \lesssim_{p,P} b$ and $b \lesssim_{p,P} a$.

We will frequently write something like $\stackrel{(*)}{\leq}$ or $\stackrel{(*)}{\lesssim}$, where $(*)$ for instance refers to an equation, to indicate that we use $(*)$ to get $\leq$ or $\lesssim$, respectively.

\section{Preliminaries}\label{IR:sec:prelim}

\subsection{Anisotropy and decomposition}

\subsubsection{Anisotropy on $\R^{d}$}

An \emph{anisotropy} on $\R^{d}$ 
is a real $d \times d$ matrix $A$ with spectrum $\sigma(A) \subset \C_{+}$. An anisotropy $A$ on $\R^{d}$ gives rise to a one-parameter group of expansive dilations $(A_{t})_{t \in \R_{+}}$ given by
\[
A_{t} = t^{A} = \exp[A\ln(t)], \quad t \in \R_{+},
\]
where $\R_{+}$ is considered as multiplicative group.

In the special case $A=\mathrm{diag}(\vec{a})$ with $\vec{a}=(a_{1},\ldots,a_{d}) \in (0,\infty)^{d}$, the associated one-parameter group of expansive dilations $(A_{t})_{t \in \R_{+}}$ is given by
\[
A_{t} = \exp[A\ln(t)] = \mathrm{diag}(t^{a_{1}},\ldots,t^{a_{d}}), \quad t \in \R_{+}
\]

Given an anisotropy $A$ on $\R^{d}$, an $A$-homogeneous distance function is a Borel measurable mapping $\rho:\R^{d} \longra [0,\infty)$ satisfying
\begin{enumerate}[(i)]
\item $\rho(x)=0$ if and only if $x=0$ (\emph{non-degenerate});
\item $\rho(A_{t}x) = t\rho(x)$ for all $x \in \R^{d}$, $t \in \R_{+}$ (\emph{$(A_{t})_{t \in \R_{+}}$-homogeneous});
\item there exists $c \in [1,\infty)$ so that $\rho(x+y) \leq c(\rho(x)+\rho(y))$ for all $x,y \in \R^{d}$ (\emph{quasi-triangle inequality}). The smallest such $c$ is denoted $c_{\rho}$.
\end{enumerate}

Any two homogeneous quasi-norms $\rho_{1}$, $\rho_{2}$ associated with an anisotropy $A$ on $\R^{d}$ are equivalent in the sense that
\[
\rho_{1}(x) \eqsim_{\rho_{1},\rho_{2}} \rho_{2}(x), \qquad x \in \R^{d}.
\]

If $\rho$ is a quasi-norm associated with an anisotropy $A$ on $\R^{d}$ and $\lambda$ denotes the Lebesgue measure on $\R^{d}$, then $(\R^{d},\rho,\lambda)$ is a space of homogeneous type.

Given an anisotropy $A$ on $\R^{d}$, we define the quasi-norm $\rho_{A}$ associated with $A$ as follows: we put $\rho_{A}(0):=0$ and for $x \in \R^{d} \setminus \{0\}$ we define $\rho_{A}(x)$ to be the unique number $\rho_{A}(x)=\lambda \in (0,\infty)$ for which $A_{\lambda^{-1}}x \in S^{d-1}$, where $S^{d-1}$ denotes the unit sphere in $\R^{d}$. Then
\[
\rho_{A}(x):= \min\{ \lambda > 0 : |A_{\lambda^{-1}}x| \leq 1 \},\qquad x \neq 0.
\]
The quasi-norm $\rho_{A}$ is $C^{\infty}$ on $\R^{d} \setminus \{0\}$.
We write
\[
B^{A}(x,r) := B_{\rho_{A}}(x,r) = \{ y \in \R^{d} : \rho_{A}(x-y) \leq r \}, \qquad x \in \R^{d}, r \in (0,\infty).
\]
We furthermore write $c_{A} := c_{\rho_{A}}$.

Given an anisotropy $A$ on $\R^{d}$, we write
\[
\lambda^{A}_{\min} := \min\{\Re(\lambda) : \lambda \in \sigma(A)\}, \quad \lambda^{A}_{\max} := \max\{ \Re(\lambda) : \lambda \in \sigma(A)\}.
\]
Note that $0<\lambda^{A}_{\min} \leq \lambda^{A}_{\max} < \infty$.
Given $\varepsilon \in (0,\lambda^{A}_{\min})$, it holds that
\[
\begin{array}{llll}
t^{\lambda^{A}_{\min}-\varepsilon}|x| &\lesssim_{\varepsilon} |A_{t}x| &\lesssim_{\varepsilon} t^{\lambda^{A}_{\max}+\varepsilon}|x|,& \quad |t| \geq 1,\\
t^{\lambda^{A}_{\max}+\varepsilon}|x| &\lesssim_{\varepsilon} |A_{t}x| &\lesssim_{\varepsilon} t^{\lambda^{A}_{\min}-\varepsilon}|x|,& \quad |t| \leq 1,\\
\end{array}
\]
and
\[
\begin{array}{llll}
t^{1/(\lambda^{A}_{\max}+\varepsilon)}\rho_{A}(x) &\lesssim_{\varepsilon} \rho_{A}(tx) &\lesssim_{\varepsilon} t^{1/(\lambda^{A}_{\min}-\varepsilon)}\rho_{A}(x),& \quad |t| \geq 1,\\
t^{1/(\lambda^{A}_{\min}-\varepsilon)}\rho_{A}(x) &\lesssim_{\varepsilon} \rho_{A}(tx) &\lesssim_{\varepsilon} t^{1/(\lambda^{A}_{\max}+\varepsilon)}\rho_{A}(x),& \quad |t| \leq 1.\\
\end{array}
\]
Furthermore,
\[
\begin{array}{llll}
\rho_{A}(x)^{\lambda^{A}_{\min}-\varepsilon} &\lesssim_{\varepsilon} |x| &\lesssim_{\varepsilon} \rho_{A}(x)^{\lambda^{A}_{\max}+\varepsilon},& \quad |x| \geq 1,\\
\rho_{A}(x)^{\lambda^{A}_{\max}+\varepsilon} &\lesssim_{\varepsilon} |x| &\lesssim_{\varepsilon} \rho_{A}(x)^{\lambda^{A}_{\min}-\varepsilon},& \quad |x| \leq 1,\\
\end{array}
\]

\subsubsection{$\mathpzc{d}$-Decompositions and anisotropy}

Let $\mathpzc{d} = (\mathpzc{d}_{1},\ldots,\mathpzc{d}_{\ell}) \in (\N_{1})^{\ell}$ be such that $d = |\mathpzc{d}|_{1} = \mathpzc{d}_{1} + \ldots + \mathpzc{d}_{\ell}$. The decomposition
\[
\R^{d} = \R^{\mathpzc{d}_{1}} \times \ldots \times \R^{\mathpzc{d}_{\ell}}.
\]
is called the $\mathpzc{d}$-\emph{decomposition} of $\R^{d}$.
For $x \in \R^{d}$ we accordingly write $x = (x_{1},\ldots,x_{\ell})$ and $x_{j}=(x_{j,1},\ldots,x_{j,\mathpzc{d}_{j}})$, where $x_{j} \in \R^{\mathpzc{d}_{j}}$ and $x_{j,i} \in \R$ $(j=1,\ldots,\ell; i=1,\ldots,\mathpzc{d}_{j})$. We also say that we view $\R^{d}$ as being
$\mathpzc{d}$-\emph{decomposed}. Furthermore, for each $k \in \{1,\ldots,\ell\}$ we define the inclusion map
\begin{equation}\label{eq:prelim:incl_map}
\iota_{k} = \iota_{[\mathpzc{d};k]} : \R^{\mathpzc{d}_{k}} \longra \R^{n},\, x_{k} \mapsto (0,\ldots,0,x_{k},0,\ldots,0),
\end{equation}
and the projection map
\begin{equation*}
\pi_{k} = \pi_{[\mathpzc{d};k]} :  \R^{n} \longra \R^{\mathpzc{d}_{k}},\, x = (x_{1},\ldots,x_{\ell}) \mapsto x_{k}.
\end{equation*}

A $\mathpzc{d}$-anisotropy is a tuple $\vec{A}=(A_{1},\ldots,A_{\ell})$ with each $A_{j}$ an anisotropy on $\R^{\mathpzc{d}_{j}}$. A $\mathpzc{d}$-anisotropy $\vec{A}$ gives rise to a one-parameter group of expansive dilations $(\vec{A}_{t})_{t \in \R_{+}}$ given by
\[
\vec{A}_{t}x = (A_{1,t}x_{1},\ldots,A_{\ell,t}x_{l}), \qquad x \in \R^{d},t \in \R_{+},
\]
where $A_{j,t}=\exp[A_{j}\ln(t)]$. Note that $\vec{A}^{\oplus}:=\oplus_{j=1}^{\ell}A_{j}$ is an anisotropy on $\R^{d}$ with $\vec{A}^{\oplus}_{t}=\vec{A}_{t}$ for every $t \in \R_{+}$.
We define the $\vec{A}^{\oplus}$-homogeneous distance function $\rho_{\vec{A}}$ by
\[
\rho_{\vec{A}}(x) := \max\{\rho_{A_{1}}(x_1),\ldots,\rho_{A_{\ell}}(x_{\ell})\},\qquad x \in \R^{d}.
\]
We write
\[
B^{\vec{A}}(x,R) := B_{\rho_{\vec{A}}}(x,R), \qquad x \in \R^{d}, R \in (0,\infty),
\]
and
\[
B^{\vec{A}}(x,\vec{R}) := B^{A_{1}}(x_{1},R_{1}) \times \ldots \times B^{A_{\ell}}(x_{\ell},R_{\ell}),
\qquad x \in \R^{d}, \vec{R} \in (0,\infty)^{\ell}.
\]
Note that $B^{\vec{A}}(x,R) = B^{\vec{A}}(x,\vec{R})$ when $\vec{R}=(R,\ldots,R)$.

\subsection{Quasi-Banach Function Spaces}

For the theory of quasi-Banach spaces, or more generally, $F$-spaces, we refer the reader to \cite{Kalton2003,Kalton&Peck&James1984}.

Let $Y$ be a vector space. A semi-quasi-norm is a mapping $p:Y \longra [0,\infty)$ with the following two properties:
\begin{itemize}
\item \emph{Homogeneity.} $p(\lambda y) = |\lambda| \cdot\,p(y)$ for all $y \in Y$ and $\lambda \in \C$.
\item \emph{Quasi-triangle inequality.} There exists a finite constant $c \geq 1$ such that, for all $y,z \in Y$,
\[
p(y+z) \leq c[p(y) + p(z)].
\]    
\end{itemize}
A quasi-norm is a semi-quasi-norm $p$ that satisfies:
\begin{itemize}
\item \emph{Definiteness.} If $y \in Y$ satisfies $p(y)=0$, then $y=0$.\
\end{itemize}

Let $Y$ be a vector space and $\kappa \in (0,1]$. A $\kappa$-norm is a function $\normm{\,\cdot\,}:Y \longra [0,\infty)$ with the following three properties:
\begin{itemize}
\item \emph{Homogeneity.} $\normm{\lambda y} = |\lambda| \cdot\,\normm{y}$ for all $y \in Y$ and $\lambda \in \C$.
\item \emph{$\kappa$-triangle inequality.} For all $y,z \in Y$,
\[
\normm{y+z}^{\kappa} \leq \normm{y}^{\kappa} + \normm{z}^{\kappa}.
\]
\item \emph{Definiteness.} If $y \in Y$ satisfies $\normm{y}=0$, then $y=0$.\
\end{itemize}

Note that every $\kappa$-norm is a quasi-norm. The Aoki--Rolewitz theorem \cite{Aoki1942,Rolewicz1957} says that, conversely, given a quasi-normed space $(Y,\norm{\,\cdot\,})$ there exists $r \in (0,1]$ and an $r$-norm $\normm{\,\cdot\,}$ on $Y$ that is equivalent to $\norm{\,\cdot\,}$.

Let $Y$ be a quasi-Banach space with a quasi-norm that is equivalent to some $\kappa$-norm, $\kappa \in (0,1]$. If $(y_{n})_{n} \subset Y$ satisfies $\sum_{n=0}^{\infty}\norm{y_{n}}_{Y}^{\kappa} < \infty$, then $\sum_{n \in \N}y_n$ converges in $Y$ and $\norm{\sum_{n=0}^{\infty}y_{n}}_{Y} \lesssim \left(\sum_{n=0}^{\infty}\norm{y_{n}}^{\kappa}_{Y}\right)^{1/\kappa}$.

Let $(T,\mathscr{B},\nu)$ be a $\sigma$-finite measure space. A quasi-Banach function space $F$ on $T$ is an order ideal in $L_{0}(T)$ that has been equipped with a quasi-Banach norm $\norm{\,\cdot\,}$ with the property that $\norm{\,|f|\,} = \norm{f}$ for all $f \in F$.

A quasi-Banach function space $F$ on $T$ has the Fatou property if and only if, for every increasing sequence $(f_{n})_{n \in \N}$ in $F$ with supremum $f$ in $L_{0}(T)$ and $\sup_{n \in \N}\norm{f_{n}}_{F} < \infty$, it holds that $f \in F$ with $\norm{f}_{F} = \sup_{n \in \N}\norm{f_{n}}_{F}$.

%
%

\begin{lemma}\label{IR:lemma:qB_via_Fatou}
Let $V$ be a quasi-normed space continuously embedded into a complete topological vector space $W$.
Suppose that $V$ has the \emph{Fatou property} with respect to $W$, i.e.\ for all $(v_{n})_{n \in \N} \subset V$ the following implication holds:
\[
\lim_{n \to \infty}v_{n} = v \:\mbox{in}\: W,\: \liminf_{n \to \infty}\norm{v_{n}}_{V} < \infty \:\Longra\:
v \in V,\: \norm{f}_{V} \leq \liminf_{n \to \infty}\norm{f_{n}}_{V}.
\]
Then $V$ is complete.
\end{lemma}

\subsection{Vector-valued Functions and Distributions}

As general reference to the theory of vector-valued distributions we mention \cite{Amann2003_Vector-valued_distributions} and \cite{Schwartz1966}.

Let $G$ be a topological vector space.
The space of $G$-valued tempered distributions $\mathcal{S}'(\R^{d};G)$ is defined as $\mathcal{S}'(\R^{d};G) := \mathcal{L}(\mathcal{S}(\R^{d}),G)$, the space of continuous linear operators from the Schwartz space $\mathcal{S}(\R^{d})$ to $G$. In this chapter we equip $\mathcal{S}'(\R^{d};G)$ with the topology of pointwise convergence.
Standard operators (derivatives, Fourier transform, convolution, etc.) on $\mathcal{S}'(\R^{d};G)$ can be defined as in the scalar case.

By a combination of \cite[Theorem~1.4.3]{Amann2003_Vector-valued_distributions} and (the proof of) \cite[Lemma~1.4.6]{Amann2003_Vector-valued_distributions}, the space of finite rank operators $\mathcal{S}'(\R^{d}) \otimes G$ is sequentially dense in $\mathcal{S}'(\R^{d};G)$.
Furthermore, as a consequence of the Banach-Steinhaus Theorem (see \cite[Theorem~2.8]{Ru91}), if $G$ is sequentially complete, then so is $\mathcal{S}'(\R^{d};G)$.

Given a quasi-Banach space $X$, denote by $\mathscr{O}_{\mathrm{M}}(\R^d;X)$ the space of slowly increasing smooth functions on $\R^d$. This means that $f \in \mathscr{O}_{\mathrm{M}}(\R^d;X)$ if and only if $f \in C^\infty(\R^d;X)$ and, for each $\alpha \in \N^d$, there exist $m_\alpha \in \N$ and $c_\alpha >0 $ such that
$$
\norm{D^\alpha f(x)}_{X} \leq c_\alpha (1+|x|^2)^{m_\alpha}, \qquad x \in \R^d. 
$$
The topology of $\mathscr{O}_{\mathrm{M}}(\R^d;X)$ is induced by the family of semi-quasi-norms
$$
p_{\phi,\alpha}(f) := \norm{\phi D^\alpha f}_\infty, \qquad \phi \in \mathcal{S}(\R^d), \alpha \in \N^d.
$$

Let $(T,\mathscr{B},\nu)$ be a $\sigma$-finite measure space and let $G$ be a topological vector space.
We define $L_{0}(T;G)$ as the space as of all $\nu$-a.e.\ equivalence classes of $\nu$-strongly measurable functions $f:T \to G$.
Suppose there is a system $\mathcal{Q}$ of semi-quasi-norms generating the topology of $G$.
We equip $L_{0}(T;G)$ with the topology generated by the translation invariant pseudo-metrics
\[
\rho_{B,q}(f,g) := \int_{B}(q(f-g) \wedge 1)\,\mathrm{d}\nu, \qquad B \in \mathscr{B}, \nu(B) < \infty, q \in \mathcal{Q}.
\]
This topological vector space topology on $L_{0}(T;G)$ is independent of $\mathcal{Q}$ and is called the topology of convergence in measure.
Note that $L_{0}(T) \otimes G$ is sequentially dense in $L_{0}(T;G)$ as a consequence of the dominated convergence theorem and the definitions.

If $G$ is an $F$-space, then $L_{0}(T;G)$ is an $F$-space as well. Here we could for example take $G=L_{\vec{r},\mathpzc{d},\loc}(\R^{d};X)$ with $\vec{r} \in (0,\infty]^{\ell}$ and $X$ a Banach space,
where
\[
L_{\vec{r},\mathpzc{d},\loc}(\R^{d}) = \left\{ f \in L_{0}(\R^{d}) : f1_{B} \in L_{\vec{r},\mathpzc{d}}(\R^{d}), B \subset \R^{d} \:\text{bounded Borel}\: \right\}
\]
and
\[
L_{\vec{r},\mathpzc{d}}(\R^{d}) = L_{r_{\ell}}(\R^{\mathpzc{d}_{\ell}})[\ldots[L_{r_{1}}(\R^{\mathpzc{d}_{1}})]\ldots].
\]

Let $X$ be a Banach space. Then $L_{0}(T) \otimes \mathcal{S}'(\R^{d}) \otimes X$ is sequentially dense in both of $L_{0}(T;\mathcal{S}'(\R^{d};X))$ and $\mathcal{S}'(\R^{d};L_{0}(T;X))$, while the two induced topologies on $L_{0}(T) \otimes \mathcal{S}'(\R^{d}) \otimes X$ coincide. Therefore, we can naturally identify
\[
L_{0}(T;\mathcal{S}'(\R^{d};X)) \cong \mathcal{S}'(\R^{d};L_{0}(T;X)).
\]

A function $g:T \longra X^{*}$ is called \emph{$\sigma(X^{*},X)$-measurable} (or \emph{$X$-weakly measurable}) if $\ip{x}{g}:T \longra \C$ is measurable for all $x \in X$.
We denote by $L^{0}(T;X^{*},\sigma(X^{*},X))$ the vector space of all $\mu$-a.e.\
equivalence classes of \emph{$\sigma(X^{*},X)$-measurable} functions $g:T \longra X^{*}$.

As in \cite{Nowak2000}, we may define the \emph{abstract norm} $\vartheta:L_{0}(T;X^{*},\sigma(X^{*},X)) \longra L_{0}(T)$ by
\[
\vartheta(g) := \sup\{ \,|\ip{x}{g}| \,:\, x \in B_{X} \}, \quad\quad g \in L_{0}(T;X^{*},\sigma(X^{*},X)).
\]
Note that $L_{0}(T;X^{*}) \subset L_{0}(T;X^{*},\sigma(X^{*},X))$ and that $\vartheta(g) = \norm{g}_{X^{*}}$ for all $g \in L_{0}(T;X^{*})$.

We equip $L_{0}(T;X^{*},\sigma(X^{*},X))$ with the topology generated by the system of translation invariant pseudo-metrics
\[
\rho_{B}(f,g) := \int_{B}(\vartheta(f-g) \wedge 1)\mathrm{d}\nu, \qquad B \in \mathscr{B}, \nu(B) < \infty.
\]
In this way, $L_{0}(T;X^{*},\sigma(X^{*},X))$ becomes a topological vector space.

For a Banach function space $E$ on $T$ we define $E(X^{*},\sigma(X^{*},X))$ by
\[
E(X^{*},\sigma(X^{*},X)) := \{ f \in  L_{0}(T;X^{*},\sigma(X^{*},X)) : \vartheta(f) \in E\}.
\]
Endowed with the norm
\[
\norm{f}_{E(X^{*},\sigma(X^{*},X))} := \norm{\vartheta(f)}_{E},
\]
$E(X^{*},\sigma(X^{*},X))$ becomes a Banach space.

Let $E$ be a Banach function space on $T$ with an order continuous norm and a weak order unit (i.e.\ an element $\xi \in E$ with $\xi > 0$ pointwise a.e.). Then (see \cite{Nowak2000})
\[
[E(X)]^{*} = E^{\times}(X^{*},\sigma(X^{*},X))
\]
under the natural pairing,
where $E^{\times}$ is the K\"othe dual of $E$ given by
\[
E^{\times} = \{ g \in L_{0}(T) : \forall f \in E, fg \in L_{1}(T) \}, \quad \norm{g}_{E^{\times}} = \sup_{f \in E, \norm{f}_{E} \leq 1}\int_{T} fg \,\mathrm{d}\nu.
\]
Moreover, if $X^{*}$ has the Radon--Nykod\'ym property with respect to $\nu$, then
\[
[E(X)]^{*} = E^{\times}(X^{*},\sigma(X^{*},X)) = E^{\times}(X^{*}).
\]

\section{Definitions and Basic Properties}\label{IR:sec:def&basic_properties}

Suppose that $\R^{d}$ is $\mathpzc{d}$-decomposed with $\mathpzc{d} \in (\N_{1})^{\ell}$ and let $\vec{A}=(A_{1},\ldots,A_{\ell})$ be a $\mathpzc{d}$-anisotropy. Let $X$ be a Banach space, $(S,\mathscr{A},\mu)$ a $\sigma$-finite measure space, $\varepsilon_{+},\varepsilon_{-} \in \R$ and $\vec{r} \in (0,\infty)^{\ell}$.

For $j \in \{1,\ldots,\ell\}$, we define the maximal function operator $M^{A_{j}}_{r_{j};[\mathpzc{d};j]}$ on $L_{0}(S \times \R^{d})$ by
\[
M^{A_{j}}_{r_{j};[\mathpzc{d};j]}(f)(s,x) := \sup_{\delta > 0}\fint_{B^{A_{j}}(0,\delta)}|f(s,x+\iota_{[\mathpzc{d};j]}y_{j})|\ud y_{j},
\]
where $\iota_{[\mathpzc{d};j]}:\R^{\mathpzc{d}_j} \to \R^d$ is the inclusion mapping from \eqref{eq:prelim:incl_map}.
We define the maximal function operator $M^{\vec{A}}_{\vec{r}}$ by iteration:
\[
M^{\vec{A}}_{\vec{r}}(f) := M^{A_{\ell}}_{r_{\ell};[\mathpzc{d};\ell]}(\ldots(M^{A_{1}}_{r_{1};[\mathpzc{d};1]}(f))\ldots).
\]
We write $M^{\vec{A}} := M^{\vec{A}}_{\vec{1}}$.

The following definition is an extension of \cite[Definition~1.1.1]{Hedberg&Netrusov2007} to the anisotropic setting with some extra underlying measure space $(S,\mathscr{A},\mu)$. The extra measure space provides the right setting for intersection representations, see Section~\ref{IR:sec:IR}.

\begin{definitie}\label{IR:def:S}
We define $\mathcal{S}(\varepsilon_{+},\varepsilon_{-},\vec{A},\vec{r},(S,\mathscr{A},\mu))$ as the set of all quasi-Banach function spaces $E$ on $\R^{d} \times \N \times S$ with the Fatou property for which the following two properties are fulfilled:
\begin{enumerate}[(a)]
\item $S_{+},S_{-}$, the left respectively right shift on $\N$, are bounded on $E$ with
\[
\norm{(S_{+})^{k}}_{\mathcal{B}(E)} \lesssim 2^{-\varepsilon_{+}k} \quad \text{and} \quad \norm{(S_{-})^{k}}_{\mathcal{B}(E)} \lesssim 2^{\varepsilon_{-}k}, \qquad k \in \N.
\]
\item\label{IR:it:def:S:M} $M^{\vec{A}}_{\vec{r}}$ is bounded on $E$:
\[
\norm{M^{\vec{A}}_{\vec{r}}(f_{n})}_{E} \lesssim \norm{(f_{n})}_{E}, \qquad (f_{n}) \in E.
\]
\end{enumerate}
We similarly define $\mathcal{S}(\varepsilon_{+},\varepsilon_{-},\vec{A},\vec{r})$ without the presence of $(S,\mathscr{A},\mu)$, or equivalently, $\mathcal{S}(\varepsilon_{+},\varepsilon_{-},\vec{A},\vec{r}) = \mathcal{S}(\varepsilon_{+},\varepsilon_{-},\vec{A},\vec{r},(\{0\},\{\emptyset,\{0\}\},\#))$.
\end{definitie}

\begin{remark}\label{IR:rmk:def:S;epsilon}
Note that $\varepsilon_{+} \leq \varepsilon_{-}$ when $E \neq \{0\}$, which can be seen by considering $(S_{+})^{k} \circ (S_{-})^{k}$, $k \in \N$.
\end{remark}

\begin{remark}\label{IR:rmk:def:S;monotonicity_r}
Note that
\[
\mathcal{S}(\varepsilon_{+},\varepsilon_{-},\vec{A},\vec{r},(S,\mathscr{A},\mu)) \subset
\mathcal{S}(\varepsilon_{+},\varepsilon_{-},\vec{A},\tilde{\vec{r}},(S,\mathscr{A},\mu)), \quad \vec{r} \geq \tilde{\vec{r}}.
\]
\end{remark}

\begin{example}\label{IR:ex:def:S;classical_isotropic}
Suppose that $\ell=1$ and $\vec{A}=A=I_d$, so that we are in the classical isotropic setting. Then $\vec{r}=r \in (0,\infty)$ and 
\[
M^{\vec{A}}_{\vec{r}}(f)\,(x) = M_r(f)\,(x) = \sup_{\delta > 0} \Big(\fint_{B(0,\delta)}|f(x+y)|^{r}\ud y \Big)^{1/r}
\]
on $L_0(\R^d)$. By the Fefferman-Stein vector-valued maximal inequality (see e.g. \cite[Section~1.2.3]{Triebel1983_TFS_I}) and the Hardy-Littlewood maximal inequality, we thus obtain the following examples.
\begin{enumerate}[(i)]
\item\label{IR:it:ex:def:S;2;classical_isotropic} Let $p \in (0,\infty)$, $q \in (0,\infty]$ and $s \in \R$.
    If $r \in (0,\infty)$ is such that $r < p \wedge q$, then
    \[
    E=L_{p}(\R^{d})[\ell^{s}_{q}(\N)] \in
    \mathcal{S}(s,s,I_d,r).
    \]
\item\label{IR:it:ex:def:S;3;classical_isotropic} Let $p \in (0,\infty)$, $q \in (0,\infty]$ and $s \in \R$.
    If $r \in (0,\infty)$ is such that $r < p$, then
    \[
    E=\ell^{s}_{q}(\N)[L_{p}(\R^{d})] \in
    \mathcal{S}(s,s,I_d,r).
    \]
\end{enumerate}
\end{example}

The following example generalizes the previous example to the anisotropic weighted mixed-norm setting. Furthermore, it also goes beyond the case of a trivial underlying measure space $(S,\mathscr{A},\mu)$. 

\begin{example}\label{IR:ex:def:S}
Let us give some concrete choices of $E \in \mathcal{S}(\varepsilon_{+},\varepsilon_{-},\vec{A},\vec{r},(S,\mathscr{A},\mu))$.
Condition~\eqref{IR:it:def:S:M} in Definition~\ref{IR:def:S} can be covered by means of the lattice Hardy–Littlewood maximal function operator: if $F$ is a UMD Banach function space on $S$, $A$ an anisotropy, $p \in (1,\infty)$, and $w \in A_{p}(\R^{d},A)$ then (see \cite{Bourgain1984,Garcia-Cuerva&Macias&Torrea1993,Hanninen&Lorist2019,Rubio_de_Francia1986,Tozoni1995})
\[
Mf\,(x) := \sup_{\delta > 0}\fint_{B^{A}(x,\delta)} |f(y)|\,dy
\]
defines a bounded sublinear operator on $L_{p}(\R^{d},w;F) = L_{p}(\R^{d},w)[F]$.
The latter induces a bounded sublinear operator on $L_{p}(\R^{d},w)[F[\ell_{\infty}]]$ in the natural way.
Let us furthermore remark that the mixed-norm space $F[G]$ of two UMD Banach function spaces $F$ and $G$ is again a UMD Banach function space (see \cite[page~214]{Rubio_de_Francia1986}).
This leads to the following examples:
\begin{enumerate}[(i)]
\item\label{IR:it:ex:def:S;2} Let $\vec{p} \in (0,\infty)^{\ell}$, $q \in (0,\infty]$, $\vec{w} \in \prod_{j=1}^{\ell}A_{\infty}(\R^{\mathpzc{d}_{j}},A_{j})$ and $s \in \R$.
    If $\vec{r} \in (0,\infty)^{\ell}$ is such that $r_{j} < p_{1} \wedge \ldots \wedge p_{j} \wedge q$ for $j=1,\ldots,\ell$ and
    $\vec{w} \in \prod_{j=1}^{\ell}A_{p_{j}/r_{j}}(\R^{\mathpzc{d}_{j}},A_{j})$, then
    \[
    E=L_{\vec{p}}(\R^{d},\vec{w})[\ell^{s}_{q}(\N)] \in
    \mathcal{S}(s,s,\vec{A},\vec{r}).
    \]
\item\label{IR:it:ex:def:S;3} Let $\vec{p} \in (0,\infty)^{\ell}$, $q \in (0,\infty]$, $\vec{w} \in \prod_{j=1}^{\ell}A_{\infty}(\R^{\mathpzc{d}_{j}},A_{j})$ and $s \in \R$.
    If $\vec{r} \in (0,\infty)^{\ell}$ is such that $r_{j} < p_{1} \wedge \ldots \wedge p_{j}$ for $j=1,\ldots,\ell$ and $\vec{w} \in \prod_{j=1}^{\ell}A_{p_{j}/r_{j}}(\R^{\mathpzc{d}_{j}},A_{j})$, then
    \[
    E=\ell^{s}_{q}(\N)[L_{\vec{p}}(\R^{d},\vec{w})] \in
    \mathcal{S}(s,s,\vec{A},\vec{r}).
    \]
\item\label{IR:it:ex:def:S;1} Let $\vec{p} \in (0,\infty)^{\ell}$, $q \in (0,\infty]$ and $\vec{w} \in \prod_{j=1}^{\ell}A_{\infty}(\R^{\mathpzc{d}_{j}},A_{j})$, $s \in \R$ and $F$ a quasi-Banach function space on $S$. If $\vec{r} \in (0,\infty)^{\ell}$ is such that $r_{j} < p_{1} \wedge \ldots \wedge p_{j} \wedge q$ for $j=1,\ldots,\ell$ and $\vec{w} \in \prod_{j=1}^{\ell}A_{p_{j}/r_{j}}(\R^{\mathpzc{d}_{j}},A_{j})$ and $F^{[\vec{r}_{\max}]}$ is a UMD Banach function space, where
    \begin{equation*}\label{IR:prelim:concavication}
    F^{[r]} := \{ f \in L_{0}(S) : |f|^{1/r} \in F \}, \qquad \norm{f}_{F^{[r]}} := \norm{|f|^{1/r}}_{F}^{r},
    \end{equation*}
     then
    \[
    E=L_{\vec{p}}(\R^{d},\vec{w})[F[\ell^{s}_{q}(\N)]] \in
    \mathcal{S}(s,s,\vec{A},\vec{r},(S,\mathscr{A},\mu)).
    \]
\end{enumerate}
\end{example}

\begin{remark}\label{IR:rmk:ex:def:S}
Note that we can take $\vec{r} = \vec{1}$ in Example~\ref{IR:ex:def:S} when, in each of the corresponding examples:
\begin{enumerate}[(i)]
    \item $\vec{p} \in (1,\infty)^\ell$, $q \in (1,\infty]$ and $\vec{w} \in \prod_{j=1}^{\ell}A_{p_j}(\R^{\mathpzc{d}_j},A_j)$;
    \item $\vec{p} \in (1,\infty)^\ell$, $q \in (0,\infty]$ and $\vec{w} \in \prod_{j=1}^{\ell}A_{p_j}(\R^{\mathpzc{d}_j},A_j)$;
    \item $\vec{p} \in (1,\infty)^\ell$, $q \in (1,\infty]$, $\vec{w} \in \prod_{j=1}^{\ell}A_{p_j}(\R^{\mathpzc{d}_j},A_j)$ and $F$ is a UMD Banach function space.
\end{enumerate}
\end{remark}

For a quasi-Banach function space $E$ on $\R^{d} \times \N \times S$ we define the quasi-Banach function space $E^{\vec{A}}_{\otimes}$ on $S$ by
\[
\norm{f}_{E^{\vec{A}}_{\otimes}} := \norm{ 1_{B^{\vec{A}}(0,1) \times \{0\}} \otimes f}_{E}, \qquad f \in L_{0}(S).
\]
Note that $E^{\vec{A}}_{\otimes} \cong \C$ in case that $(S,\mathscr{A},\mu) = (\{0\},\{\emptyset,\{0\}\},\#)$.

\begin{example}\label{IR:ex:qBFS_part_on_S}
In the situation of Example~\ref{IR:ex:def:S}.\eqref{IR:it:ex:def:S;1}, $E^{\vec{A}}_{\otimes} = F$ with
$$
\norm{f}_{E^{\vec{A}}_{\otimes}} = \norm{ 1_{B^{\vec{A}}(0,1)}}_{L_{\vec{p}}(\R^{d},\vec{w})} \norm{f}_F, \qquad f \in F.
$$
\end{example}

Let $\vec{p} \in (0,\infty)^{\ell}$ and $w:[1,\infty)^{\ell} \to (0,\infty)$. We define the quasi-Banach function space
\begin{equation}\label{IR:prelim:BFS_weighted_Beurling}
B^{\vec{p},w}_{\vec{A}} := \big\{ f \in L_{0}(S) : \sup_{\vec{R} \in [1,\infty)^{\ell}}w(\vec{R})\norm{f}_{L_{\vec{p},\mathpzc{d}}(B^{\vec{A}}(0,\vec{R}))} < \infty \big\}
\end{equation}
which is an extension of (a slight variant of) the space $B^{p}$ considered by Beurling in \cite{Beurling1964} (see \cite{Rafeiro&Samento&Samko2013}).

Let $\vec{p},\vec{q} \in (0,\infty)^{\ell}$.
We define $w_{\vec{A},\vec{q}}:[1,\infty)^{\ell} \to \R_{+}$ by
\[
w_{\vec{A},\vec{q}}(\vec{R}) := \vec{R}^{-\mathrm{tr}(\vec{A})\vec{q}^{-1}}
= \prod_{j=1}^{\ell}R_{j}^{-\mathrm{tr}(A_{j})/q_{j}}, \qquad \vec{R} \in [1,\infty)^{\ell}.
\]
The quasi-Banach function space $B^{\vec{p},w_{\vec{A},\vec{q}}}_{\vec{A}} \hookrightarrow L_{\vec{p},\mathpzc{d},\loc}(\R^{d})$ introduced in \eqref{IR:prelim:BFS_weighted_Beurling} will be convenient to formulate some of the estimates we will obtain.
Note that, if $\vec{p} \in [1,\infty)^{\ell}$, then
\[
B^{\vec{p},w_{\vec{A},\vec{q}}}_{\vec{A}}(X) \hookrightarrow \mathcal{S}'(\R^{d};X).
\]

\begin{lemma}\label{IR:lemma:HN_L1.1.4}
Let $E \in \mathcal{S}(\varepsilon_{+},\varepsilon_{-},\vec{A},\vec{r},(S,\mathscr{A},\mu))$ and $\lambda \in (-\infty,\varepsilon_{+})$.
For $F=(f_{n})_{n} \in E$ and $g:= \sum_{n=0}^{\infty}2^{n\lambda}|f_{n}|$ we have
\begin{equation}\label{IR:eq:lemma:HN_L1.1.4;1}
\norm{(\delta_{0,n}g)_{n}}_{E} \lesssim \norm{F}_{E}.
\end{equation}
Moreover, $g \in E^{\vec{A}}_{\otimes}[B^{\vec{r},w_{\vec{A},\vec{r}}}_{\vec{A}}] \hookrightarrow E^{\vec{A}}_{\otimes}[L_{\vec{r},\mathpzc{d},\mathrm{loc}}(\R^{d})]$ with
\begin{equation}\label{IR:eq:lemma:HN_L1.1.4;3}
\norm{g}_{E^{\vec{A}}_{\otimes}[B^{\vec{r},w_{\vec{A},\vec{r}}}_{\vec{A}}]} \lesssim \norm{F}_{E}.
\end{equation}
%
\end{lemma}

\begin{remark}\label{IR:rmk:lemma:HN_L1.1.4;abs_conv_sum}
Suppose that $\varepsilon_{+}>0$ and $\lambda \in (0,\varepsilon_{+})$ in Lemma~\ref{IR:lemma:HN_L1.1.4}.
Let $\kappa \in (0,1]$ with $\kappa \leq \vec{r}_{\min}$ be such that $\norm{\,\cdot\,}_{E}$ is a equivalent to a $\kappa$-norm.
Then, in particular, $2^{n\lambda}f_{n} \in E^{\vec{A}}_{\otimes}[B^{\vec{r},w_{\vec{A},\vec{r}}}_{\vec{A}}]$ with $\norm{2^{n\lambda}f_{n}}_{E^{\vec{A}}_{\otimes}[B^{\vec{r},w_{\vec{A},\vec{r}}}_{\vec{A}}]} \lesssim \norm{F}_{E}$, so that
\[
\sum_{n=0}^{\infty}\norm{f_{n}}_{E^{\vec{A}}_{\otimes}[B^{\vec{r},w_{\vec{A},\vec{r}}}_{\vec{A}}]}^{\kappa} =
\sum_{n=0}^{\infty}2^{-n\lambda\kappa}\norm{2^{n\lambda}f_{n}}_{E^{\vec{A}}_{\otimes}[B^{\vec{r},w_{\vec{A},\vec{r}}}_{\vec{A}}]}^{\kappa}
\lesssim \sum_{n=0}^{\infty}2^{-n\lambda\kappa}\norm{F}_{E} \lesssim \norm{F}_{E}.
\]
\end{remark}

\begin{remark}\label{IR:rmk:lemma:HN_L1.1.4}
Let $E \in \mathcal{S}(\varepsilon_{+},\varepsilon_{-},\vec{A},\vec{r},(S,\mathscr{A},\mu))$.
Similarly to the proof of Lemma~\ref{IR:lemma:HN_L1.1.4} (but simpler) it can be shown that
\[
E_{i} \hookrightarrow E^{\vec{A}}_{\otimes}[B^{\vec{r},w_{\vec{A},\vec{r}}}_{\vec{A}}].
\]
\end{remark}

\begin{proof}[Proof of Lemma~\ref{IR:lemma:HN_L1.1.4}]
This can be shown similarly to \cite[Lemma~1.1.4]{Hedberg&Netrusov2007}. Let us just provide the details for \eqref{IR:eq:lemma:HN_L1.1.4;3}.
As $|B^{A_{j}}(x_{j},R_{j})| \eqsim R_{j}^{\mathrm{tr}(A_{j})/r_{j}}$, $j=1,\ldots,\ell$, for any $x \in \R^{d}$ and $\vec{R} \in (0,\infty)^{\ell}$,
we have
\[
1_{B^{\vec{A}}(0,\vec{R})} \otimes \norm{g}_{L_{\vec{r},\mathpzc{d}}(B^{\vec{A}}(0,\vec{R}))} \lesssim \prod_{j=1}^{\ell}R_{j}^{\mathrm{tr}(A_{j})/r_{j}}M^{\vec{A}}_{\vec{r}}(g), \qquad \vec{R} \in [1,\infty)^{\ell}.
\]
Therefore,
\[
1_{B^{\vec{A}}(0,\vec{1})} \otimes w_{\vec{A},\vec{r}}(\vec{R})\norm{g}_{L_{\vec{r},\mathpzc{d}}(B^{\vec{A}}(0,\vec{R}))} \lesssim
M^{\vec{A}}_{\vec{r}}(g), \qquad \vec{R} \in [1,\infty)^{\ell},
\]
so that
\[
1_{B^{\vec{A}}(0,\vec{1})} \otimes \norm{g}_{B^{\vec{r},w_{\vec{A},\vec{r}}}_{\vec{A}}} \lesssim
M^{\vec{A}}_{\vec{r}}(g).
\]
It thus follows that
\begin{align*}
\norm{g}_{E^{\vec{A}}_{\otimes}[B^{\vec{r},w_{\vec{A},\vec{r}}}_{\vec{A}}]}
&= \normb{1_{B^{\vec{A}}(0,\vec{1}) \times \{0\}} \otimes \norm{g}_{B^{\vec{r},w_{\vec{A},\vec{r}}}_{\vec{A}}}}_{E} \lesssim \norm{M^{\vec{A}}_{\vec{r}}(\delta_{0,n}g)_{n}}_{E}.
\end{align*}
Using the boundedness of $M^{\vec{A}}_{\vec{r}}$ on $E$ in combination with \eqref{IR:eq:lemma:HN_L1.1.4;1} we obtain the desired estimate \eqref{IR:eq:lemma:HN_L1.1.4;3}.
\end{proof}

Having introduced the classes of 'admissible' quasi-Banach function spaces in Definition~\ref{IR:def:S} and having discussed some basic properties of these, let us now proceed with the associated function spaces. Let us for introductory purposes first have a look at the classical isotropic Lizorkin-Triebel and Besov spaces.

In the setting of Example~\ref{IR:ex:def:S;classical_isotropic}, we would like to associate to $E=L_{p}(\R^{d})[\ell^{s}_{q}(\N)]$ and $E=\ell^{s}_{q}(\N)[L_{p}(\R^{d})]$ the classical Lizorkin-Triebel space $Y(E) = F^s_{p,q}(\R^d)$ and the classical Besov space $Y(E) = B^s_{p,q}(\R^d)$, respectively. 

A standard way to define the Lizorkin-Triebel and Besov spaces is by means of a smooth resolution of unity/Littlewood-Paley decomposition, as in \cite[Section~2.3.1, Definition~2]{Triebel1983_TFS_I}. However, there are many other ways.
For instance, $F^s_{p,q}(\R^d)$ and $B^s_{p,q}(\R^d)$ could alternatively be defined through the Nikol'skij representations as in \cite[Section~2.5.2]{Triebel1983_TFS_I} (also see the references therein), which may be characterized as a "decomposition of the given distribution by entire analytic functions of exponential type".
This decomposition is a representation as a series of entire analytic functions of exponential type whose spectra lie in dyadic annuli. 
The annuli can be even replaced by balls when $s>d(\frac{1}{r}-1)_{+}$, where $r$ is as in Example~\ref{IR:ex:def:S;classical_isotropic}, see \cite[Section~2.3.2]{Runst&Sickel1996}, \cite[Section~3.6]{JS_traces} or \cite[Proposition~1.1.12]{Hedberg&Netrusov2007}. Moreover, in the latter situation, $F^s_{p,q}(\R^d)$ and $B^s_{p,q}(\R^d)$ consist of regular distributions and the series not only converges in a distributional sense (in $\mathcal{S}'$) but also in a measure theoretic sense~(in $L_{1,\loc}$).  
The characterization through the series representation with the dyadic ball condition and the convergence in a measure theoretic sense, valid under the restriction $s>d(\frac{1}{r}-1)_{+}$, has turned out to be quite useful. Such a description is taken as the definition of the spaces of measurable functions $FL^s_{p,q}(\R^d)$ and $BL^s_{p,q}(\R^d)$ for $s \in (0,\infty)$, so that
$F^s_{p,q}(\R^d) =  FL^s_{p,q}(\R^d)$ and $B^s_{p,q}(\R^d) = BL^s_{p,q}(\R^d)$ when $s>d(\frac{1}{r}-1)_{+}$. As is mentioned in \cite[page 9]{Hedberg&Netrusov2007}, the spaces $FL^s_{p,q}(\R^d)$ and $BL^s_{p,q}(\R^d)$ have been less studied in the range $s \leq d(\frac{1}{r}-1)_{+}$, where they do not coincide with the Lizorkin-Triebel and Besov spaces, but see~\cite{Netrusov1989b,Netrusov1989}.

We will associate to $E=L_{p}(\R^{d})[\ell^{s}_{q}(\N)]$ and $E=\ell^{s}_{q}(\N)[L_{p}(\R^{d})]$ the spaces of distributions $Y(E) = F^s_{p,q}(\R^d)$ and $Y(E) = B^s_{p,q}(\R^d)$, respectively, through the Nikol'skij representation discussed above. We will furthermore associate to these choices of $E$, under the restriction that $s \in (0,\infty)$, the respective spaces of measurable functions $YL(E) = FL^s_{p,q}(\R^d)$ and $YL(E) = BL^s_{p,q}(\R^d)$.

Let us now turn back to the general setting. In Definitions \ref{IR:def:YL} and \ref{IR:def:YL_widetilde} we will define the spaces $YL^{\vec{A}}(E;X)$ and $\widetilde{YL}^{\vec{A}}(E;X)$, respectively, which are both generalizations of $YL(E)$ from \cite[Definition~1.1.15]{Hedberg&Netrusov2007} to our setting.
The difference between $YL^{\vec{A}}(E;X)$ and $\widetilde{YL}^{\vec{A}}(E;X)$ will be a matter of the $X$-valued setting. Whereas $YL^{\vec{A}}(E;X)$ will be defined in a more straightforward way, simply replacing $E$ by $E(X)$ compared to the scalar-valued setting, 
the definition of $\widetilde{YL}^{\vec{A}}(E;X)$ will be more technical, involving testing with functionals $x^{*} \in X^{*}$ in combination with, and in interplay with, some kind of domination. The motivation for the more technical space $\widetilde{YL}^{\vec{A}}(E;X)$ comes from Remark~\ref{IR:rmk:thm:HN_T1.1.14_YL} on estimates involving differences.

In Definition~\ref{IR:def:Y} we will define the space $Y^{\vec{A}}(E;X)$ through a Nikol'skij representation type of approach, which is a generalization of $Y(E)$ from \cite[Definition~1.1.16]{Hedberg&Netrusov2007} to our setting. 
The equivalent Littlewood-Paley description will follow in Proposition~\ref{IR:prop:LP-decomp_characterization}. 
Concrete examples will be given Example~\ref{IR:ex:prop:LP-decomp_characterization}, which includes the classical Lizorkin-Triebel and Besov spaces discusssed above.
Furthermore, in Theorem~\ref{IR:thm:incl_comparY&YL} we will see that, under a suitable restriction, $Y^{\vec{A}}(E;X)$ coincides with $YL^{\vec{A}}(E;X)$ and $\widetilde{YL}^{\vec{A}}(E;X)$.  

\begin{definitie}\label{IR:def:YL}
Suppose that $\varepsilon_{+},\varepsilon_{-}>0$ and let $E \in \mathcal{S}(\varepsilon_{+},\varepsilon_{-},\vec{A},\vec{r},(S,\mathscr{A},\mu))$.
We define $YL^{\vec{A}}(E;X)$ as the space of all $f \in L_{0}(S;L_{\vec{r},\mathpzc{d},\mathrm{loc}}(\R^{d};X))$ which have a representation
\[
f = \sum_{n=0}^{\infty}f_{n} \quad \text{in} \quad L_{0}(S;L_{\vec{r},\mathpzc{d},\mathrm{loc}}(\R^{d};X))
\]
with $(f_{n})_{n} \subset L_{0}(S;\mathcal{S}'(\R^{d};X))$ satisfying the spectrum condition
\[
\supp \hat{f}_{n} \subset \overline{B}^{\vec{A}}(0,2^{n+1}), \qquad n \in \N,
\]
and $(f_{n})_{n} \in E(X)$. We equip $YL^{\vec{A}}(E;X)$  with the quasinorm
\[
\norm{f}_{YL^{\vec{A}}(E;X)} := \inf \norm{(f_{n})}_{E(X)},
\]
where the infimum is taken over all representations as above.
We write $YL^{\vec{A}}(E) := YL^{\vec{A}}(E;\C)$.
\end{definitie}

\begin{definitie}\label{IR:def:YL_widetilde}
Suppose that $\varepsilon_{+},\varepsilon_{-}>0$ and let $E \in \mathcal{S}(\varepsilon_{+},\varepsilon_{-},\vec{A},\vec{r},(S,\mathscr{A},\mu))$.
We define $\widetilde{YL}^{\vec{A}}(E;X)$ as the space of all $f \in L_{0}(S;L_{\vec{r},\mathpzc{d},\mathrm{loc}}(\R^{d};X))$ for which there exists $(g_{n})_{n} \in E_{+}$ such that, for all $x^{*} \in X^{*}$, $\ip{f}{x^{*}}$ has a representation
\[
\ip{f}{x^{*}} = \sum_{n=0}^{\infty}f_{x^{*},n} \quad \text{in} \quad L_{0}(S;L_{\vec{r},\mathpzc{d},\mathrm{loc}}(\R^{d}))
\]
with $(f_{x^{*},n})_{n} \subset L_{0}(S;\mathcal{S}'(\R^{d}))$ satisfying the spectrum condition
\[
\supp \hat{f}_{x^{*},n} \subset \overline{B}^{\vec{A}}(0,2^{n+1}), \qquad n \in \N,
\]
and the domination $|f_{x^{*},n}| \leq \norm{x^{*}}g_{n}$.
We equip $\widetilde{YL}^{\vec{A}}(E;X)$  with the quasinorm
\[
\norm{f}_{\widetilde{YL}^{\vec{A}}(E;X)} := \inf \norm{(g_{n})}_{E},
\]
where the infimum is taken over all $(g_{n})_{n}$ as above.
We write $\widetilde{YL}^{\vec{A}}(E) := \widetilde{YL}^{\vec{A}}(E;\C)$.
\end{definitie}

\begin{remark}\label{IR:rmk:def:YL_widetilde;scalar-valued}
Note that $\widetilde{YL}^{\vec{A}}(E) = YL^{\vec{A}}(E)$.
\end{remark}

\begin{remark}\label{IR:rmk:indep_r_YL}
Suppose that $\varepsilon_{+},\varepsilon_{-}>0$ and let $E \in \mathcal{S}(\varepsilon_{+},\varepsilon_{-},\vec{A},\vec{r},(S,\mathscr{A},\mu))$. 
Then the following statements hold:
\begin{enumerate}[(i)]
\item $YL^{\vec{A}}(E;X) \subset \widetilde{YL}^{\vec{A}}(E;X)$ with equality of norms.
\item\label{IR:it:rmk:indep_r_YL;YL} Let $f \in YL^{\vec{A}}(E;X)$ with $(f_{n})_{n}$ as in Definition~\ref{IR:def:YL} with $\norm{(f_{n})_{n}}_{E(X)} \leq 2\norm{f}_{YL^{\vec{A}}(E;X)}$. Let $\tilde{\vec{r}} \in (0,\infty)^{\ell}$ be such that
    \begin{equation}\label{IR:eq:rmk:indep_r_YL;tilde_r}
    E \in \mathcal{S}(\varepsilon_{+},\varepsilon_{-},\vec{A},\tilde{\vec{r}},(S,\mathscr{A},\mu)).
    \end{equation}
    Then, by Remark~\ref{IR:rmk:lemma:HN_L1.1.4;abs_conv_sum}, as
    \[
    E^{\vec{A}}_{\otimes}(B^{\tilde{\vec{r}},w_{\vec{A},\tilde{\vec{r}}}}_{\vec{A}}(X)) \hookrightarrow L_{0}(S;L_{\tilde{\vec{r}},\mathpzc{d},\mathrm{loc}}(\R^{d};X)) \hookrightarrow L_{0}(S;L_{\tilde{\vec{r} \wedge \vec{r}},\mathpzc{d},\mathrm{loc}}(\R^{d};X)),
    \]
    there is the convergence $f=\sum_{n=0}^{\infty}f_{n}$ in $E^{\vec{A}}_{\otimes}(B^{\tilde{\vec{r}},w_{\vec{A},\tilde{\vec{r}}}}_{\vec{A}}(X))$ with
    \[
    \norm{f}_{E^{\vec{A}}_{\otimes}(B^{\tilde{\vec{r}},w_{\vec{A},\tilde{\vec{r}}}}_{\vec{A}}(X))}
    \lesssim \norm{(f_{n})_{n}}_{E(X)} \leq 2\norm{f}_{YL^{\vec{A}}(E;X)}.
    \]
    In particular, $YL^{\vec{A}}(E;X)$ does not depend on $\vec{r}$ and
    \[
    YL^{\vec{A}}(E;X) \hookrightarrow E^{\vec{A}}_{\otimes}(B^{\vec{r},w_{\vec{A},\vec{r}}}(X)).
    \]
\item Let $f \in \widetilde{YL}^{\vec{A}}(E;X)$ with $(g_{n})_{n} \in E_{+}$ and $\{f_{x^{*},n}\}_{(x^{*},n)}$ as in Definition~\ref{IR:def:YL_widetilde} with $\norm{(g_{n})_{n}}_{E} \leq 2\norm{f}_{\widetilde{YL}^{\vec{A}}(E;X)}$.
    Let $\tilde{\vec{r}} \in (0,\infty)^{\ell}$ satisfy \eqref{IR:eq:rmk:indep_r_YL;tilde_r}.
    Then $\norm{f}_{X} \leq \sum_{n=0}^{\infty}g_{n}$, so that $f \in E^{\vec{A}}_{\otimes}(B^{\tilde{\vec{r}},w_{\vec{A},\tilde{\vec{r}}}}_{\vec{A}}(X)) \subset L_{0}(S;L_{\tilde{\vec{r}},\mathpzc{d},\mathrm{loc}}(\R^{d};X))$ with
    \[
    \norm{f}_{E^{\vec{A}}_{\otimes}(B^{\tilde{\vec{r}},w_{\vec{A},\tilde{\vec{r}}}}_{\vec{A}}(X))} \lesssim \norm{(g_{n})_{n}}_{E} \leq 2\norm{f}_{\widetilde{YL}^{\vec{A}}(E;X)}
    \]
    by Remark~\ref{IR:rmk:lemma:HN_L1.1.4;abs_conv_sum}.
    By \eqref{IR:it:rmk:indep_r_YL;YL} it furthermore holds that
    \[
    \ip{f}{x^{*}} = \sum_{n=0}^{\infty}f_{x^{*},n} \quad \text{in} \quad L_{0}(S;L_{\tilde{\vec{r}},\mathpzc{d},\mathrm{loc}}(\R^{d})).
    \]
    Therefore, $\widetilde{YL}^{\vec{A}}(E;X)$ does not depend on $\vec{r}$ and
    \[
    \widetilde{YL}^{\vec{A}}(E;X) \hookrightarrow E^{\vec{A}}_{\otimes}(B^{\vec{r},w_{\vec{A},\vec{r}}}(X)).
    \]
\end{enumerate}
\end{remark}

\begin{definitie}\label{IR:def:Y}
Let $E \in \mathcal{S}(\varepsilon_{+},\varepsilon_{-},\vec{A},\vec{r},(S,\mathscr{A},\mu))$.
We define $Y^{\vec{A}}(E;X)$ as the space of all $f \in L_{0}(S;\mathcal{S}'(\R^{d};X))$ which have a representation
\[
f = \sum_{n=0}^{\infty}f_{n} \quad \text{in} \quad L_{0}(S;\mathcal{S}'(\R^{d};X))
\]
with $(f_{n})_{n} \subset L_{0}(S;\mathcal{S}'(\R^{d};X))$ satisfying the spectrum condition
\begin{align*}
\supp \hat{f}_{0} & \subset \overline{B}^{\vec{A}}(0,2) \\
\supp \hat{f}_{n} & \subset \overline{B}^{\vec{A}}(0,2^{n+1}) \setminus B^{\vec{A}}(0,2^{n-1}), \qquad n \geq 1,
\end{align*}
and $(f_{n})_{n} \in E(X)$. We equip $Y^{\vec{A}}(E;X)$  with the quasinorm
\[
\norm{f}_{Y^{\vec{A}}(E;X)} := \inf \norm{(f_{n})}_{E(X)},
\]
where the infimum is taken over all representations as above.
\end{definitie}

\begin{example}\label{IR:ex:def:Y}
In the setting of Example~\ref{IR:ex:def:S;classical_isotropic}, 
$$
Y^{\vec{A}}(E) = Y^{I_d}(E) = \begin{cases}
F^s_{p,q}(\R^d), & \text{if}\: E=L_{p}(\R^{d})[\ell^{s}_{q}(\N)], \\
B^s_{p,q}(\R^d), & \text{if}\: E=\ell^{s}_{q}(\N)[L_{p}(\R^{d})],
\end{cases}
$$ 
see for instance \cite[Section~2.5.2]{Triebel1983_TFS_I}.
\end{example}
 
More examples will be given in Example~\ref{IR:ex:prop:LP-decomp_characterization}, after the Littlewood-Paley description given in Proposition~\ref{IR:prop:LP-decomp_characterization}.

\begin{prop}\label{IR:prop:YL_completeness}
Suppose that $\varepsilon_{+},\varepsilon_{-}>0$ and let $E \in \mathcal{S}(\varepsilon_{+},\varepsilon_{-},\vec{A},\vec{r},(S,\mathscr{A},\mu))$.
Then $YL^{\vec{A}}(E;X)$ and $\widetilde{YL}^{\vec{A}}(E;X)$ are quasi-Banach spaces with
\[
YL^{\vec{A}}(E;X) \subset \widetilde{YL}^{\vec{A}}(E;X) \hookrightarrow E^{\vec{A}}_{\otimes}(B^{\vec{r},\vec{w}_{\vec{A},\vec{r}}}_{\vec{A}};X)),
\]
where $YL^{\vec{A}}(E;X)$ is a closed subspace of $\widetilde{YL}^{\vec{A}}(E;X)$.
\end{prop}
\begin{proof}
By Remark~\ref{IR:rmk:indep_r_YL},
\begin{equation}\label{IR:eq:prop:YL_completeness;incl}
YL^{\vec{A}}(E;X), \widetilde{YL}^{\vec{A}}(E;X) \hookrightarrow E^{\vec{A}}_{\otimes}(B^{\vec{r},\vec{w}_{\vec{A},\vec{r}}}_{\vec{A}};X)).
\end{equation}
That $YL^{\vec{A}}(E;X) \subset \widetilde{YL}^{\vec{A}}(E;X)$ with $\norm{f}_{YL^{\vec{A}}(E;X)} = \norm{f}_{\widetilde{YL}^{\vec{A}}(E;X)}$ for all $f \in YL^{\vec{A}}(E;X)$ follows easily from the definitions. So it remains to be shown that $YL^{\vec{A}}(E;X)$ and $\widetilde{YL}^{\vec{A}}(E;X)$ are complete.

Let us first treat $YL^{\vec{A}}(E;X)$. To this end, let the subspace $E(X)_{\vec{A}}$ of $E(X)$ be defined by
\begin{equation*}\label{IR:eq:prop:YL_completeness;E_A}
E(X)_{\vec{A}} := \left\{ (f_{n})_{n} \in E(X) : f_{n} \in L_{0}(S;\mathcal{S}'(\R^{d};X)),\:
\supp \hat{f}_{n} \subset \overline{B}^{\vec{A}}(0,2^{n+1}) \right\}
\end{equation*}
By Lemma~\ref{IR:lemma:HN_L1.1.4},
\[
\Sigma:E(X)_{\vec{A}} \longra E^{\vec{A}}_{\otimes}[L_{\vec{r}}(\R^{d},\vec{w})](X) \hookrightarrow L_{0}(S;L_{\vec{r},\mathpzc{d},\loc}(\R^{d};X)),\, (f_{n})_{n} \mapsto \sum_{n=0}^{\infty}f_{n}
\]
is a well-defined continuous linear mapping.
As
\[
YL^{\vec{A}}(E;X) \simeq \faktor{E(X)_{\vec{A}}}{\ker(\Sigma)} \quad\text{isometrically},
\]
it suffices to show that $E(X)_{\vec{A}}$ is complete.

In order to show that $E(X)_{\vec{A}}$ is complete, we prove that it is a closed subspace of the quasi-Banach space $E(X)$.
Put $w(x):=\prod_{j=1}^{\ell}(1+\rho_{A_{j}}(x_{j}))^{\mathrm{tr}(A_{j})/r_{j}}$. Then it is enough to show that, for each $k \in \N$,
\begin{equation}\label{IR:eq:prop:YL_completeness;proof;EA}
E(X)_{\vec{A}} \longra L_{0}(S;BC(\R^{d},w;X)),\: (f_{n})_{n} \mapsto f_{k},
\end{equation}
continuously, where $BC(\R^{d},w;X)=\{h \in C(\R^{d};X): wh \in L_{\infty}(\R^{d};X)\}$.
Indeed, $BC(\R^{d},w;X) \hookrightarrow \mathcal{S}'(\R^{d};X)$.

In order to establish \eqref{IR:eq:prop:YL_completeness;proof;EA}, let $(f_{n})_{n} \in E(X)_{\vec{A}}$.
By Corollary~\ref{IR:appendix:cor:lemma:Peetre-Feffeman-Stein},
\[
\sup_{z \in B^{\vec{A}}(0,2^{-n})}\norm{f_{n}}_{X} \lesssim M^{\vec{A}}_{\vec{r}}(\norm{f_{n}}_{x})(x),
\]
so that
\begin{align*}
\norm{f_{n}(x)}_{X}
&\lesssim \inf_{z \in B^{\vec{A}}(0,2^{-n})}M^{\vec{A}}_{\vec{r}}(\norm{f_{n}}_{X})(x+z) \\
&\lesssim 2^{n\mathrm{tr}(\vec{A})\bcdot\vec{r}^{-1}}
\normb{M^{\vec{A}}_{\vec{r}}(\norm{f_{n}}_{X})}_{L_{\vec{r},\mathpzc{d}}(B^{\vec{A}}(x,2^{-n}))}.
\end{align*}
For $\vec{R} \in [1,\infty)^{\ell}$ we can thus estimate
\begin{align}
\sup_{z \in B^{\vec{A}}(0,\vec{R})} \norm{f_{n}(x)}_{X}
&\lesssim 2^{n\mathrm{tr}(\vec{A})\bcdot\vec{r}^{-1}}
\normb{M^{\vec{A}}_{\vec{r}}(\norm{f_{n}}_{X})}
_{L_{\vec{r},\mathpzc{d}}(B^{\vec{A}}(0,\vec{c}_{\vec{A}}[\vec{R}+2^{-n}\vec{1}]))} \nonumber \\
&\lesssim 2^{n\mathrm{tr}(\vec{A})\bcdot\vec{r}^{-1}}
\normb{M^{\vec{A}}_{\vec{r}}(\norm{f_{n}}_{X})}
_{L_{\vec{r},\mathpzc{d}}(B^{\vec{A}}(0,2\vec{c}_{\vec{A}}\vec{R}))} \nonumber \\
&\lesssim 2^{n\mathrm{tr}(\vec{A})\bcdot\vec{r}^{-1}}
\inf_{z \in B^{\vec{A}}(0,\vec{R})}\normb{M^{\vec{A}}_{\vec{r}}(\norm{f_{n}}_{X})}
_{L_{\vec{r},\mathpzc{d}}(B^{\vec{A}}(0,2\vec{c}_{\vec{A}}(\vec{c}_{\vec{A}}+\vec{1})\vec{R}))} \nonumber \\
&\lesssim 2^{n\mathrm{tr}(\vec{A})\bcdot\vec{r}^{-1}}\vec{R}^{\mathrm{tr}(\vec{A})\vec{r}^{-1}}
\inf_{z \in B^{\vec{A}}(0,\vec{R})}M^{\vec{A}}_{\vec{r}}(M^{\vec{A}}_{\vec{r}}(\norm{f_{n}}_{X}))(z).
\label{IR:eq:prop:YL_completeness;proof;HN(1.1.17)}
\end{align}
The latter implies that
\[
1_{B^{\vec{A}}(0,\vec{R})} \otimes \norm{f_{n}}_{L_{\infty}(B^{\vec{A}}(0,\vec{R});X)}
\lesssim 2^{n\mathrm{tr}(\vec{A})\bcdot\vec{r}^{-1}}\vec{R}^{\mathrm{tr}(\vec{A})\vec{r}^{-1}} M^{\vec{A}}_{\vec{r}}(M^{\vec{A}}_{\vec{r}}(\norm{f_{n}}_{X}))
\]
for $\vec{R} \in [1,\infty)^{\ell}$. It thus follows that
\begin{align*}
\norm{f_{n}}_{E^{\vec{A}}_{\otimes}(L_{\infty}(B^{\vec{A}}(0,\vec{R});X))}
&\leq \normB{1_{B^{\vec{A}}(0,\vec{R}) \times \{0\}} \otimes \norm{f_{n}}_{L_{\infty}(B^{\vec{A}}(0,\vec{R});X)}}_{E} \\
&\lesssim 2^{n\mathrm{tr}(\vec{A})\bcdot\vec{r}^{-1}}\vec{R}^{\mathrm{tr}(\vec{A})\vec{r}^{-1}}
\norm{(\delta_{0,k}M^{\vec{A}}_{\vec{r}}(\norm{f_{n}}_{X}))_{k}}_{E} \\
&\lesssim 2^{n(\mathrm{tr}(\vec{A})\bcdot\vec{r}^{-1}-\varepsilon_{+})}\vec{R}^{\mathrm{tr}(\vec{A})\vec{r}^{-1}}
\norm{(h_{k})_{k}}_{E(X)}.
\end{align*}

Let us finally prove that $\widetilde{YL}^{\vec{A}}(E;X)$ is complete.
To this end, let $\kappa \in (0,1]$ with $\kappa \leq \vec{r}_{\min}$ be such that $\norm{\,\cdot\,}_{E}$ is equivalent to a $\kappa$-norm. Then $\norm{\,\cdot\,}_{\widetilde{YL}^{\vec{A}}(E;X)}$ and $\norm{\,\cdot\,}_{E^{\vec{A}}_{\otimes}[L_{\vec{r}}(\R^{d},\vec{w})](X)}$ are equivalent to $\kappa$-norms as well.
It suffices to show that, if $(f^{(k)})_{k \in \N} \subset \widetilde{YL}^{\vec{A}}(E;X)$ satisfies $\sum_{k=0}^{\infty}\norm{f^{(k)}}_{\widetilde{YL}^{\vec{A}}(E;X)}^{\kappa} < \infty$, then $\sum_{k=0}^{\infty}f^{(k)}$ is a convergent series in $\widetilde{YL}^{\vec{A}}(E;X)$.
So fix such a $(f^{(k)})_{k \in \N}$.
As a consequence of \eqref{IR:eq:prop:YL_completeness;incl},
\[
\sum_{k=0}^{\infty}\norm{f^{(k)}}_{E^{\vec{A}}_{\otimes}[L_{\vec{r}}(\R^{d},\vec{w})]}^{\kappa}
\lesssim \sum_{k=0}^{\infty}\norm{f^{(k)}}_{\widetilde{YL}^{\vec{A}}(E;X)}^{\kappa} < \infty.
\]
As $E^{\vec{A}}_{\otimes}[L_{\vec{r}}(\R^{d},\vec{w})]$ is a quasi-Banach space with a $\kappa$-norm, $\sum_{k=0}^{\infty}f^{(k)}$ converges to some $F$ in $E^{\vec{A}}_{\otimes}[L_{\vec{r}}(\R^{d},\vec{w})]$.
To finish the proof, we show that $F \in \widetilde{YL}^{\vec{A}}(E;X)$ with convergence $F=\sum_{k=0}^{\infty}f^{(k)}$ in $\widetilde{YL}^{\vec{A}}(E;X)$.

For each $k \in \N$ there exists $(g^{(k)}_{n})_{n} \in E_{+}$ with $\norm{(g^{(k)}_{n})_{n}}_{E} \leq 2\norm{f^{(k)}}_{\widetilde{YL}^{\vec{A}}(E;X)}$ such that, for every $x^{*} \in X^{*}$, $\ip{f^{(k)}}{x^{*}}$ has the representation
\[
\ip{f^{(k)}}{x^{*}} = \sum_{n=0}^{\infty}f^{(k)}_{x^{*},n} \quad \text{in} \quad L_{0}(S;L_{\vec{r},\mathpzc{d},\mathrm{loc}}(\R^{d}))
\]
for some $(f^{(k)}_{x^{*},n})_{n} \in E_{\vec{A}}$ with $|f^{(k)}_{x^{*},n}| \leq \norm{x^{*}}g^{(k)}_{n}$.
By Remark~\ref{IR:rmk:indep_r_YL},
\begin{align*}
\sum_{k=0}^{\infty}\sum_{n=0}^{\infty}
\norm{f^{(k)}_{x^{*},n}}_{E^{\vec{A}}_{\otimes}[L_{\vec{r}}(\R^{d},\vec{w})]}^{\kappa}
\lesssim \sum_{k=0}^{\infty}\norm{f^{(k)}}_{\widetilde{YL}^{\vec{A}}(E;X)}^{\kappa} < \infty.
\end{align*}
As $E^{\vec{A}}_{\otimes}[L_{\vec{r}}(\R^{d},\vec{w})] \hookrightarrow L_{0}(S;L_{\vec{r},\mathpzc{d},\mathrm{loc}}(\R^{d})) \hookrightarrow L_{0}(S \times \R^{d})$ is a quasi-Banach space with a $\kappa$-norm, we thus find that
$F=\sum_{n=0}^{\infty}F_{x^{*},n}$ in $L_{0}(S;L_{\vec{r},\mathpzc{d},\mathrm{loc}}(\R^{d}))$ with $F_{x^{*},n} := \sum_{k=0}^{\infty}f^{(k)}_{x^{*},n}$ in $L_{0}(\R^{d} \times S)$ satisfying $|F_{x^{*},n}| \leq \sum_{k=0}^{\infty}|f^{(k)}_{x^{*},n}| \leq \norm{x^{*}}\sum_{k=0}^{\infty}g^{(k)}_{n}$.
As $E_{\vec{A}}$ is a closed subspace of the quasi-Banach function space $E$ on $\R^{d} \times \N \times S$ with $\kappa$-norm, it follows from
\[
\sum_{k=0}^{\infty}\norm{(f^{(k)}_{x^{*},n})_{n}}_{E}^{\kappa} \leq \norm{x^{*}}^{\kappa} \sum_{k=0}^{\infty}\norm{f^{(k)}}_{\widetilde{YL}^{\vec{A}}(E;X)}^{\kappa} < \infty
\]
that $(F_{x^{*},n})_{n} = \sum_{k=0}^{\infty}f^{(k)}_{x^{*},n}$ in $E$ and thus that $(F_{x^{*},n})_{n} \in E_{\vec{A}}$. Moreover, $G_{n}:=\sum_{k=0}^{\infty}g^{(k)}_{n}$ defines $(G_{n})_{n} \in E_{+}$ with
\[
\norm{(G_{n})_{n}}_{E}^{\kappa} \leq \sum_{k=0}^{\infty}\norm{(g^{k}_{n})_{n}}_{E}^{\kappa} \leq 2\sum_{k=0}^{\infty}\norm{f^{(k)}}_{\widetilde{YL}^{\vec{A}}(E;X)}^{\kappa}
\]
and $|F_{x^{*},n}| \leq \norm{x^{*}}G_{n}$.
This shows that $F \in \widetilde{YL}^{\vec{A}}(E;X)$ with convergence $F=\sum_{k=0}^{\infty}f^{(k)}$ in $\widetilde{YL}^{\vec{A}}(E;X)$.
\end{proof}

The content of the following proposition is a Littlewood-Paley characterization for $Y^{\vec{A}}(E;X)$.
Before we state it, we first need to introduce the set $\Phi^{\vec{A}}(\R^{d})$ of all $\vec{A}$-anisotropic Littlewood-Paley sequences $\varphi=(\varphi_{n})_{n \in \N}$.

\begin{definition}\label{IR:def:LP-sequences}
For $0<\gamma < \delta < \infty$ we define $\Phi^{\vec{A}}_{\gamma,\delta}(\R^{d})$ as the set of all sequences $\varphi=(\varphi_{n})_{n \in \N} \subset \mathcal{S}(\R^{d})$ that can be constructed in the following way: given $\varphi_{0} \in \mathcal{S}(\R^{d})$ satisfying
\[
0 \leq \hat{\varphi}_{0} \leq 1,\:\: \hat{\varphi}_{0}(\xi) = 1 \:\:\text{if}\:\:\rho_{\vec{A}}(\xi) \leq \gamma,\:\: \hat{\varphi}_{0}(\xi)=0 \:\:\text{if}\:\:\rho_{\vec{A}}(\xi) \geq \delta,
\]
$(\varphi_{n})_{n \geq 1} \subset \mathcal{S}(\R^{d})$ is obtained through
\[
\hat{\varphi}_{n} = \hat{\varphi}_{1}(\vec{A}_{2^{-n+1}}\,\cdot\,) =
\hat{\varphi}_{0}(\vec{A}_{2^{-n}}\,\cdot\,) - \hat{\varphi}_{0}(\vec{A}_{2^{-n+1}}\,\cdot\,), \qquad n \geq 1.
\]

We define $\Phi^{\vec{A}}(\R^{d}) := \bigcup_{0<\gamma < \delta < \infty} \Phi^{\vec{A}}_{\gamma,\delta}(\R^{d})$.
\end{definition}

Let $\varphi = (\varphi_{n})_{n \in \N} \in \Phi^{\vec{A}}_{\gamma,\delta}(\R^{d})$. Then $\sum_{n=0}^{\infty}\hat{\varphi}_{n}=1$ in $\mathscr{O}_{M}(\R^{d})$ with
\[
\supp \hat{\varphi}_{0} \subset \{ \xi : \rho_{\vec{A}}(\xi) \leq \gamma \},\qquad
\supp \hat{\varphi}_{n} \subset \{ \xi : 2^{n-1}\gamma\leq \rho_{\vec{A}}(\xi) \leq 2^{n}\delta \},
\quad n \geq 1,
\]
To $\varphi$ we associate the family of convolution operators $(S_{n})_{n \in \N} = (S^{\varphi}_{n})_{n \in \N} \subset \mathcal{L}(\mathcal{S}'(\R^{d};X),\check{\mathcal{E}}'(\R^{d};X))$ given by
\[
S_{n}f = S_{n}^{\varphi}f := \varphi_{n}*f = \mathcal{F}^{-1}[\hat{\varphi}_{n}\hat{f}].
\]

\begin{prop}\label{IR:prop:LP-decomp_characterization}
Let $E \in \mathcal{S}(\varepsilon_{+},\varepsilon_{-},\vec{A},\vec{r},(S,\mathscr{A},\mu))$ and $\varphi=(\varphi_{n})_{n \in \N} \in \Phi^{\vec{A}}(\R^{d})$ with associated sequence of convolution operators $(S_{n})_{n \in \N}$. Then
\[
Y^{\vec{A}}(E;X) = \{ f \in L_{0}(S;\mathcal{S}'(\R^{d};X)) : (S_{n}f)_{n} \in E(X) \}
\]
with
\[
\norm{f}_{Y^{\vec{A}}(E;X)} \eqsim \norm{(S_{n}f)_{n}}_{E(X)}.
\]
\end{prop}

Before we go the proof of Proposition~\ref{IR:prop:LP-decomp_characterization}, let us first consider the following.
\begin{example}\label{IR:ex:prop:LP-decomp_characterization}
In the following three points we let the notation be as in Example~\ref{IR:ex:def:S}.\eqref{IR:it:ex:def:S;2}, Example~\ref{IR:ex:def:S}.\eqref{IR:it:ex:def:S;3} and Example~\ref{IR:ex:def:S}.\eqref{IR:it:ex:def:S;1}, respectively. We define:
\begin{enumerate}[(i)]
\item\label{IR:it:ex:prop:LP-decomp_characterization;2} $F^{s,\vec{A}}_{\vec{p},q}(\R^{d},\vec{w};X) := Y^{\vec{A}}(E;X)$ for $E=L_{\vec{p}}(\R^{d},\vec{w})[\ell^{s}_{q}(\N)]$;
\item\label{IR:it:ex:prop:LP-decomp_characterization;3} $B^{s,\vec{A}}_{\vec{p},q}(\R^{d},\vec{w};X) := Y^{\vec{A}}(E;X)$ for $E=\ell^{s}_{q}(\N)[L_{\vec{p}}(\R^{d},\vec{w})]$;
\item\label{IR:it:ex:prop:LP-decomp_characterization;1} $\F^{s,\vec{A}}_{\vec{p},q}(\R^{d},\vec{w};F;X) := Y^{\vec{A}}(E;X)$ for $E=L_{\vec{p}}(\R^{d},\vec{w})[F[\ell^{s}_{q}(\N)]]$.
\end{enumerate}
Restricting to special cases we find, in view of Proposition~\ref{IR:prop:LP-decomp_characterization}, $B$- and $F$-spaces that have been studied in the literature:
\begin{itemize}
\item[\eqref{IR:it:ex:prop:LP-decomp_characterization;2}$\&$\eqref{IR:it:ex:prop:LP-decomp_characterization;3}:]
\begin{enumerate}[(a)]
\item In case $\ell=1$, $w=1$ and $X=\C$, $F^{s,\vec{A}}_{\vec{p},q}(\R^{d},\vec{w};X)$ and $B^{s,\vec{A}}_{\vec{p},q}(\R^{d},\vec{w};X)$ reduce to the anisotropic Besov and Lizorkin-Triebel spaces considered in e.g.\ \cite{Dappa&Trebels1989,Dintelmann1995}. The latter are special cases of the anisotropic spaces from the more general \cite{Barrios&Betancor2011,Bownik2003,Bownik&Ho2006} by taking $2^{A}$ as the expansive dilation in the approach there.
\item In case $\ell=d$, $\vec{A}=\mathrm{diag}(\vec{a})$ with $\vec{a} \in (0,\infty)$, $\vec{w}=\vec{1}$ and $X=\C$, $F^{s,\vec{A}}_{\vec{p},q}(\R^{d},\vec{w};X)$ and $B^{s,\vec{A}}_{\vec{p},q}(\R^{d},\vec{w};X)$ reduce to the anisotropic mixed-norm Besov and Lizorkin-Triebel spaces considered in e.g.\ \cite{Johnsen&Mucn_Hansen&Sickel2015,JS_traces}.
\item In case $\vec{A}=(a_{1}I_{\mathpzc{d}_{1}},\ldots,a_{\ell}I_{\mathpzc{d}_{\ell}})$ with $\vec{a} \in (0,\infty)$, $F^{s,\vec{A}}_{\vec{p},q}(\R^{d},\vec{w};X)$ and $B^{s,\vec{A}}_{\vec{p},q}(\R^{d},\vec{w};X)$ reduce to the anisotropic weighted mixed-norm Besov and Lizorkin-Triebel spaces considered in \cite{Lindemulder_master-thesis,lindemulder2017maximal}.

\item In case $\ell=1$ and $A=I$, $F^{s,\vec{A}}_{\vec{p},q}(\R^{d},\vec{w};X)$ and $B^{s,\vec{A}}_{\vec{p},q}(\R^{d},\vec{w};X)$ reduce to the weighted Besov and Lizorkin-Triebel spaces considered in e.g. \cite{Bui&Paluszynski&Taibleson1996,Bui1982,Bui1994,Haroske&Piotrowska2008,
    Haroske&Skrzypczak2008_EntropyI,Haroske&Skrzypczak2011_EntropyII,
    Haroske&Skrzypczak2011_EntropyIII,Lindemulder2018_DSOP,Sickel&Skrzypczak&Vybiral2014} ($X=\C$) and \cite{Meyries&Veraar2012_sharp_embedding,Meyries&Veraar2014_traces,
    Meyries&Veraar2015_pointwise_multiplication} ($X$ a general Banach space). In the case $w=1$ these further reduces to the classical Besov and Lizorkin-Triebel spaces (see e.g.\ \cite{Scharf&Schmeisser&Sickel_Traces_vector-valued_Sobolev,Triebel1983_TFS_I,Triebel1992_TFS_II}).
\end{enumerate}
\item[\eqref{IR:it:ex:prop:LP-decomp_characterization;1}:]
\begin{enumerate}[(a)]
\item In case $\ell=1$, $A=I$, $p \in (1,\infty)$, $q \in [1,\infty]$, $w=1$, $F$ is a UMD Banach function space and $X=\C$, $\F^{s,\vec{A}}_{\vec{p},q}(\R^{d},\vec{w};F;X)$ reduces to a special case of the generalized Lizorkin-Triebel spaces considered in~\cite{Kunstmann&Ullmann2014}.
\item In case $\ell=1$, $A=I$, $p \in (1,\infty)$, $q = 2$, $w \in A_{p}(\R^{d})$, $F$ is a UMD Banach function space and $X$ is a Hilbert space, $\F^{s,\vec{A}}_{\vec{p},q}(\R^{d},\vec{w};F;X)$ coincides with the weighted Bessel potential space $H^{s}_{p}(\R^{d},w;F(X))$ (which can be seen as a special case of \cite[Proposition~3.2]{Meyries&Veraar2015_pointwise_multiplication} through the use of the Khintchine-Maurey theorem (see e.g.\ \cite[Theorem~7.2.13]{Hytonen&Neerven&Veraar&Weis2016_Analyis_in_Banach_Spaces_II})).
\end{enumerate}
\end{itemize}
\end{example}

The proof of Proposition~\ref{IR:prop:LP-decomp_characterization} basically only consists of proving the estimate in the following lemma. We have extracted it as a lemma as it is interesting on its own. A consequence of the lemma for instance is that the spectrum condition in Definition~\ref{IR:def:Y} could be slightly modified.
\begin{lemma}\label{IR:lemma:prop:LP-decomp_characterization}
Let $E \in \mathcal{S}(\varepsilon_{+},\varepsilon_{-},\vec{A},\vec{r},(S,\mathscr{A},\mu))$, $c \in (1,\infty)$ and $\varphi=(\varphi_{n})_{n \in \N} \in \Phi^{\vec{A}}(\R^{d})$ with associated sequence of convolution operators $(S_{n})_{n \in \N}$.
For all $f \in L_{0}(S;\mathcal{S}'(\R^{d};X))$ which have a representation
\[
f = \sum_{n=0}^{\infty}f_{n} \quad \text{in} \quad L_{0}(S;\mathcal{S}'(\R^{d};X))
\]
with $(f_{n})_{n} \subset L_{0}(S;\mathcal{S}'(\R^{d};X))$ satisfying the spectrum condition
\begin{align*}
\supp \hat{f}_{0} & \subset \overline{B}^{\vec{A}}(0,c) \\
\supp \hat{f}_{n} & \subset \overline{B}^{\vec{A}}(0,c2^{n}) \setminus B^{\vec{A}}(0,c^{-1}2^{n}), \qquad n \geq 1,
\end{align*}
there is the estimate
\[
\norm{(S_{n}f)_{n}}_{E(X)} \lesssim \norm{(f_{n})_{n}}_{E(X)}.
\]
\end{lemma}
\begin{proof}
This can be established as in \cite[Lemma~5.2.10]{Lindemulder_master-thesis} (also see \cite[Section~2.3.2]{Triebel1983_TFS_I} and \cite[Section~15.5]{Triebel2011_Fractals}), using
a combination of Corollary~\ref{IR:appendix:cor:lemma:Peetre-Feffeman-Stein} and Lemma~\ref{IR:appendix:lemma:master_thesis_Prop.3.4.8_pointwise_ineq}.
\end{proof}

\begin{proof}[Proof of Proposition~\ref{IR:prop:LP-decomp_characterization}]
Let $f \in Y^{\vec{A}}(E;X)$. Take $(f_{n})_{n}$ as in Definition~\ref{IR:def:Y} with $\norm{(f_{n})_{n}}_{E(X)} \leq 2\norm{f}_{Y^{\vec{A}}(E;X)}$.
Lemma~\ref{IR:lemma:prop:LP-decomp_characterization} (with $c=2$) then gives
\[
\norm{(S_{n}f)_{n}}_{E(X)} \lesssim \norm{(f_{n})_{n}}_{E(X)} \leq 2\norm{f}_{Y^{\vec{A}}(E;X)}.
\]

For the reverse direction, let $f \in L_{0}(S;\mathcal{S}'(\R^{d};X))$ be such that $(S_{n}f)_{n} \in E(X)$. Pick $\psi = (\psi_{n})_{n \in \N} \in \Phi^{\vec{A}}(\R^{d})$ such that
\[
\supp \hat{\psi}_{0} \subset \overline{B}^{\vec{A}}(0,2),\qquad
\supp \hat{\psi}_{n} \subset \overline{B}^{\vec{A}}(0,2^{n+1}) \setminus B^{\vec{A}}(0,2^{n-1}), \quad n \geq 1,
\]
and let $(T_{n})_{n \in \N}$ denote the associated sequence of convolution operators.
Then
\begin{equation}\label{IR:eq:prop:LP-decomp_characterization;proof}
\supp \widehat{T_{0}f} \subset \overline{B}^{\vec{A}}(0,2),\qquad
\supp \widehat{T_{n}f} \subset \overline{B}^{\vec{A}}(0,2^{n}) \setminus B^{\vec{A}}(0,2^{n-1}), \quad n \geq 1,
\end{equation}
Picking $c \in (1,\infty)$ such that
\[
\supp \hat{\varphi}_{0} \subset \overline{B}^{\vec{A}}(0,c),\qquad
\supp \hat{\varphi}_{n} \subset \overline{B}^{\vec{A}}(0,c2^{n}) \setminus B^{\vec{A}}(0,c^{-1}2^{n}), \quad n \geq 1,
\]
we furthermore have
\[
\supp \widehat{S_{n}f} \subset \overline{B}^{\vec{A}}(0,c),\qquad
\supp \widehat{S_{n}f} \subset \overline{B}^{\vec{A}}(0,c2^{n}) \setminus B^{\vec{A}}(0,c^{-1}2^{n}), \quad n \geq 1.
\]
As $f=\sum_{n=0}^{\infty}S_{n}f$ in $L_{0}(S;\mathcal{S}'(\R^{d};X))$, Lemma~\ref{IR:lemma:prop:LP-decomp_characterization} gives
\[
\norm{(T_{n}f)_{n}}_{E(X)} \lesssim \norm{(S_{n}f)_{n}}_{E(X)}.
\]
Since $f=\sum_{n=0}^{\infty}S_{n}f$ in $L_{0}(S;\mathcal{S}'(\R^{d};X))$ with \eqref{IR:eq:prop:LP-decomp_characterization;proof}, it follows that $f \in Y^{\vec{A}}(E;X)$ with
\[
\norm{f}_{Y^{\vec{A}}(E;X)} \leq \norm{(T_{n}f)_{n}}_{E(X)} \lesssim \norm{(S_{n}f)_{n}}_{E(X)}. \qedhere
\]
\end{proof}

%

\begin{thm}\label{IR:thm:incl_comparY&YL}
Let $E \in \mathcal{S}(\varepsilon_{+},\varepsilon_{-},\vec{A},\vec{r},(S,\mathscr{A},\mu))$.
Suppose that $\varepsilon_{+} > \mathrm{tr}(\vec{A})\bcdot(\vec{r}^{-1}-\vec{1})_{+}$, where $\mathrm{tr}(\vec{A})$ is the component-wise trace of $\vec{A}$ given by $\mathrm{tr}(\vec{A}):=(\mathrm{tr}(A_1),\ldots,\mathrm{tr}(A_\ell))$.
Then
\begin{equation}\label{IR:thm:incl_comparY&YL;1}
\widetilde{YL}^{\vec{A}}(E;X) \hookrightarrow E^{\vec{A}}_{\otimes}(B^{1,w_{\vec{A},\vec{r} \wedge \vec{1}}}_{\vec{A}}(X)) \hookrightarrow L_{0}(S;L_{\vec{1} \wedge \vec{r},\mathpzc{d},\loc}(\R^{d};X))
\end{equation}
and
\begin{align}
Y^{\vec{A}}(E;X) &\hookrightarrow E^{\vec{A}}_{\otimes}(B^{1,w_{\vec{A},\vec{r} \wedge \vec{1}}}_{\vec{A}}(X)) \hookrightarrow \mathcal{S}'(\R^{d};E^{\vec{A}}_{\otimes}(X)) \nonumber \\
& \qquad\qquad \hookrightarrow \mathcal{S}'(\R^{d};L_{0}(S;X)) = L_{0}(S;\mathcal{S}'(\R^{d};X)) \label{IR:thm:incl_comparY&YL;2}
\end{align}
and there is the identity
\begin{equation}\label{IR:thm:incl_comparY&YL;3}
Y^{\vec{A}}(E;X) = YL^{\vec{A}}(E;X) = \widetilde{YL}^{\vec{A}}(E;X).
\end{equation}
\end{thm}

\begin{remark}\label{IR:rmk:thm:incl_comparY&YL}
Note that the condition $\varepsilon_{+} > \mathrm{tr}(\vec{A})\bcdot(\vec{r}^{-1}-\vec{1})_{+}$ is for instance fulfilled when~$\vec{r} \geq \vec{1}$.
\end{remark}

We will use the following lemma in the proof of Theorem~\ref{IR:thm:incl_comparY&YL}.

\begin{lemma}\label{IR:lemma:thm:incl_comparY&YL;series_ball_cond}
Let the notations and assumptions be as in Theorem~\ref{IR:thm:incl_comparY&YL}.
Let $c \in (0,\infty)$.
If 
\[
(f_{n})_{n} \in E(X)_{\vec{A},c} := \left\{ (h_{n})_{n} \in E(X) : h_{n} \in L_{0}(S;\mathcal{S}'(\R^{d};X)),\:
\supp \hat{h}_{n} \subset \overline{B}^{\vec{A}}(0,c2^{n+1}) \right\},
\]
then $\sum_{n \in \N}f_{n}$ is a convergent series in $L_{0}(S;B^{1,w_{\vec{A},\vec{r} \wedge \vec{1}}}_{\vec{A}}(X))$ with
\begin{equation*}\label{IR:thm:incl_comparY&YL;4}
\normb{\sum_{n=0}^{\infty}f_{n}}_{E^{\vec{A}}_{\otimes}(B^{1,w_{\vec{A},\vec{r} \wedge \vec{1}}}_{\vec{A}}(X))} \leq
\normB{\sum_{n=0}^{\infty}\norm{f_{n}}_{X}}_{E^{\vec{A}}_{\otimes}(B^{1,w_{\vec{A},\vec{r} \wedge \vec{1}}}_{\vec{A}})} \lesssim \norm{(f_{n})_{n}}_{E(X)}.
\end{equation*}
\end{lemma}
\begin{proof}
It suffices to prove the second estimate.
We may without loss of generality assume that $\vec{r} \in (0,1]^{\ell}$.
Choose $\kappa > 0$ such that $E^{\vec{A}}_{\otimes}$ has a $\kappa$-norm.
For simplicity of notation we only present the case $\ell=2$ and $c=1$, the general case being the same.

Let $(f_{n})_{n} \in E(X)_{\vec{A}}$. Let $\vec{R} \in [1,\infty)^{2}$.
As a consequence of the Paley-Wiener-Schwartz theorem,
\[
\check{\mathcal{E}}'_{\overline{B}^{\vec{A}}(0,2^{n})}(\R^{d};X) \hookrightarrow C^{\infty}(\R^{\mathpzc{d}_{2}};\check{\mathcal{E}}'_{\overline{B}^{A_{1}}(0,2^{n})}(\R^{\mathpzc{d}_{1}};X))
\cap C^{\infty}(\R^{\mathpzc{d}_{1}};\check{\mathcal{E}}'_{\overline{B}^{A_{2}}(0,2^{n})}(\R^{\mathpzc{d}_{2}};X)).
\]
In particular, as in \eqref{IR:eq:prop:YL_completeness;proof;HN(1.1.17)} we find that
\begin{equation}\label{IR:eq:lemma:thm:incl_comparY&YL;series_ball_cond;proof;1}
\norm{f_{n}(x_{1},z_{2})}_{X} \lesssim (2^{n}R_{1})^{\mathrm{tr}(A_{1})/r_{1}}M^{A_{1}}_{r_{1};[\mathpzc{d};1]}(M^{A_{1}}_{r_{1},[\mathpzc{d};1]}
(\norm{f_{n}}_{X}))(y_{1},z_{2})
\end{equation}
for all $x_{1},y_{1} \in B^{A_{1}}(0,R_{1})$ and $z_{1} \in \R^{\mathpzc{d}_{1}}$,
and
\begin{equation}\label{IR:eq:lemma:thm:incl_comparY&YL;series_ball_cond;proof;2}
\norm{f_{n}(z_{1},x_{2})}_{X} \lesssim (2^{n}R_{2})^{\mathrm{tr}(A_{2})/r_{2}}M^{A_{2}}_{r_{2};[\mathpzc{d};2]}(M^{A_{2}}_{r_{2},[\mathpzc{d};2]}
(\norm{f_{n}}_{X}))(z_{1},y_{2})
\end{equation}
for all $x_{2},y_{2} \in B^{A_{2}}(0,R_{2})$ and $z_{2} \in \R^{\mathpzc{d}_{2}}$.

Then, for $z \in B^{\vec{A}}(0,\vec{R})$,
\begin{align*}
&\int_{B^{\vec{A}}(0,\vec{R})}\norm{f_{n}(x)}_{X}\,\mathrm{d}x \\
&\qquad= \int_{B^{A_{2}}(0,R_{2})}\int_{B^{A_{1}}(0,R_{1})}\norm{f_{n}(x_{1},x_{2})}_{X}
\,\mathrm{d}x_{1}\mathrm{d}x_{2} \\
&\qquad\stackrel{\eqref{IR:eq:lemma:thm:incl_comparY&YL;series_ball_cond;proof;1}}{\lesssim}
((2^{n}R_{1})^{\mathrm{tr}(A_{1})/r_{1}})^{1-r_{1}} \int_{B^{A_{2}}(0,R_{2})}M^{A_{1}}_{r_{1};[\mathpzc{d};1]}(M^{A_{1}}_{r_{1};[\mathpzc{d};1]}(\norm{f_{n}(\,\cdot\,,x_{2})}_{X}))(z_{1})^{r_{1}-1} \\
&\qquad\qquad \:\cdot\:
\int_{B^{A_{1}}(0,R_{1})}\norm{f_{n}(x_{1},x_{2})}_{X}^{r_{1}}
\,\mathrm{d}x_{1}\mathrm{d}x_{2} \\
&\qquad\lesssim 2^{n\mathrm{tr}(A_{1})(1-r_{1})/r_{1}}R_{1}^{\mathrm{tr}(A_{1})/r_{1}}
\int_{B^{A_{2}}(0,R_{2})}M^{A_{1}}_{r_{1};[\mathpzc{d};1]}(M^{A_{1}}_{r_{1};[\mathpzc{d};1]}(\norm{f_{n}(\,\cdot\,,x_{2})}_{X}))(z_{1})\mathrm{d}x_{2} \\
&\qquad\stackrel{\eqref{IR:eq:lemma:thm:incl_comparY&YL;series_ball_cond;proof;2}}{\lesssim}
2^{n\big(\mathrm{tr}(A_{1})(1-r_{1})/r_{1}+\mathrm{tr}(A_{2})(1-r_{2})/r_{2}\big)}
\vec{R}^{\mathrm{tr}(\vec{A})\vec{r}^{-1}} \\
&\qquad\qquad \:\cdot\: M^{A_{1}}_{r_{1};[\mathpzc{d};1]}M^{A_{1}}_{r_{1};[\mathpzc{d};1]}M^{A_{2}}_{r_{2};[\mathpzc{d};2]}
M^{A_{2}}_{r_{2};[\mathpzc{d};2]}(\norm{f_{n}}_{X}))(z_{1},z_{2})^{1-r_{2}} \\
&\qquad\qquad \:\cdot\: M^{A_{2}}_{r_{2};[\mathpzc{d};1]}
M^{A_{2}}_{r_{2};[\mathpzc{d};2]}M^{A_{1}}_{r_{1};[\mathpzc{d};2]}
M^{A_{1}}_{r_{1};[\mathpzc{d};1]}(\norm{f_{n}}_{X}))(z_{1},z_{2})^{r_{2}} \\
&\qquad \leq 2^{n\mathrm(\vec{A})\bcdot(\vec{r}^{-1}-\vec{1})}
\vec{R}^{\mathrm{tr}(\vec{A})\vec{r}^{-1}} [M^{\vec{A}}_{\vec{r}}]^{4}(\norm{f_{n}}_{X})(z).
\end{align*}
This implies that
\[
1_{B^{\vec{A}}(0,\vec{R})} \otimes \int_{B^{\vec{A}}(0,\vec{R})}\sum_{n=0}^{\infty}\norm{f_{n}(x)}_{X}\,\mathrm{d}x
\lesssim \vec{R}^{\mathrm{tr}(\vec{A})\vec{r}^{-1}}
\sum_{n=0}^{\infty}2^{n\mathrm(\vec{A})\bcdot(\vec{r}^{-1}-\vec{1})}
[M^{\vec{A}}_{\vec{r}}]^{4}(\norm{f_{n}}_{X}).
\]
Since $\varepsilon_{+} > \mathrm{tr}(\vec{A})\bcdot(\vec{r}^{-1}-\vec{1})_{+}$, it follows that
\begin{align*}
\normB{\sum_{n=0}^{\infty}\norm{f_{n}}_{X}}_{E^{\vec{A}}_{\otimes}(B^{1,w_{\vec{A},\vec{r} \wedge \vec{1}}}_{\vec{A}})]}
&\stackrel{\eqref{IR:eq:lemma:HN_L1.1.4;1}}{\lesssim} \norm{([M^{\vec{A}}_{\vec{r}}]^{4}(\norm{f_{n}}_{X}))_{n}}_{E} \\
&\lesssim \norm{(f_{n})}_{E(X)}. \qedhere
\end{align*}
\end{proof}

\begin{proof}[Proof of Theorem~\ref{IR:thm:incl_comparY&YL}]
We may without loss of generality assume that $\vec{r} \in (0,1]^{\ell}$.

As $L_{0}(S;B^{1,w_{\vec{A},\vec{r} \wedge \vec{1}}}_{\vec{A}}(X)) \hookrightarrow L_{0}(S;\mathcal{S}'(\R^{d};X))$, the first inclusion in \eqref{IR:thm:incl_comparY&YL;2} follows from Lemma~\ref{IR:lemma:thm:incl_comparY&YL;series_ball_cond}.
So in \eqref{IR:thm:incl_comparY&YL;2} it remains to prove the second inclusion.
To this end, let us first note that
\[
\mathcal{S}(\R^{d}) \hookrightarrow \mathcal{B}(B^{1,w_{\vec{A},\vec{r} \wedge \vec{1}}}_{\vec{A}}(X),X),\, \phi \mapsto \ip{\,\cdot\,}{\phi}.
\]
This induces
\[
\mathcal{S}(\R^{d}) \hookrightarrow \mathcal{B}(E^{\vec{A}}_{\otimes}(B^{1,w_{\vec{A},\vec{r} \wedge \vec{1}}}_{\vec{A}}(X)),E^{\vec{A}}_{\otimes}(X)),\, \phi \mapsto \ip{\,\cdot\,}{\phi}.
\]
Therefore, $f \mapsto [\phi \mapsto \ip{f}{\phi}]$ is a continuous linear operator from $E^{\vec{A}}_{\otimes}(B^{1,w_{\vec{A},\vec{r} \wedge \vec{1}}}_{\vec{A}}(X))$ to $\mathcal{L}(\mathcal{S}(\R^{d});E^{\vec{A}}_{\otimes}(X))$, which is a reformulation of the required inclusion.

As $L_{0}(S;B^{1,w_{\vec{A},\vec{r} \wedge \vec{1}}}_{\vec{A}}) \hookrightarrow L_{0}(S;L_{\vec{r},\mathpzc{d},\loc}(\R^{d}))$, the inclusion
\begin{equation*}
Y^{\vec{A}}(E) \hookrightarrow E^{\vec{A}}_{\otimes}(B^{1,w_{\vec{A},\vec{r} \wedge \vec{1}}}_{\vec{A}})
\end{equation*}
follows from Lemma~\ref{IR:lemma:thm:incl_comparY&YL;series_ball_cond}.
We thus get a continuous bilinear mapping
\[
\widetilde{YL}^{\vec{A}}(E,X) \times X^{*} \longra YL^{\vec{A}}(E) \hookrightarrow L_{0}(S;\mathcal{S}'(\R^{d})),\,(f,x^{*}) \mapsto \ip{f}{x^{*}}.
\]
and a continuous linear mapping
\begin{equation}\label{IR:thm:incl_comparY&YL;proof:T_f}
\widetilde{YL}^{\vec{A}}(E,X) \longra L_{0}(S;\mathcal{S}'(\R^{d};X^{**})),\,f \mapsto T_{f},
\end{equation}
defined by
\[
\ip{x^{*}}{T_{f}(\phi)} := \ip{f}{x^{*}}(\phi), \qquad \phi \in \mathcal{S}(\R^{d}), x^{*} \in X^{*}.
\]

Let us now show that $f \mapsto T_{f}$ \eqref{IR:thm:incl_comparY&YL;proof:T_f} restricts to a bounded linear mapping
\begin{equation}\label{IR:thm:incl_comparY&YL;proof:T_f;bdd_Y}
\widetilde{YL}^{\vec{A}}(E,X) \longra Y^{\vec{A}}(E;X^{**}),\, f \mapsto T_{f}.
\end{equation}
To this end, let $f \in \widetilde{YL}^{\vec{A}}(E;X)$ and put $F:=T_{f}$. Let $(g_{n})_{n}$ and $(f_{x^{*},n})_{(x^{*},n)}$ be as in Definition~\ref{IR:def:YL_widetilde} with $\norm{(g_{n})_{n}}_{E} \leq 2\norm{f}_{\widetilde{YL}^{\vec{A}}(E;X)}$. It will convenient to put $g_{n}:=0$ and $f_{x^{*},n}:=0$ for $n \in \Z_{<0}$.
By Lemma~\ref{IR:lemma:thm:incl_comparY&YL;series_ball_cond}, as $(f_{x^{*},n})_{n} \in E_{\vec{A}}$ and $B^{1,w_{\vec{A},\vec{r} \wedge \vec{1}}}_{\vec{A}} \hookrightarrow \mathcal{S}'(\R^{d})$,
\[
\ip{f}{x^{*}} = \sum_{k=0}^{\infty}f_{x^{*},k} \quad \text{in} \quad L_{0}(S;B^{1,w_{\vec{A},\vec{r} \wedge \vec{1}}}_{\vec{A}}) \hookrightarrow L_{0}(S;\mathcal{S}'(\R^{d})), \qquad x^{*} \in X^{*}.
\]
Now let $(S_{n})_{n \in \N}$ be as in Proposition~\ref{IR:prop:LP-decomp_characterization}.
There exists $h \in \N$ independent of $f$ such that $S_{n}f_{x^{*},k}=0$ for all $x^{*} \in X^{*}$, $n \in \N$ and $k \in \Z_{<n-h}$.
Let $x^{*} \in X^{*}$. Then
\begin{align*}
\ip{x^{*}}{S_{n}F} &= S_{n}\ip{x^{*}}{F} = S_{n}\ip{f}{x^{*}}= S_{n}\sum_{k=0}^{\infty}f_{x^{*},k} = \sum_{k=0}^{\infty}S_{n}f_{x^{*},k} \\
&= \sum_{k=n-h}^{\infty}S_{n}f_{x^{*},k} = \sum_{k=0}^{\infty}S_{n}f_{x^{*},k+n-h}
\end{align*}
with convergence in $L_{0}(S;\mathcal{S}'(\R^{d}))$.
Together with Corollary~\ref{IR:appendix:cor:prop:Marschall}, this implies the pointwise estimates
\begin{align*}
|\ip{x^{*}}{S_{n}F}| &\leq \sum_{k=0}^{\infty}|S_{n}f_{x^{*},k+n-h}|
\lesssim \sum_{k=0}^{\infty}2^{(k-h)_{+}\mathrm{tr}(\vec{A})\bcdot(\vec{r}^{-1}-\vec{1})}
M^{\vec{A}}_{\vec{r}}(f_{k+n-h,x^{*}}) \\
&\leq \norm{x^{*}}\sum_{k=0}^{\infty}2^{(k-h)_{+}\mathrm{tr}(\vec{A})\bcdot(\vec{r}^{-1}-\vec{1})}
M^{\vec{A}}_{\vec{r}}(g_{k+n-h}).
\end{align*}
Taking the supremum over $x^{*} \in X^{*}$ with $\norm{x^{*}} \leq 1$, we obtain
\[
\norm{S_{n}F}_{X^{**}} \leq \sum_{k=0}^{\infty}2^{(k-h)_{+}\mathrm{tr}(\vec{A})\bcdot(\vec{r}^{-1}-\vec{1})}
M^{\vec{A}}_{\vec{r}}(g_{k+n-h}).
\]
Picking $\kappa > 0$ such that $E$ has a $\kappa$-norm, we find that
\begin{align*}
\norm{(S_{n}F)_{n}}_{E(X^{**})}^{\kappa} &= \normb{(\norm{S_{n}f}_{X^{**}})_{n}}_{E}^{\kappa} \\
&\lesssim \sum_{k=0}^{\infty}2^{\kappa(k-h)_{+}\mathrm{tr}(\vec{A})\bcdot(\vec{r}^{-1}-\vec{1})}
\normb{M^{\vec{A}}_{\vec{r}}(g_{k+n-h})_{n}}_{E}^{\kappa} \\
\end{align*}
Since
\begin{align*}
\normb{M^{\vec{A}}_{\vec{r}}(g_{k+n-h})_{n}}_{E}
&= \normb{(g_{k+n-h})_{n}}_{E} \lesssim
\left\{\begin{array}{ll}
\normb{(S_{-})^{h-k}(g_{n})_{n}}_{E}, & k \leq h,\\
\normb{(S_{+})^{k-h}(g_{k+n-h})_{n}}_{E}, & k \geq h,
\end{array}\right. \\
&\lesssim \left(2^{\varepsilon_{-}(h-k)_{+}} + 2^{-\varepsilon_{+}(k-h)_{+}}\right)\norm{(g_{n})_{n}}_{E} \\
&\lesssim 2^{-\varepsilon_{+}(k-h)_{+}}\norm{f}_{\widetilde{YL}^{\vec{A}}(E;X)}
\end{align*}
for all $k \in \N$, it follows that
\begin{align*}
\norm{(S_{n}F)_{n}}_{E(X^{**})}^{\kappa}
&\lesssim \sum_{k=0}^{\infty}2^{\kappa(k-h)_{+}\left(\mathrm{tr}(\vec{A})\bcdot(\vec{r}^{-1}-\vec{1})-
\varepsilon_{+}\right)}\norm{f}_{\widetilde{YL}^{\vec{A}}(E;X)}^{\kappa}.
\end{align*}
As $\varepsilon_{+} > \mathrm{tr}(\vec{A})\bcdot(\vec{r}^{-1}-\vec{1})$, we find that $\norm{(S_{n}F)_{n}}_{E(X^{**})} \lesssim \norm{f}_{\widetilde{YL}^{\vec{A}}(E;X)}$ and thus that $F \in Y^{\vec{A}}(E;X^{**})$ with $\norm{F}_{Y^{\vec{A}}(E;X^{**})} \lesssim \norm{f}_{\widetilde{YL}^{\vec{A}}(E;X)}$ (see Proposition~\ref{IR:prop:LP-decomp_characterization}).
So we obtain the desired \eqref{IR:thm:incl_comparY&YL;proof:T_f;bdd_Y}.

Next we prove that
\begin{equation}\label{IR:thm:incl_comparY&YL;proof:T_f;incl_YL-widetilde_Y}
\widetilde{YL}^{\vec{A}}(E;X) \hookrightarrow Y^{\vec{A}}(E;X).
\end{equation}
So let $f \in \widetilde{YL}^{\vec{A}}(E;X)$. A combination of \eqref{IR:thm:incl_comparY&YL;proof:T_f;bdd_Y} and \eqref{IR:thm:incl_comparY&YL;2} gives that $F:=T_{f} \in L_{0}(S;X^{**}))$. Since $f \in L_{0}(S;L_{\vec{r},\mathpzc{d},\loc}(\R^{d};X))$ with $\ip{x^{*}}{F} = \ip{f}{x^{*}}$ for every $x^{*} \in X^{*}$, it follows that
\[
f = F \in L_{0}(S;B^{1,w_{\vec{A},\vec{r} \wedge \vec{1}}}_{\vec{A}}(X^{**})) \cap L_{0}(S;L_{\vec{r},\mathpzc{d},\loc}(\R^{d};X)) \subset
L_{0}(S;B^{1,w_{\vec{A},\vec{r} \wedge \vec{1}}}_{\vec{A}}(X)).
\]
Therefore, by boundedness of \eqref{IR:thm:incl_comparY&YL;proof:T_f;bdd_Y},
\[
\widetilde{YL}^{\vec{A}}(E;X) \hookrightarrow \big\{g \in Y^{\vec{A}}(E;X^{**}) : g \in L_{0}(S;\mathcal{S}'(\R^{d};X))\big\} = Y^{\vec{A}}(E;X). \qedhere
\]

\end{proof}

For a quasi-Banach function space $E$ on $\R^{d} \times \N \times S$ and a number $\sigma \in \R$ we define the quasi-Banach function space $E^{\sigma}$ on $\R^{d} \times \N \times S$ by
\[
\norm{(f_{n})_{n}}_{E^{\sigma}} := \norm{(2^{n\sigma}f_{n})_{n}}_{E}, \qquad (f_{n})_{n} \in L_{0}(\R^{d} \times \N \times S).
\]
Note that $E^{\sigma} \in \mathcal{S}(\varepsilon_{+}+\sigma,\varepsilon_{-}+\sigma,\vec{A},\vec{r},(S,\mathscr{A},\mu))$ when $E \in \mathcal{S}(\varepsilon_{+},\varepsilon_{-},\vec{A},\vec{r},(S,\mathscr{A},\mu))$.

\begin{prop}\label{IR:prop:lifting}
Let $E \in \mathcal{S}(\varepsilon_{+},\varepsilon_{-},\vec{A},\vec{r},(S,\mathscr{A},\mu))$ and $\sigma \in \R$. Let $\psi \in \mathscr{O}_{\mathrm{M}}(\R^{d})$ be such that $\psi(\xi)=\rho_{\vec{A}}(\xi)$ for $\rho_{\vec{A}}(\xi) \geq 1$ and $\psi(\xi) \neq 0$ for $\rho_{\vec{A}}(\xi) \leq 1$.
Then $\phi(D) \in \mathcal{L}(L_{0}(S;\mathcal{S}'(\R^{d};X)))$ restricts to an isomorphism
\[
\phi(D): Y^{\vec{A}}(E^{\sigma};X) \stackrel{\simeq}{\longra} Y^{\vec{A}}(E;X).
\]
\end{prop}
\begin{proof}
Using Proposition~\ref{IR:prop:LP-decomp_characterization} and Lemma~\ref{IR:appendix:lemma:master_thesis_Prop.3.4.8_pointwise_ineq},
this can be proved as \cite[Lemma~5.2.28]{Lindemulder_master-thesis} (also see \cite[Theorem~2.3.8]{Triebel1983_TFS_I}).
\end{proof}

\begin{prop}\label{IR:prop:incl_S'}
Let $E \in \mathcal{S}(\varepsilon_{+},\varepsilon_{-},\vec{A},\vec{r},(S,\mathscr{A},\mu))$.
Then
\[
Y^{\vec{A}}(E;X) \hookrightarrow \mathcal{S}'(\R^{d};E^{\vec{A}}_{\otimes}(X)) \hookrightarrow \mathcal{S}'(\R^{d};L_{0}(S;X)) = L_{0}(S;\mathcal{S}'(\R^{d};X))
\]
and $Y^{\vec{A}}(E;X)$, when equipped with an equivalent quasi-norm from Proposition~\ref{IR:prop:LP-decomp_characterization}, has the Fatou property with respect to $L_{0}(S;\mathcal{S}'(\R^{d};X))$. As a consequence (see Lemma \ref{IR:lemma:qB_via_Fatou}), $Y^{\vec{A}}(E;X)$ is a quasi-Banach space.
\end{prop}
\begin{proof}
The chain of inclusions follow from a combination of Theorem~\ref{IR:thm:incl_comparY&YL} and Proposition~\ref{IR:prop:lifting}.

In order to establish the Fatou property, suppose that $Y^{\vec{A}}(E;X)$ has been equipped with an equivalent quasi-norm from Proposition~\ref{IR:prop:LP-decomp_characterization}.
Let $f_{k} \to f$ in $L_{0}(S;\mathcal{S}'(\R^{d};X))$ with $\liminf_{k \to \infty}\norm{f_{k}}_{Y^{\vec{A}}(E;X)} < \infty$.
Then
\[
S_{n}f = \lim_{k \to \infty}S_{n}f_{k} \quad \text{in} \quad L_{0}(S;\mathscr{O}_{M}(\R^{d};X)) \hookrightarrow L_{0}(S;L_{1,\mathrm{loc}}(\R^{d};X)) \hookrightarrow L_{0}(\R^{d} \times S;X),
\]
so that
\[
(S_{n}f)_{n \in \N} = \lim_{k \to \infty}(S_{n}f_{k})_{n \in \N} \quad \text{in} \quad L_{0}(\R^{d} \times S;X).
\]
By passing to a suitable subsequence we may without loss of generality assume that $(S_{n}f_{k})_{n \in \N} \to (S_{n}f)_{n \in \N}$ pointwise a.e.\ as $k \to \infty$.
Using the Fatou property of $E$, we find
\begin{align*}
\norm{f}_{Y^{\vec{A}}(E;X)} & = \normb{(\norm{S_{n}f}_{X})_{n}}_{E} = \normb{\liminf_{k \to \infty} (\norm{S_{n}f_{k}}_{X})_{n}}_{E} \\
 &\leq \liminf_{k \to \infty} \normb{(\norm{S_{n}f_{k}}_{X})_{n}}_{E} = \liminf_{k \to \infty}\norm{f_{k}}_{Y^{\vec{A}}(E;X)}. \qedhere
\end{align*}
\end{proof}

\begin{proposition}\label{IR:prop:scaling anisotropy}
Let $F \in \mathcal{S}(0,0,\vec{A},\vec{r},(S,\mathscr{A},\mu))$, $s \in \R$ and $\lambda \in (0,\infty)$. Suppose that there exists a constant $C \in [1,\infty)$ such that, $\norm{(f_{j(n)})_{n \in \N}}_F \leq C\norm{(f_{n})_{n \in \N}}_F$ for all $\{f_n\}_{n \in \N \cup \{*\}} \subset F$ with $f_{*}=0$ and mappings $j:\N \to \N \cup \{*\}$ with the property that $\#j^{-1}(k) \leq 1$ for every $k \in \N$.
Then
$$
Y^{\vec{A}}(F^s;X) = Y^{\lambda\vec{A}}(F^{\lambda s};X)
$$
with an equivalence of quasi-norms.
\end{proposition}

The following lemma constitutes the main step in the proof of Proposition~\ref{IR:prop:scaling anisotropy}.

\begin{lemma}\label{IR:lem:prop:scaling anisotropy}
Let $E \in \mathcal{S}(\varepsilon_{+},\varepsilon_{-},0,\vec{A},\vec{r},(S,\mathscr{A},\mu))$, $s \in \R$ and $\lambda \in (0,\infty)$. 
Set $h := \lfloor\frac{1}{\lambda}\rfloor + 2$.
For all $f \in L_{0}(S;\mathcal{S}'(\R^{d};X))$ of the form
\[
f = \sum_{n=0}^{\infty}f_{n} \quad \text{in} \quad L_{0}(S;\mathcal{S}'(\R^{d};X))
\]
with $(f_{n})_{n \in \Z} \subset L_{0}(S;\mathcal{S}'(\R^{d};X))$ satisfying the spectrum condition
\begin{equation}\label{IR:eq:lem:prop:scaling anisotropy}
\begin{cases}
\supp \hat{f}_{0} & \subset \overline{B}^{\lambda\vec{A}}(0,2), \\
\supp \hat{f}_{n} & \subset \overline{B}^{\lambda\vec{A}}(0,2^{n+1}) \setminus B^{\lambda\vec{A}}(0,2^{n-1}), \qquad n \geq 1,
\end{cases}
\end{equation}
and $f_n = 0$ for $n \in \Z_{<0}$, there is the estimate
$$
\norm{f}_{Y^{\vec{A}}(E^s;X)} \lesssim \sum_{m=-h}^{h}\normb{(2^{\lambda s\lfloor\frac{n}{\lambda}\rfloor} f_{m+\lfloor\frac{n}{\lambda}\rfloor})_n}_{E(X)}. 
$$
\end{lemma}
\begin{proof}
Let $\varphi=(\varphi_{n})_{n \in \N} \in \Phi^{\vec{A}}_{1,2}(\R^{d})$ with associated sequence of convolution operators $(S_{n})_{n \in \N}$.

In view of the spectrum conditions of $(\varphi_{n})_{n \in \N}$ and $(f_n)_{n \in \N}$ and the fact that $\rho_{\lambda\vec{A}} = \rho_{\vec{A}}^{\lambda}$, it holds true that $S_n f_k = 0$ for every $n \in \N$ and $k \in \Z$ satisfying $|k-\lfloor\frac{n}{\lambda}\rfloor| \leq \lfloor\frac{1}{\lambda}\rfloor + 2$.
Since
$$
S_n f =  S_n\left(\sum_{k=0}^{\infty}f_{k}\right) = \sum_{k=0}^{\infty}S_n f_{k} \quad \text{in} \quad L_{0}(S;\mathcal{S}'(\R^{d};X)),
$$
it follows that
\begin{equation*}
S_n f =  \sum_{m=-h}^{h}S_n f_{m+\lfloor\frac{n}{\lambda}\rfloor}, \qquad n \in \N.
\end{equation*}
As
$$
\supp(\hat{f}_{m+\lfloor\frac{n}{\lambda}\rfloor}) 
\subset \overline{B}^{\lambda\vec{A}}(0,2^{m+\lfloor\frac{n}{\lambda}\rfloor}) 
\subset \overline{B}^{\lambda\vec{A}}(0,2^{h+\frac{n}{\lambda}}) = \overline{B}^{\vec{A}}(0,2^{\lambda h + n})
$$
for all $n \in \N$ and $m \in \{-h,\ldots,h\}$,
a combination of  Proposition~\ref{IR:prop:LP-decomp_characterization}, Corollary~\ref{IR:appendix:cor:lemma:Peetre-Feffeman-Stein} and Lemma~\ref{IR:appendix:lemma:master_thesis_Prop.3.4.8_pointwise_ineq} thus yields that
\begin{equation*}
\norm{f}_{Y^{\vec{A}}(E^s;X)} \lesssim \norm{(S_n f)_n}_{E^s(X)} \lesssim \sum_{m=-h}^{h}\norm{(f_{m+\lfloor\frac{n}{\lambda}\rfloor})_{n}}_{E^s(X)}.
\end{equation*}
The desired estimate finally follows from the observation that $2^{ns} \eqsim 2^{\lambda\lfloor\frac{n}{\lambda}\rfloor}$ for all $n \in \N$.
\end{proof}

\begin{proof}[Proof of Proposition~\ref{IR:prop:scaling anisotropy}]
It suffices to show that $Y^{\lambda\vec{A}}(F^{\lambda s};X) \hookrightarrow Y^{\vec{A}}(F^s;X)$, the reverse inclusion also being of this form (for suitable choices of parameters). Let $f \in Y^{\lambda\vec{A}}(F^{\lambda s};X)$. Then $f$ has a representation 
as a convergence series 
$$
f = \sum_{n=0}^{\infty}f_{n} \quad \text{in} \quad L_{0}(S;\mathcal{S}'(\R^{d};X))
$$
with $(f_{n})_{n \in \N} \subset L_{0}(S;\mathcal{S}'(\R^{d};X))$ satisfying the spectrum condition \eqref{IR:eq:lem:prop:scaling anisotropy} and $\norm{(f_n)_n}_{F^{\lambda s}(X)} \leq 2\norm{f}_{Y^{\lambda\vec{A}}(F^{\lambda s};X)}$. Set $f_n := 0$ for $n \in \Z_{<0}$. The assumptions on $F$ and the observation that $\#\{n:\lfloor\frac{n}{\lambda}\rfloor=k\} \leq \lfloor \lambda \rfloor + 1$ for all $k \in \N$, give us the estimates
$$
\normb{(2^{\lambda s\lfloor\frac{n}{\lambda}\rfloor} f_{m+\lfloor\frac{n}{\lambda}\rfloor})_n}_{F(X)} \leq C(\lfloor \lambda \rfloor + 1)\norm{(2^{\lambda s k}f_k)_k}_{F(X)}, \qquad m \in \Z.
$$
An application of Lemma~\ref{IR:lem:prop:scaling anisotropy} finishes the proof.
\end{proof}

\begin{example}\label{IR:ex:prop:scaling anisotropy}
In the setting of Example~\ref{IR:ex:prop:LP-decomp_characterization}, Proposition~\ref{IR:prop:scaling anisotropy} yields:
\begin{enumerate}[(i)]
    \item $F^{ s,\vec{A}}_{\vec{p},q}(\R^{d},\vec{w};X) = F^{\lambda s,\lambda\vec{A}}_{\vec{p},q}(\R^{d},\vec{w};X)$,
    \item $B^{s,\vec{A}}_{\vec{p},q}(\R^{d},\vec{w};X) = B^{\lambda s,\lambda\vec{A}}_{\vec{p},q}(\R^{d},\vec{w};X)$,
 \item $\F^{s,\vec{A}}_{\vec{p},q}(\R^{d},\vec{w};F;X) = \F^{\lambda s,\lambda\vec{A}}_{\vec{p},q}(\R^{d},\vec{w};F;X)$, 
\end{enumerate}
with an equivalence of quasi-norms depending on $\lambda \in (0,\infty)$. In particular, in the special case that $\vec{A}=a\vec{I}_{\mathpzc{d}}=a(I_{\mathpzc{d}_1},\ldots,I_{\mathpzc{d}_\ell})$ for some $a \in (0,\infty)$, taking $\lambda=1/a$ yields a description as an isotropic space.
\end{example}

\section{Difference Norms}\label{IR:sec:Diff_norms}

In this section we derive several estimates for $YL^{\vec{A}}(E;X)$ and $\widetilde{YL}^{\vec{A}}(E;X)$, as well as for $Y^{\vec{A}}(E;X)$.
The main interest lies in the estimates involving differences, as these form the basis for the intersection representation in Section~\ref{IR:sec:IR}.

\subsection{Some notation}

Let $X$ be a Banach space. For each $M \in \N_{1}$ and $h \in \R^{d}$ we define difference operator $\Delta^{M}_{h}$ on $L_{0}(\R^{d};X)$ by
$$
\Delta^{M}_{h} := (L_{h}-I)^{M} = \sum_{i=0}^{M}(-1)^{i}{M \choose i}L_{(M-i)h},
$$
where $L_{h}$ denotes the left translation by $h$.

For $N \in \N$ we denote by $\mathcal{P}^{d}_{N}$ the space of polynomials of degree at most $N$ on $\R^{d}$. We write $\mathcal{P}^{d}_{N}(\Q) \subset \mathcal{P}^{d}_{N}$ for the subset of polynomials having rational coefficients.

Let $M \in \N_{1}$. Let $F = L_{\vec{p},\mathpzc{d}} = L_{\vec{p},\mathpzc{d}}(\R^{d})$ with $\vec{p} \in (0,\infty)^{\ell}$. Let $B \subset \R^{d}$ be a bounded Borel set of non-zero measure.
For $f \in L_{0}(\R^{d})$ we define
\[
\mathcal{E}_{M}(f,B,F) := \inf_{\pi \in \mathcal{P}^{d}_{M-1}}\norm{(f-\pi)1_{B}}_{F} = \inf_{\pi \in \mathcal{P}^{d}_{M-1}(\Q)}\norm{(f-\pi)1_{B}}_{F}
\]
and
\[
\overline{\mathcal{E}}_{M}(f,B,F) := \frac{\mathcal{E}_{M}(f,B,F)}{\mathcal{E}_{M}(1,B,F)}.
\]

We define the collection of dyadic anisotropic cubes $\{Q^{\vec{A}}_{n,k}\}_{(n,k) \in \Z \times \Z^{d}}$ by
\[
Q^{\vec{A}}_{n,k} := \vec{A}_{2^{-n}}\left( [0,1)^{d} + k \right).
\]
For $b \in (0,\infty)$ we define $\{Q^{\vec{A}}_{n,k}(b)\}_{(n,k) \in \Z \times \Z^{d}}$ by
\[
Q^{\vec{A}}_{n,k}(b) := \vec{A}_{2^{-n}}\left( [0,1)^{d}(b) + k \right),
\]
where $[0,1)^{d}(b)$ is the cube concentric to $[0,1)^{d}$ with sidelength $b$:
\[
[0,1)^{d}(b):= \left[\frac{1-b}{2},\frac{1+b}{2}\right)^{d}.
\]
We furthermore define the corresponding families of indicator functions $\{\chi^{\vec{A}}_{n,k}\}_{(n,k) \in \Z \times \Z^{d}}$ and $\{\chi^{\vec{A},b}_{n,k}\}_{(n,k) \in \Z \times \Z^{d}}$:
\[
\chi^{\vec{A}}_{n,k} := 1_{Q^{\vec{A}}_{n,k}} \qquad \text{and} \qquad
\chi^{\vec{A},b}_{n,k} := 1_{Q^{\vec{A}}_{n,k}(b)}.
\]

\begin{definition}
Let $E \in \mathcal{S}(\varepsilon_{+},\varepsilon_{-},\vec{A},\vec{r},(S,\mathscr{A},\mu))$.
We define $y^{\vec{A}}(E)$ as the space of all $(s_{n,k})_{(n,k)\in \N \times \Z^{d}} \subset L_{0}(S)$ for which $(\sum_{k \in \Z^{d}}s_{n,k}\chi^{\vec{A}}_{n,k})_{n \in \N} \in E$.
We equip $y^{\vec{A}}(E)$ with the quasi-norm
\[
\norm{(s_{n,k})_{(n,k)}}_{y^{\vec{A}}(E)} :=
\normB{\big(\sum_{k \in \Z^{d}}s_{n,k}\chi^{\vec{A}}_{n,k}\big)_{n}}_{E}.
\]
\end{definition}

\begin{definition}
Let $F$ be a quasi-Banach function space on the $\sigma$-finite measure space $(T,\mathscr{B},\nu)$.
We define $\mathscr{F}_{\mathrm{M}}(X^{*};F)$ as the space of all $\{F_{x^{*}}\}_{x^{*} \in X^{*}} \subset L_{0}(T)$ for which there exists $G \in F_{+}$ such that $|F_{x^{*}}| \leq \norm{x^{*}}G$ for all $x^{*} \in X^{*}$. We equip $\mathscr{F}_{\mathrm{M}}(X^{*};F)$ with the quasi-norm
\[
\norm{\{F_{x^{*}}\}_{x^{*}}}_{\mathscr{F}_{\mathrm{M}}(X^{*};F)} := \inf\norm{G}_{F},
\]
where the infimum is taken over all majorants $G$ as above.
\end{definition}

In the special case that $F=E \in \mathcal{S}(\varepsilon_{+},\varepsilon_{-},\vec{A},\vec{r},(S,\mathscr{A},\mu))$ in the above definition,
it will be convenient to view $\mathscr{F}_{\mathrm{M}}(X^{*};E)$ as the space of all $\{g_{x^{*},n}\}_{(x^{*},n) \in X^{*} \times \N} \subset L_{0}(S)$ for which there exists $(g_{n})_{n} \in E_{+}$ such that $|g_{x^{*},n}| \leq \norm{x^{*}}g_{n}$, equipped with the quasi-norm
\[
\norm{\{g_{x^{*},n}\}_{(x^{*},n)}}_{\mathscr{F}_{\mathrm{M}}(X^{*};E)} := \inf\norm{(g_{n})_{n}}_{E},
\]
where the infimum is taken over all majorants $(g_{n})_{n}$ as above.

Note that the corresponding properties from Definition~\ref{IR:def:S} for $\mathscr{F}_{\mathrm{M}}(X^{*};E)$ are inherited from $E$.

\begin{definition}
Let $E \in \mathcal{S}(\varepsilon_{+},\varepsilon_{-},\vec{A},\vec{r},(S,\mathscr{A},\mu))$.
We define $\widetilde{y}^{\vec{A}}(E;X)$ as the space of all $(s_{x^{*},n,k})_{(x^{*},n,k)\in X^{*} \times \N \times \Z^{d}} \subset L_{0}(S)$ for which $(\sum_{k \in \Z^{d}}s_{x^{*},n,k}\chi^{\vec{A}}_{n,k})_{n \in \N} \in \mathscr{F}_{\mathrm{M}}(X^{*};E)$.
We equip $\widetilde{y}^{\vec{A}}(E;X)$ with the quasi-norm
\[
\norm{(s_{x^{*},n,k})_{(n,k)}}_{\widetilde{y}^{\vec{A}}(E;X)} :=
\normB{\big(\sum_{k \in \Z^{d}}s_{x^{*},n,k}\chi^{\vec{A}}_{n,k}\big)_{n}}_{\mathscr{F}_{\mathrm{M}}(X^{*};E)}.
\]
\end{definition}

\subsection{Statements of the results}\label{IR:subsec:sec:Diff_norms;results}

The first two main results of this section, Theorems \ref{IR:thm:HN_T1.1.14_YL} and \ref{IR:thm:HN_T1.1.14_YL_widetilde}, contain estimates for $YL^{\vec{A}}(E;X)$ and $\widetilde{YL}^{\vec{A}}(E;X)$, respectively, involving differences, as well as atoms and oscillations, in the general case $\vec{r} \in (0,\infty)^{\ell}$. The third main result of this section, Theorem~\ref{IR:thm:difference_norm_'Ap-case'}, provides estimates for $Y^{\vec{A}}(E;X)=YL^{\vec{A}}(E;X)=\widetilde{YL}^{\vec{A}}(E;X)$ involving differences in the special case that $\vec{r}=1$ (in which case, indeed, $Y^{\vec{A}}(E;X)=YL^{\vec{A}}(E;X)=\widetilde{YL}^{\vec{A}}(E;X)$ by Theorem~\ref{IR:thm:incl_comparY&YL} (and Remark~\ref{IR:rmk:thm:incl_comparY&YL})).
In line with Remark~\ref{IR:rmk:ex:prop:LP-decomp_characterization}, some things simplify here when $\vec{r}=1$.

\begin{theorem}\label{IR:thm:HN_T1.1.14_YL}
Let $E \in \mathcal{S}(\varepsilon_{+},\varepsilon_{-},\vec{A},\vec{r},(S,\mathscr{A},\mu))$ and suppose that $\varepsilon_{+},\varepsilon_{-}>0$.
Let $\vec{p} \in (0,\infty)^{\ell}$ and $M \in \N$ satisfy $\varepsilon_{+}>\mathrm{tr}(\vec{A})\bcdot(\vec{r}^{-1}-\vec{p}^{-1})$ and $M\lambda^{\vec{A}}_{\min} > \varepsilon_{-}$, where  $\mathrm{tr}(\vec{A})=(\mathrm{tr}(A_1),\ldots,\mathrm{tr}(A_\ell))$.
Given $f \in L_{0}(S;L_{\vec{r},\mathpzc{d}}(\R^{d};X))$, consider the following statements:
\begin{enumerate}[(i)]
\item\label{IR:it:thm:HN_T1.1.14_YL;(i)} $f \in YL^{\vec{A}}(E;X)$.
\item \label{IR:it:thm:HN_T1.1.14_YL;(iv)} There exist $(s_{n,k})_{(n,k)} \in y^{\vec{A}}(E)$ and $(b_{n,k})_{(n,k) \in \N \times \Z^{d}} \subset L_{0}(S;C^{M}_{c}([-1,2]^{d}))$ with $\norm{b_{n,k}}_{C^{M}_{b}} \leq 1$ such that, setting $a_{n,k}:=b_{n,k}(\vec{A}_{2^{n}}\,\cdot\,-k)$, $f$ has the representation
    \begin{equation}\label{IR:eq:thm:HN_T1.1.14_YL;(iv);convergence}
    f = \sum_{(n,k) \in \N \times \Z^{d}}s_{n,k}a_{n,k} \qquad \text{in} \qquad
    L_{0}(S;L_{\vec{p},\mathpzc{d},\loc}(\R^{d};X)).
    \end{equation}
\item\label{IR:it:thm:HN_T1.1.14_YL;(ii)} $f \in E_{0}(X) \cap L_{0}(S;L_{\vec{p},\mathpzc{d},\loc}(\R^{d};X))$ and $(d^{\vec{A},\vec{p}}_{M}(f)_{n})_{n \geq 1} \in E(\N_{1})$, where
\[
d^{\vec{A},\vec{p}}_{M,n}(f) := 2^{n\mathrm{tr}(\vec{A})\bcdot\vec{p}^{-1}}\normb{z \mapsto \Delta_{z}^{M}f}_{L_{\vec{p},\mathpzc{d}}(B^{\vec{A}}(0,2^{-n});X)}, \qquad n \in \N.
\]
%
\end{enumerate}
For these statements, there is the chain of implications  \eqref{IR:it:thm:HN_T1.1.14_YL;(i)} $\lra$ \eqref{IR:it:thm:HN_T1.1.14_YL;(iv)} $\ra$ \eqref{IR:it:thm:HN_T1.1.14_YL;(ii)}. Moreover, there are the following estimates:
\begin{align*}
\norm{f}_{E_{0}(X)} + \norm{(d^{\vec{A},\vec{p}}_{M,n}(f))_{n \geq 1}}_{E(\N_{1})} \lesssim \norm{f}_{YL^{\vec{A}}(E;X)} \eqsim \norm{(s_{n,k})_{(n,k)}}_{y^{\vec{A}}(E)}.
\end{align*}
\end{theorem}

\begin{remark}\label{IR:rmk:thm:HN_T1.1.14_YL}
Theorem~\ref{IR:thm:HN_T1.1.14_YL} is partial extension of
\cite[Theorem~1.1.14]{Hedberg&Netrusov2007}, which is concerned with $YL(E)$ with $E \in \mathcal{S}(\varepsilon_{+},\varepsilon_{-},I,\vec{r})$.
That result actually extends completely to the anisotropic scalar-valued setting
$YL^{\vec{A}}(E)$ with $E \in \mathcal{S}(\varepsilon_{+},\varepsilon_{-},\vec{A},\vec{r})$.
However, in the general Banach space-valued case there arises a difficulty due to the unavailability of the Whitney inequality \cite[(1.2.2)/Theorem~A.1]{Hedberg&Netrusov2007} (see \cite{Whitney1934,Whitney1948}) and the derived Lemma~\ref{IR:lemma:HN_L1.2.1}.
We overcome this issue in Theorem~\ref{IR:thm:HN_T1.1.14_YL_widetilde} by extending \cite[Theorem~1.1.14]{Hedberg&Netrusov2007} to $\widetilde{YL}^{\vec{A}}(E;X)$ instead of $YL^{\vec{A}}(E;X)$ (recall Remark~\ref{IR:rmk:def:YL_widetilde;scalar-valued}).
This was actually the motivation for introducing the space $\widetilde{YL}^{\vec{A}}(E;X)$, which is connected to $YL^{\vec{A}}(E;X)$ and $Y^{\vec{A}}(E;X)$ through Theorem~\ref{IR:thm:incl_comparY&YL}. 
\end{remark}

\begin{theorem}\label{IR:thm:HN_T1.1.14_YL_widetilde}
Let $E \in \mathcal{S}(\varepsilon_{+},\varepsilon_{-},\vec{A},\vec{r},(S,\mathscr{A},\mu))$ and suppose that $\varepsilon_{+},\varepsilon_{-}>0$.
Let $\vec{p} \in (0,\infty)^{\ell}$ and $M \in \N$ satisfy $\varepsilon_{+}>\mathrm{tr}(\vec{A})\bcdot(\vec{r}^{-1}-\vec{p}^{-1})$ and $M\lambda^{\vec{A}}_{\min} > \varepsilon_{-}$.
Given $f \in L_{0}(S;L_{\vec{r},\mathpzc{d}}(\R^{d};X))$, consider the following statements:
\begin{enumerate}[(I)]
\item\label{IR:it:thm:HN_T1.1.14_YL_widetilde;(i)} $f \in \widetilde{YL}^{\vec{A}}(E;X)$.
\item \label{IR:it:thm:HN_T1.1.14_YL_widetilde;(iv)} There exist $(s_{x^{*},n,k})_{(n,k)} \in \widetilde{y}^{\vec{A}}(E;X)$ and $(b_{x^{*},n,k})_{(x^{*},n,k) \in X^{*} \times \N \times \Z^{d}} \subset L_{0}(S;C^{M}_{c}([-1,2]^{d}))$ with $\norm{b_{x^{*},n,k}}_{C^{M}_{b}} \leq 1$ such that, setting $a_{x^{*},n,k}:=b_{x^{*}n,k}(\vec{A}_{2^{n}}\,\cdot\,-k)$, for all $x^{*} \in X^{*}$, $\ip{f}{x^{*}}$ has the representation
    \[
    \ip{f}{x^{*}} = \sum_{(n,k) \in \N \times \Z^{d}}s_{x^{*},n,k}a_{x^{*},n,k} \qquad \text{in} \qquad
    L_{0}(S;L_{\vec{p},\mathpzc{d},\loc}(\R^{d})).
    \]
\item\label{IR:it:thm:HN_T1.1.14_YL_widetilde;(ii)}
$f \in E_{0}(X) \cap L_{0}(S;L_{\vec{p},\mathpzc{d},\loc}(\R^{d};X))$ and
\[
\{d^{\vec{A},\vec{p}}_{M,x^{*},n}(f)\}_{(x^{*},n) \in X^{*}\times\N_{\geq1}} \in \mathscr{F}_{\mathrm{M}}(X^{*};E(\N_{1})),
\]
where
\[
d^{\vec{A},\vec{p}}_{M,x^{*},n}(f) := 2^{n\mathrm{tr}(\vec{A})\bcdot\vec{p}^{-1}}\normb{z \mapsto \Delta_{z}^{M}\ip{f}{x^{*}}}_{L_{\vec{p},\mathpzc{d}}(B^{\vec{A}}(0,2^{-n}))}, \qquad n \in \N.
\]
\item\label{IR:it:thm:HN_T1.1.14_YL_widetilde;(iii)}
$f \in E_{0}(X) \cap L_{0}(S;L_{\vec{p},\mathpzc{d},\loc}(\R^{d};X))$ and
\[
\{\mathcal{E}^{\vec{A},\vec{p}}_{M,x^{*},n}(f)\}_{(x^{*},n) \in X^{*} \times \N_{1}} \in \mathscr{F}_{\mathrm{M}}(X^{*};E(\N_{1})),
\]
where
    \[
     \mathcal{E}^{\vec{A},\vec{p}}_{M,x^{*},n}(f)(x) := \overline{\mathcal{E}}_{M}(\ip{f}{x^{*}},B^{\vec{A}}(x,2^{-n}),L_{\vec{p},\mathpzc{d}}), \qquad x^{*} \in X^{*}, n \in \N.
    \]
\item \label{IR:it:thm:HN_T1.1.14_YL_widetilde;(v)} $f \in E_{0}(X)$ and there is $\{\pi_{x^{*},n,k}\}_{(x^{*},n,k) \in X^{*} \times \N_{1} \times \Z} \in \mathcal{P}^{d}_{M-1}$ such that
    \[
    g_{x^{*},n}:=
    \sum_{k \in \Z^{d}}|\ip{f}{x^{*}}-\pi_{x^{*},n,k}|\,1_{Q^{\vec{A}}_{n,k}(3)},\qquad n \geq 1,
    \]
    satisfies $\{g_{x^{*},n}\}_{(x^{*},n) \in X^{*} \times \N_{1}}  \in \mathscr{F}_{\mathrm{M}}(X^{*};E(\N_{1}))$.
\end{enumerate}
For $f \in L_{0}(S;L_{\vec{r},\mathpzc{d}}(\R^{d};X))$ it holds that \eqref{IR:it:thm:HN_T1.1.14_YL_widetilde;(v)} $\ra$ \eqref{IR:it:thm:HN_T1.1.14_YL_widetilde;(i)} $\lra$ \eqref{IR:it:thm:HN_T1.1.14_YL_widetilde;(iv)} $\ra$ \eqref{IR:it:thm:HN_T1.1.14_YL_widetilde;(ii)} $\&$ \eqref{IR:it:thm:HN_T1.1.14_YL_widetilde;(iii)} with corresponding estimates
\begin{align*}
\norm{f}_{E_{0}(X)}& + \norm{(d^{\vec{A},\vec{p}}_{M,x^{*},n}(f))_{(x^{*},n)}}_{\mathscr{F}_{\mathrm{M}}(X^{*};E)} + \norm{\mathcal{E}^{\vec{A},\vec{p}}_{M,x^{*},n}(f)\}_{(x^{*},n) \in X^{*} \times \N_{1}}}_{\mathscr{F}_{\mathrm{M}}(X^{*};E(\N_{\geq1}))} \\
&\qquad\lesssim \norm{f}_{\widetilde{YL}^{\vec{A}}(E;X)}
\eqsim \norm{(s_{x^{*},n,k})_{(x^{*},n,k)}}_{\widetilde{y}^{\vec{A}}(E)} \\
&\qquad\lesssim \norm{f}_{E_{0}(X)} + \normb{\{g_{x^{*},n}\}_{(x^{*},n) \in X^{*}\times\N_{\geq1}}}_{\mathscr{F}_{\mathrm{M}}(X^{*};E(\N_{1}))}.
\end{align*}
Moreover, for $f$ of the form $f=\sum_{i \in I}1_{S_{i}} \otimes f^{[i]}$ with $(S_{i})_{i \in I} \subset \mathscr{A}$ a countable family of mutually disjoint sets and $(f^{[i]})_{i \in I} \in L_{\vec{r},\mathpzc{d},\loc}(\R^{d};X)$, it holds that \eqref{IR:it:thm:HN_T1.1.14_YL_widetilde;(i)},
\eqref{IR:it:thm:HN_T1.1.14_YL_widetilde;(iv)},
\eqref{IR:it:thm:HN_T1.1.14_YL_widetilde;(ii)}, \eqref{IR:it:thm:HN_T1.1.14_YL_widetilde;(iii)},  and \eqref{IR:it:thm:HN_T1.1.14_YL_widetilde;(v)} are equivalent statements and there are the corresponding estimates
\begin{align*}
\norm{f}_{\widetilde{YL}^{\vec{A}}(E;X)}
&\eqsim \norm{(s_{x^{*},n,k})_{(x^{*},n,k)}}_{\widetilde{y}^{\vec{A}}(E)} \\
&\eqsim \norm{f}_{E_{0}(X)} + \normb{\{d^{\vec{A},\vec{p}}_{M,x^{*},n}(f)\}_{(x^{*},n) \in X^{*}\times\N_{\geq1}}}_{\mathscr{F}_{\mathrm{M}}(X^{*};E(\N_{1}))} \\
&\eqsim \norm{f}_{E_{0}(X)} + \normb{(\mathcal{E}^{\vec{A},\vec{p}}_{M,n}(f))_{n}}_{E}\\
&\eqsim \norm{f}_{E_{0}(X)} + \normb{\{g_{x^{*},n}\}_{(x^{*},n) \in X^{*}\times\N_{\geq1}}}_{\mathscr{F}_{\mathrm{M}}(X^{*};E(\N_{1}))}.
\end{align*}
\end{theorem}

\begin{corollary}\label{IR:cor:thm:HN_T1.1.14_YL(_widetilde);Y}
Let $E \in \mathcal{S}(\varepsilon_{+},\varepsilon_{-},\vec{A},\vec{r},(S,\mathscr{A},\mu))$ and suppose that $\varepsilon_{+} > \mathrm{tr}(\vec{A})\bcdot (\vec{r}^{-1}-\vec{1})_{+}$.
Let $\vec{p} \in (0,\infty]^{\ell}$ and $M \in \N$ satisfy $\varepsilon_{+}>\mathrm{tr}(\vec{A})\bcdot(\vec{r}^{-1}-\vec{p}^{-1})$ and $M\lambda^{\vec{A}}_{\min} > \varepsilon_{-}$.
Then, for each $f \in L_{0}(S;L_{\vec{r},\mathpzc{d}}(\R^{d};X))$ of the form $f=\sum_{i \in I}1_{S_{i}} \otimes f^{[i]}$ with $(S_{i})_{i \in I} \subset \mathscr{A}$ a countable family of mutually disjoint sets and $(f^{[i]})_{i \in I} \in L_{\vec{r},\mathpzc{d},\loc}(\R^{d};X)$,
\begin{align*}
\norm{f}_{Y^{\vec{A}}(E;X)} \eqsim
\norm{f}_{YL^{\vec{A}}(E;X)}
\eqsim \norm{f}_{E_{0}(X)} + \norm{(d^{\vec{A},\vec{p}}_{M,n}(f))_{n \geq 1}}_{E(\N_{1})}.
\end{align*}
\end{corollary}

Theorem~\ref{IR:thm:intro:difference_norm} from the introduction can be obtained as a special case of the following theorem.

\begin{theorem}\label{IR:thm:difference_norm_'Ap-case'}
Let $E \in \mathcal{S}(\varepsilon_{+},\varepsilon_{-},\vec{A},\vec{1},(S,\mathscr{A},\mu))$ and suppose that $\varepsilon_{+},\varepsilon_{-}>0$.
Let $\vec{p} \in [1,\infty]^{\ell}$ and $M \in \N$ satisfy $\varepsilon_{+}>\mathrm{tr}(\vec{A})\bcdot(\vec{1}-\vec{p}^{-1})$ and $M\lambda^{\vec{A}}_{\min} > \varepsilon_{-}$. Write
\[
I^{\vec{A}}_{M,n}(f) := 2^{n\mathrm{tr}(\vec{A}^{\oplus})}\int_{B^{\vec{A}}(0,2^{-n})}\Delta_{z}^{M}f\,dz, \qquad f \in L_{0}(S;L_{1,\loc}(\R^{d};X)).
\]
Then
\begin{align*}
\norm{f}_{Y^{\vec{A}}(E;X)}
&\eqsim \norm{f}_{YL^{\vec{A}}(E;X)} \eqsim \norm{f}_{\widetilde{YL}^{\vec{A}}(E;X)}\\
&\eqsim \norm{f}_{E_{0}(X)} + \norm{(I^{\vec{A}}_{M,n}(f))_{n \geq 1}}_{E(\N_{1};X)} \\
&\eqsim \norm{f}_{E_{0}(X)} + \norm{(d^{\vec{A},\vec{p}}_{M,n}(f))_{n \geq 1}}_{E(\N_{1};X)}
\end{align*}
for all $f \in E_{0}(X) \hookrightarrow E_{i} \hookrightarrow E^{\vec{A}}_{\otimes}[B^{\vec{r},w_{\vec{A},\vec{r}}}_{\vec{A}}](X)$ (see Remark~\ref{IR:rmk:lemma:HN_L1.1.4}).
\end{theorem}

\begin{remark}\label{IR:rmk:thm:difference_norm_'Ap-case';Bessel_potential}
Recall from Example~\ref{IR:ex:prop:LP-decomp_characterization} that, in case $\ell=1$, $A=I$, $p \in (1,\infty)$, $q = 2$, $w \in A_{p}(\R^{d})$, $F$ is a UMD Banach function space and $X$ is a Hilbert space, $\F^{s,\vec{A}}_{\vec{p},q}(\R^{d},\vec{w};F;X)$ coincides with the weighted vector-valued Bessel potential space $H^{s}_{p}(\R^{d},w;F(X))$. Theorem~\ref{IR:thm:difference_norm_'Ap-case'} thus especially gives a difference norm characterization for $H^{s}_{p}(\R^{d},w;F(X))$ (cf.\ \cite[Remark~4.10]{Lindemulder2016_JFA}).
\end{remark}

\begin{proposition}\label{IR:prop:HN_T1.1.14;difference+shift_one-sided}
Let $E \in \mathcal{S}(\varepsilon_{+},\varepsilon_{-},\vec{A},\vec{r},(S,\mathscr{A},\mu))$ and suppose that $\varepsilon_{+},\varepsilon_{-}>0$.
Let $c \in \R$.
Let $\vec{p} \in (0,\infty]^{\ell}$ and $M \in \N$ satisfy $\varepsilon_{+}>\mathrm{tr}(\vec{A})\bcdot(\vec{r}^{-1}-\vec{p}^{-1})$ and $M>\varepsilon_{-}$.
Then
\begin{align*}
\norm{\{d^{\vec{A},\vec{p}}_{M,c,,n}(f)\}_{n}}_{E(X)} \lesssim \norm{f}_{YL^{\vec{A}}(E;X)}, \qquad f \in L_{0}(S;L_{\vec{r},\mathpzc{d}}(\R^{d};X)),
\end{align*}
and
\begin{align*}
\norm{\{d^{\vec{A},\vec{p}}_{M,c,x^{*},n}(f)\}_{(x^{*},n)}}_{\mathscr{F}_{\mathrm{M}}(X^{*};E)} \lesssim \norm{f}_{\widetilde{YL}^{\vec{A}}(E;X)}, \qquad f \in L_{0}(S;L_{\vec{r},\mathpzc{d}}(\R^{d};X)),
\end{align*}
where
\[
d^{\vec{A},\vec{p}}_{M,c,n}(f) := 2^{n\mathrm{tr}(\vec{A})\bcdot\vec{p}^{-1}}\normb{z \mapsto L_{cz}\Delta_{z}^{M}f}_{L_{\vec{p},\mathpzc{d}}(B^{\vec{A}}(0,2^{-n};X))}
\]
and
\[
d^{\vec{A},\vec{p}}_{M,c,x^{*},n}(f) := 2^{n\mathrm{tr}(\vec{A})\bcdot\vec{p}^{-1}}\normb{z \mapsto L_{cz}\Delta_{z}^{M}\ip{f}{x^{*}}}_{L_{\vec{p},\mathpzc{d}}(B^{\vec{A}}(0,2^{-n}))}.
\]
\end{proposition}

\subsection{Some lemmas}

\begin{lemma}\label{IR:lemma:estimate_yA_fixed_n}
Let $E \in \mathcal{S}(\varepsilon_{+},\varepsilon_{-},\vec{A},\vec{r},(S,\mathscr{A},\mu))$.
Put $C:=\max_{x \in [0,1]^{d}}\rho_{\vec{A}}(x) \in [1,\infty)$. Then, for each $(s_{n,k})_{(n,k)} \in y^{\vec{A}}(E)$,
\[
\norm{s_{n,k}}_{E^{\vec{A}}_{\otimes}} \lesssim_{n,\vec{A},\vec{r}} (C+\rho_{\vec{A}}(k))^{\mathrm{tr}(\vec{A})\bcdot\vec{r}^{-1}},
\qquad (n,k) \in \N \times \Z^{d}.
\]
\end{lemma}
\begin{proof}
Fix $(i,l) \in \N \times \Z^{d}$. By Remark~\ref{IR:rmk:lemma:HN_L1.1.4}, $E_{i} \hookrightarrow E^{\vec{A}}_{\otimes}[B^{\vec{r},w_{\vec{A},\vec{r}}}_{\vec{A}}]$, so that
\begin{align}
\norm{s_{i,l}}_{E^{\vec{A}}_{\otimes}}
\norm{\chi^{\vec{A}}_{i,l}}_{B^{\vec{r},w_{\vec{A},\vec{r}}}_{\vec{A}}}
&= \norm{s_{i,l}\chi^{\vec{A}}_{i,l}}_{E^{\vec{A}}_{\otimes}[B^{\vec{r},w_{\vec{A},\vec{r}}}_{\vec{A}}]} \lesssim_{i} \norm{s_{i,l}\chi^{\vec{A}}_{i,l}}_{E_{i}} \nonumber \\
&\leq \normB{\big(\sum_{k \in \Z^{d}}s_{n,k}\chi^{\vec{A}}_{n,k}\big)_{n}}_{E}
= \norm{(s_{n,k})_{(n,k)}}_{y^{\vec{A}}(E)}. \label{IR:eq:lemma:estimate_yA_fixed_n;1}
\end{align}

Let $\vec{R}=(R,\ldots,R) \in [1,\infty)^{\ell}$ be given by $R:= c_{\vec{A}}(C+\rho_{\vec{A}}(l))$. Then
\[
\rho_{\vec{A}}(x+l) \leq c_{\vec{A}}(\rho_{\vec{A}}(x)+\rho_{\vec{A}}(l)) \leq c_{\vec{A}}(C+\rho_{\vec{A}}(l)) = R \leq 2^{i}R, \qquad x \in [0,1]^{d}.
\]
Therefore,
\[
\supp(\chi^{\vec{A}}_{i,l}) = \vec{A}_{2^{-i}}([0,1]^{d}+l) \subset B^{\vec{A}}(0,\vec{R}).
\]
As a consequence,
\begin{equation}\label{IR:eq:lemma:estimate_yA_fixed_n;2}
[c_{\vec{A}}(C+\rho_{\vec{A}}(l))]^{-\mathrm{tr}(\vec{A})\bcdot\vec{r}^{-1}}
\norm{\chi^{\vec{A}}_{i,l}}_{L_{\vec{r},\mathpzc{d}}(\R^{d})} \leq \norm{\chi^{\vec{A}}_{i,l}}_{B^{\vec{r},w_{\vec{A},\vec{r}}}_{\vec{A}}}
\end{equation}

Observing that $\norm{\chi^{\vec{A}}_{i,l}}_{L_{\vec{r},\mathpzc{d}}(\R^{d})} = c_{i,\vec{A},\vec{r}}$,
a combination of \eqref{IR:eq:lemma:estimate_yA_fixed_n;1} and \eqref{IR:eq:lemma:estimate_yA_fixed_n;2} gives the desired result.
\end{proof}

\begin{lemma}\label{IR:lemma:HN_L1.2.1}
Let $p \in (0,\infty]$ and $M \in \N_{1}$. Then there is a constant $C=C_{M,p,d}$ such that, if $f \in L_{p,\loc}(\R^{d})$ and $Q = \vec{A}_{\lambda}([0,1)^{d}+b)$ with $\lambda \in (0,\infty)$ and $b \in \R^{d}$, then there is $\pi \in \mathcal{P}^{d}_{M-1}$ satisfying (with the usual modification if $p=\infty$):
\begin{align*}
|f-\pi|\,1_{Q}
&\leq C\left(\fint_{B^{\vec{A}}(0,\lambda)}|\Delta^{M}_{z}f|^{p}\,dz\right)^{1/p} \\
&\qquad\qquad + \: C\left( \fint_{B^{\vec{A}}(0,\lambda)}\fint_{Q(2)}|\Delta^{M}_{z}f|^{p}\,dy\,dz\right)^{1/p}.
\end{align*}
\end{lemma}
\begin{proof}
The case $\lambda=1$ is contained in \cite[Lemma~1.2.1]{Hedberg&Netrusov2007}, from which the general case can be obtained by a scaling argument.
\end{proof}

From Lemma~\ref{IR:lemma:HN_L1.2.2;abstract_complemented subspace} to Corollary~\ref{IR:cor:lemma:HN_L1.2.2;abstract_complemented subspace;top_isom;polynomials} we will actually only use Corollary~\ref{IR:cor:lemma:HN_L1.2.2;abstract_complemented subspace;top_isom;polynomials} in the scalar-valued case in the proof of Theorem~\ref{IR:thm:HN_T1.1.14_YL_widetilde}.
However, although the scalar-valued case is easier, we have decided to present it in this way as it could be useful for potential extensions of Theorem~\ref{IR:thm:HN_T1.1.14_YL} along these lines.
In the latter the main obstacle is Lemma~\ref{IR:lemma:HN_L1.2.1}.

We write $\mathcal{P}^{d}_{N}(X) \simeq X^{M_{N,d}}$, where $M_{N,d} := \#\{ \alpha \in \N^{d} : |\alpha| \leq M \}$, for the space of $X$-valued polynomials of degree at most $N$ on $\R^{d}$.

\begin{lemma}\label{IR:lemma:HN_L1.2.2;abstract_complemented subspace}
Let $(T,\mathscr{B},\nu)$ a measure space, $\F \subset L_{2}(T)$ a finite dimensional subspace, $\E \subset L_{0}(T;X)$ a topological vector space with $\F \otimes X \subset \E$ such that
\[
\F \times X \longra \E,\,(p,f) \mapsto f \otimes x,
\]
and
\[
\F \times \E \longra L_{1}(T;X), (f,g) \mapsto fg,
\]
are well-defined bilinear mappings that are continuous with respect to the second variable.
Then $\F \otimes X$ is a complemented subspace of $\E$.
\end{lemma}
\begin{proof}
Choose an orthogonal basis $b_{1},\ldots,b_{n}$ of the finite dimensional subspace $\F$ of $L_{2}(T)$.
Then
\[
\pi:\E \longra \E,\,
g \mapsto \sum_{i=1}^{n}\left[\int_{T}b_{i}(t)g(t)\mathrm{d}\nu(t)\right] \otimes b_{i},
\]
is a well-defined continuous linear mapping on $\E$, which is a projection onto the linear subspace $\F \otimes X \subset \E$.
\end{proof}

\begin{corollary}\label{IR:cor:lemma:HN_L1.2.2;abstract_complemented subspace;top_isom}
If $\E$ in Lemma~\ref{IR:lemma:HN_L1.2.2;abstract_complemented subspace} is an $F$-space, then so is $(\F \otimes X,\tau_{\E})$. As a consequence, if $\tau$ is a topological vector space topology on $\F \otimes X$ with $(\F \otimes X,\tau_{\E}) \hookrightarrow (\F \otimes X,\tau)$, then the latter is in fact a topological isomorphism.
\end{corollary}

\begin{corollary}\label{IR:cor:lemma:HN_L1.2.2;abstract_complemented subspace;top_isom;polynomials}
Let $B = [-1,2]^{d}$, $N \in \N$ and $q \in [1,\infty)$. Set $B_{n,k}:=\vec{A}_{2^{-n}}(B+k)$ for $(n,k) \in \N \times \Z^{d}$. Then
\[
\norm{\pi(\vec{A}_{2^{-n}}\,\cdot\,+k)}_{C^{N}_{b}(B;X)} \lesssim 2^{n\mathrm{tr}(\vec{A}^{\oplus})/q}\norm{\pi}_{L_{q}(B_{n,k};X)},
\qquad \pi \in \mathcal{P}^{d}_{N}(X), (n,k) \in \N \times \Z^{d}.
\]
\end{corollary}
\begin{proof}
Let us first note that a substitution gives
\[
\norm{\pi(\vec{A}_{2^{-n}}\,\cdot\,+k)}_{L_{q}(B;X)} = 2^{n\mathrm{tr}(\vec{A}^{\oplus})/q}\norm{\pi}_{L_{q}(B_{n,k};X)},
\]
while $\pi(\vec{A}_{2^{-n}}\,\cdot\,+k) \in \mathcal{P}_{N}^{d}(X)$.
Applying Corollary~\ref{IR:cor:lemma:HN_L1.2.2;abstract_complemented subspace;top_isom} to $\F=\mathcal{P}_{N}^{d}$, viewed as finite dimensional subspace of $L_{2}(B)$, and $\E=C^{N}_{n}(B;X)$ and $\tau$ the topology on $\mathcal{P}_{N}(X) = \F \otimes X$ induced from $L_{q}(B;X)$, we obtain the desired result.
\end{proof}

\begin{lemma}\label{IR:lemma:HN_L1.2.3}
Let $q,p \in (0,\infty)$, $q \leq p$, $b \in (0,\infty)$ and $M \in \N_{1}$.
Let $f \in L_{p,\loc}(\R^{d})$ and let $\{\pi_{n,k}\}_{(n,k) \in \N \times \Z^{d}} \subset \mathcal{P}^{d}_{M-1}$ such that
\[
\norm{f-\pi_{n,k}}_{L_{q}(Q^{\vec{A}}_{n,k}(b))} \leq 2\mathcal{E}_{M}(f,Q^{\vec{A}}_{n,k}(b),L_{q}),
\]
and let $\{\phi_{n,k}\}_{(n,k) \in \N \times \Z^{d}} \subset L_{\infty}(\R^{d})$ be such that $\supp \phi_{n,k} \subset Q^{\vec{A}}_{n,k}(b)$, $\sum_{k \in \Z^{d}}\phi_{n,k} \equiv 1$, and $\norm{\phi_{n,k}}_{L_{\infty}} \leq 1$. Then, for $(f_{n})_{n \in \N} \subset L_{0}(S)$ defined by
\[
f_{n} := \sum_{k \in \Z^{d}}\pi_{n,k}\phi_{n,k},
\]
there is the convergence $f=\lim_{n \to \infty}f_{n}$ almost everywhere and in $L_{p,\loc}$.
\end{lemma}
\begin{proof}
This can be proved as in \cite[Lemma~1.2.3]{Hedberg&Netrusov2007}.
\end{proof}

\begin{lemma}\label{IR:lemma:HN_L1.2.4}
Let $E \in \mathcal{S}(\varepsilon_{+},\varepsilon_{-},\vec{A},\vec{r},(S,\mathscr{A},\mu))$, $b \in (0,\infty)$ and suppose that $\varepsilon_{+},\varepsilon_{-}>0$. Let $\vec{p} \in (0,\infty]^{\ell}$ satisfy $\varepsilon_{+}>\mathrm{tr}(\vec{A})\bcdot(\vec{r}^{-1}-\vec{p}^{-1})$.
Define the sublinear operator
\[
T_{\vec{p}}^{\vec{A}}:L_{0}(S)^{\N \times \Z^{d}} \longra L_{0}(S;[0,\infty])^{\N \times \Z^{d}}, \qquad
(s_{n,k})_{(n,k)} \mapsto (t_{n,k})_{(n,k)},
\]
by
\[
t_{n,k} := 2^{n\mathrm{tr}(\vec{A})\bcdot \vec{p}^{-1}}\normB{\sum_{m,l}|s_{m,l}|\chi^{\vec{A}}_{m,l}}_{L_{\vec{p},\mathpzc{d}}}
\]
and the sum is taken over all indices $(m,l) \in \N \times \Z^{d}$ such that $Q^{\vec{A}}_{m,l} \subset Q^{\vec{A}}_{n,k}(b)$ and $m \geq n$.
Then $T^{\vec{A}}_{\vec{p}}$ restricts to a bounded sublinear operator on $y^{\vec{A}}(E)$.
\end{lemma}
\begin{proof}
Let $(s_{n,k})_{(n,k)} \in y^{\vec{A}}(E)$ and $(t_{n,k})_{(n,k)} = T^{\vec{A}}_{\vec{p}}[(s_{n,k})_{(n,k)}] \in L_{0}(S;[0,\infty])^{\N \times \Z^{d}}$.
We need to show that $\norm{(t_{n,k})}_{y^{\vec{A}}(E)} \lesssim \norm{(s_{n,k})}_{y^{\vec{A}}(E)}$. Here we may without loss of generality assume that $s_{n,k} \geq 0$ for all $(n,k)$.

Set
\[
\delta:= \frac{1}{2}\left(\varepsilon_{+}-\mathrm{tr}(A)\bcdot(\vec{r}^{-1}-\vec{p}^{-1})\right) \in (0,\infty).
\]
Define
\[
g_{m} := \sum_{l \in \Z^{d}}s_{m,l}\chi^{\vec{A}}_{m,l} \in L_{0}(S), \qquad m \in \N.
\]
Then
\begin{equation}\label{IR:eq:lemma:HN_L1.2.4;proof;1}
t_{n,k} \leq 2^{n\mathrm{tr}(\vec{A})\bcdot \vec{p}^{-1}}\normB{\sum_{m=n}^{\infty}g_{m}}_{L_{\vec{p},\mathpzc{d}}(Q^{\vec{A}}_{n,k}(b))}.
\end{equation}
As the right-hand side is increasing in $\vec{p}$ by H\"older's inequality, it suffices to consider the case $\vec{p} \geq \vec{r}$.

Several applications of the elementary embedding
\[
\ell_{q_{0}}^{s_{0}}(\N) \hookrightarrow \ell_{q_{1}}^{s_{1}}(\N), \qquad s_{0}>s_{1} , q_{0},q_{1} \in (0,\infty],
\]
in combination with Fubini's theorem yield that
\begin{equation}\label{IR:eq:lemma:HN_L1.2.4;proof;2}
\normB{\sum_{m=n}^{\infty}g_{m}}_{L_{\vec{p},\mathpzc{d}}(Q^{\vec{A}}_{n,k}(b))}
\lesssim \sum_{m=n}^{\infty}2^{(m-n)\delta}\norm{g_{m}}_{L_{\vec{p},\mathpzc{d}}(Q^{\vec{A}}_{n,k}(b))}.
\end{equation}

In order to estimate the summands on the right-hand side of \eqref{IR:eq:lemma:HN_L1.2.4;proof;1}, we will use the following fact.
Let $(T_{1},\mathscr{B}_{1},\nu_{1}),\ldots,(T_{\ell},\mathscr{B}_{\ell},\nu_{\ell})$ be $\sigma$-finite measure spaces and let $I_{1},\ldots,I_{\ell}$ be countable sets. Put $T=T_{1} \times \ldots \times T_{\ell}$ and $I=I_{1}\times\ldots\times I_{\ell}$. Let $(c_{\vec{i}})_{\vec{i} \in I} \subset \C$ and, for each $j \in \{1,\ldots,\ell\}$, let $(A^{(j)}_{i_{j} \in I_{j}}) \subset \mathscr{B}_{j}$ be a sequence of mutually disjoint sets.
Then
\begin{equation}\label{IR:eq:lemma:HN_L1.2.4;proof;3}
\normb{\sum_{\vec{i} \in I}c_{\vec{i}}1_{A^{(1)}_{i_{1}} \times\ldots\times A^{(\ell)}_{i_{\ell}}} }_{L_{\vec{p}}(T)} \leq
\left(\sup_{\vec{i} \in I}\prod_{j=1}^{\ell}|A^{(j)}_{i_{j}}|^{\frac{1}{p_{j}}-\frac{1}{r_{j}}}\right) \normb{\sum_{\vec{i} \in I}c_{\vec{i}}1_{A^{(1)}_{i_{1}} \times\ldots\times A^{(\ell)}_{i_{\ell}}} }_{L_{\vec{r}}(T)}.
\end{equation}
Indeed,
\begin{align*}
& \normb{\sum_{\vec{i} \in I}c_{\vec{i}}1_{A^{(1)}_{i_{1}} \times\ldots\times A^{(\ell)}_{i_{\ell}}} }_{L_{\vec{p}}(T)} \\
&\:= \left(\sum_{i_{\ell} \in I_{\ell}}|A^{(\ell)}_{i_{\ell}}|\left(\ldots\left(\sum_{i_{1} \in I_{1}}|A^{(1)}_{i_{1}}|\,|c_{\vec{i}}|^{p_{1}} \right)^{p_{2}/p_{1}}\ldots\right)^{p_{\ell}/p_{\ell-1}} \right)^{1/p_{\ell}} \\
&\:\leq \left(\sum_{i_{\ell} \in I_{\ell}}|A^{(\ell)}_{i_{\ell}}|^{r_{\ell}/p_{\ell}}\left(\ldots\left(\sum_{i_{1} \in I_{1}}|A^{(1)}_{i_{1}}|^{r_{1}/p_{1}}\,|c_{\vec{i}}|^{r_{1}} \right)^{r_{2}/r_{1}}\ldots\right)^{r_{\ell}/r_{\ell-1}} \right)^{1/r_{\ell}} \\
&\:\leq \left(\sup_{\vec{i} \in I}\prod_{j=1}^{\ell}|A^{(j)}_{i_{j}}|^{\frac{1}{p_{j}}-\frac{1}{r_{j}}}\right)
\left(\sum_{i_{\ell} \in I_{\ell}}|A^{(\ell)}_{i_{\ell}}|\left(\ldots\left(\sum_{i_{1} \in I_{1}}|A^{(1)}_{i_{1}}|\,|c_{\vec{i}}|^{r_{1}} \right)^{r_{2}/r_{1}}\ldots\right)^{r_{\ell}/r_{\ell-1}} \right)^{1/r_{\ell}} \\
&\:= \left(\sup_{\vec{i} \in I}\prod_{j=1}^{\ell}|A^{(j)}_{i_{j}}|^{\frac{1}{p_{j}}-\frac{1}{r_{j}}}\right) \normb{\sum_{\vec{i} \in I}c_{\vec{i}}1_{A^{(1)}_{i_{1}} \times\ldots\times A^{(\ell)}_{i_{\ell}}} }_{L_{\vec{r}}(T)},
\end{align*}
where we used $\vec{p} \geq \vec{r}$ in the first inequality.

Let us now use the above fact to estimate $\norm{g_{m}}_{L_{\vec{p},\mathpzc{d}}(Q^{\vec{A}}_{n,k}(b))}$:
\begin{align}
\norm{g_{m}}_{L_{\vec{p},\mathpzc{d}}(Q^{\vec{A}}_{n,k}(b))}
&\leq \normB{\sum_{l \in \Z^{d}:Q^{\vec{A}}_{m,l} \cap Q^{\vec{A}}_{n,k}(b) \neq \emptyset}s_{m,l}\chi^{\vec{A}}_{m,l}}_{L_{\vec{p},\mathpzc{d}}(\R^{d})} \nonumber\\
&\stackrel{\eqref{IR:eq:lemma:HN_L1.2.4;proof;3}}{\leq} 2^{-m\mathrm{tr}(\vec{A})\bcdot(\vec{p}^{-1}-\vec{r}^{-1})}\normB{\sum_{l \in \Z^{d}:Q^{\vec{A}}_{m,l} \cap Q^{\vec{A}}_{n,k}(b) \neq \emptyset}s_{m,l}\chi^{\vec{A}}_{m,l}}_{L_{\vec{r},\mathpzc{d}}(\R^{d})} \nonumber \\
&\leq 2^{-m\mathrm{tr}(\vec{A})\bcdot(\vec{p}^{-1}-\vec{r}^{-1})}
\norm{g_{m}}_{L_{\vec{r},\mathpzc{d}}(Q^{\vec{A}}_{n,k}(b+2)))} \nonumber \\
&= 2^{(m-n)((\varepsilon_{+}-2\delta))-n\mathrm{tr}(\vec{A})\bcdot(\vec{p}^{-1}-\vec{r}^{-1})}
\norm{g_{m}}_{L_{\vec{r},\mathpzc{d}}(Q^{\vec{A}}_{n,k}(b+2)))} \label{IR:eq:lemma:HN_L1.2.4;proof;4}
\end{align}

Putting \eqref{IR:eq:lemma:HN_L1.2.4;proof;1}, \eqref{IR:eq:lemma:HN_L1.2.4;proof;2} and \eqref{IR:eq:lemma:HN_L1.2.4;proof;4} together, we obtain
\begin{align}
t_{n,k}\chi^{\vec{A}}_{n,k}
&\leq \sum_{m=n}^{\infty}2^{(m-n)((\varepsilon_{+}-\delta))+n\mathrm{tr}(\vec{A})\bcdot\vec{r}^{-1}}
\norm{g_{m}}_{L_{\vec{r},\mathpzc{d}}(Q^{\vec{A}}_{n,k}(b+2)))}\chi^{\vec{A}}_{n,k}
 \nonumber \\
&\lesssim_{b,\vec{A},\vec{r}} \sum_{m=n}^{\infty}2^{(m-n)(\varepsilon_{+}-\delta)} M^{\vec{A}}_{\vec{r}}(g_{m}). \label{IR:eq:lemma:HN_L1.2.4;proof;5}
\end{align}
Since
\[
\left( \sum_{m=n}^{\infty}2^{(m-n)(\varepsilon_{+}-\delta)} M^{\vec{A}}_{\vec{r}}(g_{m}) \right)_{n \in \N}
= \sum_{i=0}^{\infty}2^{i(\varepsilon_{+}-\delta)}(S_{+})^{i}M^{\vec{A}}_{\vec{r}}\left[(g_{n})_{n \in \N}\right],
\]
it follows that $(t_{n,k}) \in y^{\vec{A}}(E)$ with
\begin{align}
\norm{(t_{n,k})}_{y^{\vec{A}}(E)}^{\kappa}
&= \normB{\big(\sum_{k \in \Z^{d}}t_{n,k}\chi^{\vec{A}}_{n,k}\big)_{n}}_{E}^{\kappa} \nonumber \\
&\lesssim \sum_{i=0}^{\infty}2^{\kappa i((\varepsilon_{+}-\delta))}\norm{(S_{+})^{i}M^{\vec{A}}_{\vec{r}}\left[(g_{n})_{n}\right]}_{E}^{\kappa} \nonumber \\
&\lesssim \sum_{i=0}^{\infty}2^{-\kappa i \delta))}\norm{(g_{n})_{n}}_{E}^{\kappa} \lesssim \norm{(g_{n})_{n}}_{E}^{\kappa} \nonumber \\
&= \norm{(s_{n,k})}_{y^{\vec{A}}(E)}^{\kappa}, \label{IR:eq:lemma:HN_L1.2.4;proof;6}
\end{align}
where $\kappa$ is such that $E$ has a $\kappa$-norm.
\end{proof}

\begin{corollary}\label{IR:cor:lemma:HN_L1.2.4;HN_C1.2.5}
Let $E \in \mathcal{S}(\varepsilon_{+},\varepsilon_{-},\vec{A},\vec{r},(S,\mathscr{A},\mu))$ and suppose that $\varepsilon_{+},\varepsilon_{-}>0$.
Let $\vec{p} \in (0,\infty]^{\ell}$ satisfy $\varepsilon_{+}>\mathrm{tr}(\vec{A})\bcdot(\vec{r}^{-1}-\vec{p}^{-1})$.
Given $(s_{n,k})_{(n,k)} \in y^{\vec{A}}(E)$, set $g_{n} = \sum_{k \in \Z^{d}}s_{n,k}\chi^{\vec{A}}_{n,k}$. Then $\sum_{n=0}^{\infty}|g_{n}|$ in $L_{0}(S;L_{\vec{p},\mathpzc{d},\loc}(\R^{d}))$ and the series $\sum_{n=0}^{\infty}g_{n}$ converges almost everywhere, and in $L_{0}(S;L_{\vec{p},\mathpzc{d},\loc}(\R^{d}))$ (when $\vec{p} \in (0,\infty)^{\ell}$).
\end{corollary}
\begin{proof}
This follows from \eqref{IR:eq:lemma:HN_L1.2.4;proof;6}, see \cite[Corollary~1.2.5]{Hedberg&Netrusov2007} for more details.
\end{proof}

\begin{lemma}\label{IR:lemma:HN_L1.2.6}
Let $E \in \mathcal{S}(\varepsilon_{+},\varepsilon_{-},\vec{A},\vec{r},(S,\mathscr{A},\mu))$, $b \in (0,\infty)$ and $\lambda \in (\varepsilon_{-},\infty)$.
Define the sublinear operator
\[
T_{\lambda}:L_{0}(S)^{\N \times \Z^{d}} \longra L_{0}(S;[0,\infty])^{\N \times \Z^{d}}, \qquad
(s_{n,k})_{(n,k)} \mapsto (t_{n,k})_{(n,k)},
\]
by
\[
t_{n,k} := \sum_{m,l}2^{\lambda(n-m)}|s_{m,l}|,
\]
the sum being taken over all indices $(m,l) \in \N \times \Z^{d}$ such that $Q^{\vec{A}}_{m,l}(b) \supset Q^{\vec{A}}_{n,k}$ and $m < n$.
Then $T_{\lambda}$ restricts to a bounded sublinear operator from $y^{\vec{A}}(E)$ to $y^{\vec{A}}(E)$.
\end{lemma}
\begin{proof}
This can be proved in the same way as \cite[Lemma~1.2.6]{Hedberg&Netrusov2007}.
\end{proof}

\begin{lemma}\label{IR:lemma:HN_L1.2.7}
Let $\vec{r} \in (0,1]^{\ell}$ and $\rho \in (0,1)$ satisfy $\rho < \vec{r}_{\min}$.
Let $(\gamma_{n})_{n \in \N}$ be a sequence of measurable functions on $\R^{d}$ satisfying
\[
0 \leq \gamma_{n}(x) \lesssim (1+2^{n}\rho_{\vec{A}}(x))^{-\mathrm{tr}(\vec{A}^{\oplus})/\rho}.
\]
If $(s_{n,k})_{(n,k)} \in L_{0}(S)^{\N \times \Z^{d}}$, $g_{n}= \sum_{k \in \Z^{d}}s_{n,k}\chi^{\vec{A}}_{n,k}$ and $h_{n} = \sum_{k \in \Z^{d}}|s_{n,k}|\,\gamma_{n}(\,\cdot\,-\vec{A}_{2^{-n}}k)$, then
\[
h_{n} \lesssim M^{\vec{A}}_{\vec{r}}(g_{n}),\qquad n \in \N.
\]
\end{lemma}
\begin{proof}
We may of course without loss of generality assume that $\vec{r}=(r,\ldots,r)$ with $r \in (0,1]$.
Now the statement can be established as in \cite[Lemma~1.2.7]{Hedberg&Netrusov2007}.
\end{proof}

\begin{lemma}\label{IR:lemma:HN_L1.2.8}
Let $M \in \N$, $\lambda \in (0,\infty)$ and $\Phi \in C^{M}(\R^{d};X)$ be such that
\[
(1+\rho_{\vec{A}}(x))^{\lambda}\norm{D^{\beta}\Phi(x)}_{X} \lesssim 1, \qquad x \in \R^{d}, |\beta| \leq M,
\]
and let $\Psi \in \mathcal{S}(\R^{d})$ be such that $\Psi \perp \mathcal{P}^{d}_{M-1}$.
Set $\Psi_{t} := t^{-\mathrm{tr}(\vec{A}^{\oplus})}\Psi(\vec{A}_{t^{-1}}\,\cdot\,)$ for $t \in (0,\infty)$. Then, given $\varepsilon \in (0,\lambda^{\vec{A}}_{\min})$,
\[
\norm{\Phi*\Psi_{t}\,(x)}_{X} \lesssim_{\varepsilon} \frac{t^{(\lambda^{\vec{A}}_{\min}-\varepsilon)M}}{(1+\rho_{\vec{A}}(x))^{\lambda}}, \qquad x \in \R^{d}, t \in (0,1].
\]
\end{lemma}
\begin{proof}
As $\Psi$ is a Schwartz function, there in particular exists $C \in (0,\infty)$ such that
\[
|\Psi(x)| \leq C(1+\rho_{\vec{A}}(x))^{-\lambda}(1+|x|)^{-(d+M+1)}, \qquad x \in \R^{d}.
\]
The desired inequality can now be obtained as in \cite[Lemma~1.2.8]{Hedberg&Netrusov2007}.
\end{proof}

Lemmas \ref{IR:lemma:HN_L1.2.4;widetilde} and \ref{IR:lemma:HN_L1.2.6;widetilde} are the corresponding versions of Lemmas \ref{IR:lemma:HN_L1.2.4} and \ref{IR:lemma:HN_L1.2.6}, respectively, for $\widetilde{y}^{\vec{A}}(E;X)$ instead of $y^{\vec{A}}(E;X)$.

\begin{lemma}\label{IR:lemma:HN_L1.2.4;widetilde}
Let $E \in \mathcal{S}(\varepsilon_{+},\varepsilon_{-},\vec{A},\vec{r},(S,\mathscr{A},\mu))$, $b \in (0,\infty)$ and suppose that $\varepsilon_{+},\varepsilon_{-}>0$. Let $\vec{p} \in (0,\infty]^{\ell}$ satisfy $\varepsilon_{+}>\mathrm{tr}(\vec{A})\bcdot(\vec{r}^{-1}-\vec{p}^{-1})$.
Define the sublinear operator
\[
T_{\vec{p}}^{\vec{A}}:L_{0}(S)^{X^{*} \times \N \times \Z^{d}} \longra L_{0}(S;[0,\infty])^{X^{*} \times \N \times \Z^{d}}, \qquad
(s_{x^{*},n,k})_{(x^{*},n,k)} \mapsto (t_{x^{*},n,k})_{(x^{*},n,k)},
\]
by
\[
t_{x^{*},n,k} := 2^{n\mathrm{tr}(\vec{A})\bcdot \vec{p}^{-1}}\normB{\sum_{m,l}|s_{x^{*},m,l}|\chi^{\vec{A}}_{m,l}}_{L_{\vec{p},\mathpzc{d}}}
\]
and the sum is taken over all indices $(m,l) \in \N \times \Z^{d}$ such that $Q^{\vec{A}}_{m,l} \subset Q^{\vec{A}}_{n,k}(b)$ and $m \geq n$.
Then $T^{\vec{A}}_{\vec{p}}$ restricts to a bounded sublinear operator on $\widetilde{y}^{\vec{A}}(E)$.
\end{lemma}
\begin{proof}
Let $\delta \in (0,\infty)$ be as in the proof of Lemma~\ref{IR:lemma:HN_L1.2.4}.
Let $(s_{n,k})_{(x^{*},n,k)} \in \widetilde{y}^{\vec{A}}(E)$ and $(t_{x^{*},n,k})_{(n,k)} = T^{\vec{A}}_{\vec{p}}[(s_{x^{*},n,k})_{(x^{*},n,k)}] \in L_{0}(S;[0,\infty])^{X^{*} \times \N \times \Z^{d}}$.
Define
\[
g_{x^{*},m} := \sum_{l \in \Z^{d}}s_{x^{*},m,l}\chi^{\vec{A}}_{m,l} \in L_{0}(S), \qquad m \in \N.
\]
Then $(g_{x^{*},m})_{(x^{*},m)} \in \mathscr{F}_{\mathrm{M}}(X^{*};E)$ with $\norm{(g_{x^{*},m})_{(x^{*},m)}}_{\mathscr{F}_{\mathrm{M}}(X^{*};E)} = \norm{(s_{x^{*},n,k})_{(x^{*},n,k)}}_{\widetilde{y}^{\vec{A}}(E)}$.
So there exists $(g_{m})_{m} \in E_{+}$ with $\norm{(g_{m})_{m}} \leq 2\norm{(s_{x^{*},n,k})_{(x^{*},n,k)}}_{\widetilde{y}^{\vec{A}}(E)}$ such that $|g_{x^{*},m}| \leq \norm{x^{*}}g_{m}$.
By \eqref{IR:eq:lemma:HN_L1.2.4;proof;5} from the proof of Lemma~\ref{IR:lemma:HN_L1.2.4},
\begin{align*}
t_{x^{*},n,k}\chi^{\vec{A}}_{n,k}
&\lesssim_{b,\vec{A},\vec{r}} \sum_{m=n}^{\infty}2^{(m-n)((\varepsilon_{+}-\delta))} M^{\vec{A}}_{\vec{r}}(g_{x^{*},m}) \\
&\leq \norm{x^{*}}\sum_{m=n}^{\infty}2^{(m-n)((\varepsilon_{+}-\delta))} M^{\vec{A}}_{\vec{r}}(g_{m}).
\end{align*}
As \eqref{IR:eq:lemma:HN_L1.2.4;proof;6} in proof of Lemma~\ref{IR:lemma:HN_L1.2.4}, we find that $(t_{x^{*},n,k})_{(x^{*},n,k)} \in \widetilde{y}^{\vec{A}}(E;X)$ with
\[
\norm{(t_{x^{*},n,k})_{(x^{*},n,k)}}_{\widetilde{y}^{\vec{A}}(E;X)}
\lesssim \norm{(g_{m})_{m}} \leq 2\norm{(s_{x^{*},n,k})_{(x^{*},n,k)}}_{\widetilde{y}^{\vec{A}}(E)}. \qedhere
\]
\end{proof}

\begin{lemma}\label{IR:lemma:HN_L1.2.6;widetilde}
Let $E \in \mathcal{S}(\varepsilon_{+},\varepsilon_{-},\vec{A},\vec{r},(S,\mathscr{A},\mu))$, $b \in (0,\infty)$ and $\lambda \in (\varepsilon_{-},\infty)$.
Define the sublinear operator
\[
T_{\lambda}:L_{0}(S)^{X^{*}\times\N \times \Z^{d}} \longra L_{0}(S;[0,\infty])^{X^{*} \times \N \times \Z^{d}}, \qquad
(s_{x,^{*},n,k})_{(x^{*},n,k)} \mapsto (t_{x^{*},n,k})_{(x^{*},n,k)},
\]
by
\[
t_{x^{*},n,k} := \sum_{m,l}2^{\lambda(n-m)}|s_{x^{*},m,l}|,
\]
the sum being taken over all indices $(m,l) \in \N \times \Z^{d}$ such that $Q^{\vec{A}}_{m,l}(b) \supset Q^{\vec{A}}_{n,k}$ and $m < n$.
Then $T_{\lambda}$ restricts to a bounded sublinear operator on $\widetilde{y}^{\vec{A}}(E;X)$.
\end{lemma}
\begin{proof}
This can be proved in the same way as \cite[Lemma~1.2.6]{Hedberg&Netrusov2007}.
\end{proof}

\begin{lemma}\label{IR:lemma:one-sided_estimate_local_means}
Let $E \in \mathcal{S}(\varepsilon_{+},\varepsilon_{-},\vec{A},\vec{1},(S,\mathscr{A},\mu))$ and let $k \in L_{1,\mathrm{c}}(\R^{d})$ fulfill the Tauberian condition
\[
|\hat{k}(\xi)| > 0, \quad\quad \xi \in \R^{d}, \frac{\epsilon}{2} < \rho_{\vec{A}}(\xi) < 2\epsilon,
\]
for some $\epsilon \in (0,\infty)$.
Let $\psi \in \mathcal{S}(\R^{d})$ be such that $\supp \hat{\psi} \subset \{ \xi : \epsilon \leq \rho_{\vec{A}}(\xi) \leq B\}$ for some $B \in (\epsilon,\infty)$.
Define $(k_{n})_{n \in \N}$ and $(\psi_{n})_{n \in \N}$ by $k_{n}:=2^{n\mathrm{tr}(\vec{A}^{\oplus})}k(\vec{A}_{2^{n}}\,\cdot\,)$ and $\psi_{n}:=2^{n\mathrm{tr}(\vec{A}^{\oplus})}\psi(\vec{A}_{2^{n}}\,\cdot\,)$.
Then
\[
\norm{(\psi_{n}*f_{n})_{n}}_{E(X)} \lesssim \norm{(k_{n}*f_{n})_{n}}_{E(X)},\qquad f \in L_{0}(S;L_{1,\loc}(\R^{d};X)).
\]
\end{lemma}
\begin{proof}
Pick $\eta \in C^{\infty}_{c}(\R^{d})$ with $\supp \eta \subset B^{\vec{A}}(0,2\epsilon)$ and $\eta(\xi)=1$ for $\rho_{\vec{A}}(\xi) \leq~\frac{3\epsilon}{2}$.
Define $m \in \mathcal{S}(\R^{d})$ by $m(\xi):= [\eta(\xi)-\eta(\vec{A}_{2}\xi)]\hat{k}(\xi)^{-1}$ if $\frac{\epsilon}{2} < \rho_{\vec{A}}(\xi) < 2\epsilon$ and $m(\xi):=0$ otherwise; note that this gives a well-defined Schwartz function on $\R^{d}$ because $\eta-\eta(\vec{A}_{2}\,\cdot\,)$ is a smooth function supported in the set $\{ \xi : \frac{\epsilon}{2} < \rho_{\vec{A}}(\xi) < 2\epsilon\}$ on which the function $\hat{k} \in C^{\infty}_{L_{\infty}}(\R^{d})$ does not vanish.
Define $(m_{n})_{n \in \N}$ by $m_{n}:=m(\vec{A}_{2^{-n}}\,\cdot\,)$.
Then, by construction,
\[
\sum_{l=n}^{n+N}m_{l}\hat{k}_{l}(\xi) = \eta(\vec{A}_{2^{-(n+N)}}\xi)-\eta(\vec{A}_{2^{-n+1}}\xi) = 1
\]
for $2^{n}\epsilon \leq \rho_{\vec{A}}(\xi) \leq 2^{n+N-1}3\epsilon$, $n \in \N$, $N \in \N$.
Since $\supp \hat{\psi}_{n} \subset \{ \xi : 2^{n}\epsilon \leq \rho_{\vec{A}}(\xi) < 2^{n}B \}$ for every $n \in \N$, there thus exists $N \in \N$ such that $\sum_{l=n}^{n+N}m_{l}\hat{k}_{l} \equiv 1$ on $\supp \hat{\varphi}_{n}$ for all $n \in \N$.
For each $n \in \N$ we consequently have
\[
\psi_{n} = \psi_{n} * \left(\sum_{l=n}^{n+N}\check{m}_{l}*k_{l}\right) =
\sum_{l=n}^{n+N}\psi_{n}*\check{m}_{l}*k_{l} = \sum_{l=0}^{N}\psi_{n}*\check{m}_{n+l}*k_{n+l}.
\]
As $\psi,m \in \mathcal{S}(\R^{d})$, we obtain the pointwise estimate
\[
\norm{\psi_{n}*f}_{X} \leq \sum_{l=0}^{N}\norm{\psi_{n}*\check{m}_{n+l}*k_{n+l}*f}_{X} \lesssim \sum_{l=0}^{N}M^{\vec{A}}(M^{\vec{A}}(\norm{k_{n+l}*f}_{X})).
\]
It follows that
\begin{align*}
\norm{(\psi_{n}*f)_{n}}_{E(X)}
&\lesssim \sum_{l=0}^{N}\normb{\big(M^{\vec{A}}(M^{\vec{A}}(\norm{k_{n+l}*f}_{X})\big)_{n}}_{E} \\
&\lesssim \sum_{l=0}^{N}\norm{(k_{n+l}*f)_{n}}_{E(X)}
\lesssim \sum_{l=0}^{N}2^{-\varepsilon_{+}l}\norm{(k_{n}*f)_{n}}_{E(X)} \\
&\lesssim \norm{(k_{n}*f)_{n}}_{E(X)}. \qedhere
\end{align*}

\end{proof}

\subsection{Proofs of the results in Section~\ref{IR:subsec:sec:Diff_norms;results}}

\begin{proof}[Proof of Theorem~\ref{IR:thm:HN_T1.1.14_YL}]
\eqref{IR:it:thm:HN_T1.1.14_YL;(i)} $\ra$ \eqref{IR:it:thm:HN_T1.1.14_YL;(iv)}:
Fix $\omega \in C^{\infty}_{c}((-1,2)^{d})$ with the property that
\[
\sum_{k \in \Z^{d}}\omega(x-k) = 1, \qquad x \in \R^{d}.
\]
Let $(f_{n})_{n}$ be as in Definition~\ref{IR:def:YL} with $\norm{(f_{n})_{n}}_{E(X)} \leq 2\norm{f}_{YL^{\vec{A}}(E;X)}$.
For each $(n,k) \in \N \times \Z^{d}$, we put
\[
\tilde{a}_{n,k} := \omega(\vec{A}_{2^{n}}(\,\cdot\,-\vec{A}_{2^{-n}}k))f_{n},
\qquad s_{n,k} := \norm{\tilde{a}_{n,k}(\vec{A}_{2^{-n}}\,\cdot\,)}_{C^{M}_{b}(\R^{d};X)},
\]
and
\[
a_{n,k} := \frac{\tilde{a}_{n,k}}{s_{n,k}}\,1_{\{s_{n,k} \neq 0\}}.
\]
Note that
\begin{align*}
|s_{n,k}|
&= \norm{\tilde{a}_{n,k}(\vec{A}_{2^{-n}}\,\cdot\,)}_{C^{M}_{b}(\R^{d};X)}
= \norm{\omega(\,\cdot\,-k)f_{n}(\vec{A}_{2^{-n}}\,\cdot\,)}_{C^{M}_{b}(\R^{d};X)} \\
&\lesssim \norm{\omega(\,\cdot\,-k)}_{C^{M}_{b}(\R^{d})}
\norm{f_{n}(\vec{A}_{2^{-n}}\,\cdot\,)}_{C^{M}_{b}([-1,2]^{d}+k;X)} \\
&\lesssim \sup_{|\alpha| \leq M}\sup_{y \in [-1,2]^{d}+k}\norm{D^{\alpha}[f_{n}(\vec{A}_{2^{-n}}\,\cdot\,)](y)}_{X}
\end{align*}
Given $x \in Q^{\vec{A}}_{n,k}$ and $\tilde{x}=\vec{A}_{2^{n}}x \in [0,1)^{d}+k$, for $y \in [-1,2]^{d}+k$ we can write $y=\tilde{x}+z$ with
\[
z=y-\tilde{x} = (y-k)-(\tilde{x}-k) \in [-1,2]^{d}-[0,1)^{d}, \quad \text{so, in particular,}\: \rho_{\vec{A}}(z) \leq C_{d}.
\]
Combining the above and subsequently applying Lemma~\ref{IR:appendix:lemma:Peetre-Feffeman-Stein} to $f_{n}(\vec{A}_{2^{-n}}\,\cdot\,)$, whose spectrum satisfies $\supp\mathcal{F}[f_{n}(\vec{A}_{2^{-n}}\,\cdot\,)] \subset B^{\vec{A}}(0,2)$, we find
\begin{align*}
|s_{n,k}|
&\lesssim \sup_{|\alpha| \leq M}\sup_{\rho_{\vec{A}}(z) \leq C_{d}}\norm{D^{\alpha}[f_{n}(\vec{A}_{2^{-n}}\,\cdot\,)](\tilde{x}+z)}_{X} \\
&\lesssim M^{\vec{A}}_{\vec{r}}[\norm{f_{n}(\vec{A}_{2^{-n}}\,\cdot\,)}_{X}]\,(\vec{A}_{2^{n}}x)
= M^{\vec{A}}_{\vec{r}}(\norm{f_{n}}_{X})(x)
\end{align*}
for $x \in Q^{\vec{A}}_{n,k}$.
Therefore, $(s_{n,k})_{(n,k)} \in y^{\vec{A}}(E)$ with
\[
\norm{(s_{n,k})_{(n,k)}}_{y^{\vec{A}}(E)} \lesssim
\normb{(M^{\vec{A}}_{\vec{r}}(\norm{f_{n}}_{X}))_{n}}_{E} \lesssim \norm{(f_{n})_{n}}_{E(X)}
\leq 2\norm{f}_{YL^{\vec{A}}(E;X)}.
\]
Finally, the convergence \eqref{IR:eq:thm:HN_T1.1.14_YL;(iv);convergence} follows from Corollary~\ref{IR:cor:lemma:HN_L1.2.4;HN_C1.2.5} and the observation that
\[
f = \sum_{n=0}^{\infty}f_{n} = \sum_{n=0}^{\infty}\sum_{k \in \Z^{d}}s_{n,k}a_{n,k}
\quad \text{in} \quad L_{0}(S;L_{\vec{r},\mathpzc{d},\loc}(\R^{d};X)).
\]

\eqref{IR:it:thm:HN_T1.1.14_YL;(iv)} $\ra$ \eqref{IR:it:thm:HN_T1.1.14_YL;(i)}: Set $g_{n} := \sum_{k \in \Z^{d}}|s_{n,k}|\chi^{\vec{A}}_{n,k}$ for $n \in \N$. For $n \in \Z_{<0}$, set $f_{n}:=0$ and $g_{n}:=0$.
Pick $\kappa \in (0,1]$ such that $E$ has a $\kappa$-norm.
Pick $\varepsilon \in (0,\lambda^{\vec{A}}_{\min})$ such that $(\lambda^{\vec{A}}_{\min}-\varepsilon)M > \varepsilon_{-}$.
Pick $\lambda \in (0,\infty)$ such that $\mathrm{tr}(\vec{A}^{\oplus})/\lambda < \vec{r}_{\min} \wedge 1$.
Pick $\psi = (\psi_{n})_{n \in \N} \in \Phi^{\vec{A}}(\R^{d})$ such that
\[
\supp \hat{\psi}_{0} \subset B^{\vec{A}}(0,2),\qquad
\supp \hat{\psi}_{n} \subset B^{\vec{A}}(0,2^{n+1}) \setminus B^{\vec{A}}(0,2^{n-1}), \quad n \geq 1,
\]
and set $\Psi_{n}:= 2^{n\mathrm{tr}(\vec{A}^{\oplus})}\psi_{0}(\vec{A}_{2^{n}}\,\cdot\,)$ for each $n \in \N$.
Note that
\[
a_{n,k}*\Psi_{n} = [b_{n,k}*\Psi](\vec{A}_{2^{n}}\,\cdot\,-k)
\]
and
\[
a_{n,k}*\psi_{m} = [b_{n,k}*\psi_{m-n}](\vec{A}_{2^{n}}\,\cdot\,-k), \qquad n < m.
\]
An application of Lemma~\ref{IR:lemma:HN_L1.2.8} thus yields that
\begin{equation}\label{IR:it:thm:HN_T1.1.14_YL;proof;(iv)->(i);1}
\norm{a_{n,k}*\Psi_{n}\,(x)}_{X} \lesssim \frac{1}{(1+2^{n}\rho_{\vec{A}}(x-\vec{A}_{2^{-n}}k))^{\lambda}}
\end{equation}
and
\begin{equation}\label{IR:it:thm:HN_T1.1.14_YL;proof;(iv)->(i);2}
\norm{a_{n,k}*\psi_{m}\,(x)}_{X} \lesssim \frac{2^{-(m-n)(\lambda^{\vec{A}}_{\min}-\varepsilon)M}}{(1+2^{n}\rho_{\vec{A}}(x-\vec{A}_{2^{-n}}k))^{\lambda}}, \qquad n < m.
\end{equation}
Now put
\[
\tilde{a}_{n,k,m} := \left\{\begin{array}{ll}
a_{n,k} * \Psi_{n}, & n=m,\\
a_{n,k} * \psi_{m}, & n < m.
\end{array}\right.
\]


Let $L_{\mathrm{M}}(\R^{d};X)$ denote the Fr\'echet space of all equivalence classes of strongly measurable $X$-valued functions on $\R^{d}$ that are of polynomial growth; this space can for instance be described as
\[
L_{\mathrm{M}}(\R^{d};X) := \big\{ f \in L_{0}(\R^{d};X) : \forall \phi \in \mathcal{S}(\R^{d}), \phi f \in L_{\infty}(\R^{d};X) \big\}.
\]
Using Lemma~\ref{IR:lemma:estimate_yA_fixed_n} together with the support condition of the $a_{n,k}$ and $\norm{a_{n,k}}_{L_{\infty}(\R^{d};X)} \leq 1$, it can be shown that the series $\sum_{k \in \Z^{d}}s_{n,k}a_{n,k}$ converges in $L_{0}(S;L_{\mathrm{M}}(\R^{d};X))$.
Since $L_{\mathrm{M}}(\R^{d};X) \hookrightarrow \mathcal{S}'(\R^{d};X)$ and convolution gives rise to a separately continuous bilinear mapping $\mathcal{S} \times \mathcal{S}' \to \mathscr{O}_{\mathrm{M}}$, it follows that
\begin{equation}\label{IR:it:thm:HN_T1.1.14_YL;proof;(iv)->(i);5}
f_{n,m} := \sum_{k \in \Z^{d}}s_{n,k}\tilde{a}_{n,k,m} =
\Big( \sum_{k \in \Z^{d}}s_{n,k}a_{n,k} \Big) *
\left\{\begin{array}{ll}
\Psi_{n}, & n=m,\\
\psi_{m}, & n < m,
\end{array}\right.
\quad \text{in} \quad L_{0}(S;\mathscr{O}_{\mathrm{M}}(\R^{d};X))
\end{equation}
for each $n,m \in \N$ with $m \geq n$.

It will be convenient to define
\[
f^{+}_{n,m} := \sum_{k \in \Z^{d}}|s_{n,k}|\,\norm{\tilde{a}_{n,k,m}}_{X}, \qquad n,m \in \N, m \geq n.
\]
By a combination of \eqref{IR:it:thm:HN_T1.1.14_YL;proof;(iv)->(i);1}, \eqref{IR:it:thm:HN_T1.1.14_YL;proof;(iv)->(i);2} and Lemma~\ref{IR:lemma:HN_L1.2.7},
\[
f^{+}_{m-l,m}
\nonumber \\
\lesssim 2^{-l(\lambda^{\vec{A}}_{\min}-\varepsilon)M}M^{\vec{A}}_{\vec{r}}(g_{m-l}), \qquad m,l \in \N, m \geq l.
\]
From this it follows that
\begin{align}
\norm{(f^{+}_{m-l,m})_{m \geq l}}_{E(\N_{\geq l})}
&\lesssim 2^{-l(\lambda^{\vec{A}}_{\min}-\varepsilon)M}\norm{(M^{\vec{A}}_{\vec{r}}(g_{m-l}))_{m \geq l}}_{E(\N_{\geq l})} \nonumber\\
&= 2^{-l(\lambda^{\vec{A}}_{\min}-\varepsilon)M}\normb{(S_{-})^{l}(M^{\vec{A}}_{\vec{r}}(g_{m}))_{m \in \N}}_{E} \nonumber\\
&\lesssim 2^{-l((\lambda^{\vec{A}}_{\min}-\varepsilon)M-\varepsilon_{-})}\norm{(g_{m})_{m \in \N}}_{E} \nonumber\\
&= 2^{-l((\lambda^{\vec{A}}_{\min}-\varepsilon)M-\varepsilon_{-})}\norm{(s_{n,k})_{(n,k)}}_{y^{\vec{A}}(E)}. \label{IR:it:thm:HN_T1.1.14_YL;proof;(iv)->(i);3}
\end{align}
Therefore, by Lemma~\ref{IR:lemma:HN_L1.1.4} and the assumption $(\lambda^{\vec{A}}_{\min}-\varepsilon)M>\varepsilon_{-}$,
\[
\sum_{l=0}^{\infty}\sum_{m=l}^{\infty} f^{+}_{m-l,m} =
\sum_{l=0}^{\infty}\sum_{m=l}^{\infty}\sum_{k \in \Z^{d}}|s_{m-l,k}|\,\norm{\tilde{a}_{m-l,k,m}}_{X}
\]
belongs to $E^{\vec{A}}_{\otimes}[B^{\vec{r},w_{\vec{A},\vec{r}}}_{\vec{A}}] \hookrightarrow L_{0}(S;L_{\vec{r},\mathpzc{d},\mathrm{loc}}(\R^{d}))$.
By Lebesgue domination this implies that $\sum_{l=0}^{\infty}\sum_{m=l}^{\infty}\sum_{k \in \Z^{d}}s_{m-l,k}\tilde{a}_{m-l,k,m}$ converges unconditionally in the space $L_{0}(S;L_{\vec{r},\mathpzc{d},\mathrm{loc}}(\R^{d};X))$. In particular,
\[
\sum_{l=0}^{\infty}\sum_{m=l}^{\infty}\sum_{k \in \Z^{d}}s_{m-l,k}\tilde{a}_{m-l,k,m} =
\sum_{n=0}^{\infty}\sum_{k \in \Z^{d}}s_{n,k}\sum_{m=n}^{\infty}\tilde{a}_{n,k,m}
\quad \text{in} \quad L_{0}(S;L_{\vec{r},\mathpzc{d},\mathrm{loc}}(\R^{d};X)).
\]
Since
\[
a_{n,k} = \lim_{N \to \infty}\Psi_{N}*a_{n,k} = \lim_{N \to \infty}\sum_{m=n}^{N}\tilde{a}_{n,k,m}
\quad \text{in} \quad L_{0}(S;L_{1}(\R^{d};X)),
\]
and since $f$ has the representation \eqref{IR:eq:thm:HN_T1.1.14_YL;(iv);convergence}, it follows that
\[
f = \sum_{l=0}^{\infty}\sum_{m=l}^{\infty}\sum_{k \in \Z^{d}}s_{m-l,k}\tilde{a}_{m-l,k,m}
\quad \text{in} \quad L_{0}(S;L_{\vec{r} \wedge \vec{p},\mathpzc{d},\mathrm{loc}}(\R^{d};X)).
\]
Combining the latter with \eqref{IR:it:thm:HN_T1.1.14_YL;proof;(iv)->(i);5}, we find
\begin{equation}\label{IR:it:thm:HN_T1.1.14_YL;proof;(iv)->(i);4}
f= \sum_{l=0}^{\infty}\sum_{m=l}^{\infty}f_{m-l,m}\quad \text{in} \quad L_{0}(S;L_{\vec{r} \wedge \vec{p},\mathpzc{d},\mathrm{loc}}(\R^{d};X)).
\end{equation}

Note that
\[
\norm{(f_{m-l,m})_{m \geq l}}_{E(\N_{\geq l};X)} \lesssim
2^{-l((\lambda^{\vec{A}}_{\min}-\varepsilon)M-\varepsilon_{-})}\norm{(s_{n,k})_{(n,k)}}_{y^{\vec{A}}(E)}
\]
by \eqref{IR:it:thm:HN_T1.1.14_YL;proof;(iv)->(i);3}.
Since
\[
\supp \hat{f}_{m-l,m} \subset
\left\{\begin{array}{ll}
\supp\hat{\Psi}_{m}, & l = 0,\\
\supp \hat{\psi}_{m}, & l \geq 1,\\
\end{array}\right.
\subset \bar{B}^{\vec{A}}(0,2^{m+1}), \qquad m \geq l,
\]
it follows that (see Remark~\ref{IR:rmk:indep_r_YL})
\[
F_{l}:= \sum_{m=l}^{\infty}f_{m-l,m} \qquad \text{in} \quad L_{0}(S;L_{\vec{r},\mathpzc{d},\mathrm{loc}}(\R^{d};X)),
\]
defines an element of $YL^{\vec{A}}(E;X)$ with
\[
\norm{F_{l}}_{YL^{\vec{A}}(E;X)} \lesssim 2^{-l((\lambda^{\vec{A}}_{\min}-\varepsilon)M-\varepsilon_{-})}\norm{(s_{n,k})_{(n,k)}}_{y^{\vec{A}}(E)}.
\]
As $(\lambda^{\vec{A}}_{\min}-\varepsilon)M > \varepsilon_{-}$, we find that $F := \sum_{l=0}F_{l} \in YL^{\vec{A}}(E;X)$ with
\[
\norm{F}_{YL^{\vec{A}}(E;X)} \lesssim \norm{(s_{n,k})_{(n,k)}}_{y^{\vec{A}}(E)}.
\]
But $f=F$ in view of \eqref{IR:it:thm:HN_T1.1.14_YL;proof;(iv)->(i);4} and $YL^{\vec{A}}(E;X) \hookrightarrow L_{0}(S;L_{\vec{r},\mathpzc{d},\mathrm{loc}}(\R^{d};X))$ (see Remark~\ref{IR:rmk:indep_r_YL}), yielding the desired result.

%

%

\eqref{IR:it:thm:HN_T1.1.14_YL;(iv)} $\ra$ \eqref{IR:it:thm:HN_T1.1.14_YL;(ii)}:
We will write down the proof in such a way that the proof of Proposition~\ref{IR:prop:HN_T1.1.14;difference+shift_one-sided} only requires a slight modification.
Combining the estimate corresponding to \eqref{IR:it:thm:HN_T1.1.14_YL;(iv)} $\ra$ \eqref{IR:it:thm:HN_T1.1.14_YL;(i)} with $YL^{\vec{A}}(E;X) \hookrightarrow E_{0}(X)$ (see \eqref{IR:eq:lemma:HN_L1.1.4;1}), we find
\[
\norm{f}_{E_{0}(X)} \lesssim \norm{(s_{n,k})_{(n,k)}}_{y^{\vec{A}}(E)}.
\]
So let us focus on the remaining part of the required inequality.
To this end, fix $c \in \R$ and choose $R \in [1,\infty)$ such that
\[
\rho_{\vec{A}}(tz) \leq R\rho_{\vec{A}}(z), \qquad z \in \R^{d}, t \in [0,|c|+M].
\]
Put
\[
d^{\vec{A},\vec{p}}_{M,c,n}(f) := 2^{n\mathrm{tr}(\vec{A})\bcdot\vec{p}^{-1}}\normb{z \mapsto L_{cz}\Delta_{z}^{M}f}_{L_{\vec{p},\mathpzc{d}}(B^{\vec{A}}(0,2^{-n});X)}, \qquad n \in \N.
\]
Now let $f$ has a representation as in \eqref{IR:it:thm:HN_T1.1.14_YL;(iv)} and write $h_{n}:=\sum_{k \in \Z^{d}}s_{n,k}a_{n,k}$.
Then
\begin{align}
d^{\vec{A},\vec{p}}_{M,c,n}(f)(x)
&\lesssim 2^{n\mathrm{tr}(\vec{A})\bcdot\vec{p}^{-1}}\normB{z \mapsto \sum_{m=0}^{n-1}\norm{L_{cz}\Delta_{z}^{M}h_{m}(x)}_{X}}_{L_{\vec{p},\mathpzc{d}}(B^{\vec{A}}(0,2^{-n}))} \nonumber \\
& \quad+\:\: 2^{n\mathrm{tr}(\vec{A})\bcdot\vec{p}^{-1}}\normB{z \mapsto \sum_{m=n}^{\infty}\norm{L_{cz}\Delta_{z}^{M}h_{m}(x)}_{X}}_{L_{\vec{p},\mathpzc{d}}(B^{\vec{A}}(0,2^{-n}))}.
\label{IR:it:thm:HN_T1.1.14_YL;proof;(iv)->(ii);1}
\end{align}

We use the identity
\[
L_{cz}\Delta^{M}_{z}h_{m}(x) = \sum_{l=0}^{M}(-1)^{M-l}{M \choose l}h_{j}(x+(c+l)z)
\]
to estimate the second term in \eqref{IR:it:thm:HN_T1.1.14_YL;proof;(iv)->(ii);1} as follows
\begin{align*}
& 2^{n\mathrm{tr}(\vec{A})\bcdot\vec{p}^{-1}}\normB{z \mapsto \sum_{m=n}^{\infty}\norm{L_{cz}\Delta_{z}^{M}h_{m}(x)}_{X}}_{L_{\vec{p},\mathpzc{d}}(B^{\vec{A}}(0,2^{-n}))} \\
&\qquad \lesssim \sum_{l=0}^{M}2^{n\mathrm{tr}(\vec{A})\bcdot\vec{p}^{-1}}\normB{z \mapsto \sum_{m=n}^{\infty}\norm{h_{m}(x+(c+l)z))}_{X}}_{L_{\vec{p},\mathpzc{d}}(B^{\vec{A}}(0,2^{-n}))} \\
&\qquad \lesssim 2^{n\mathrm{tr}(\vec{A})\bcdot\vec{p}^{-1}}\normB{ \sum_{m=n}^{\infty}\norm{h_{m}}_{X}}_{L_{\vec{p},\mathpzc{d}}(B^{\vec{A}}(x,R2^{-n}))} \\
&\qquad \lesssim 2^{n\mathrm{tr}(\vec{A})\bcdot\vec{p}^{-1}}\normB{ \sum_{m=n}^{\infty}\sum_{k \in \Z^{d}}\norm{s_{m,k}}_{X}1_{Q^{\vec{A}}_{m,k}(3)}}_{L_{\vec{p},\mathpzc{d}}(B^{\vec{A}}(x,R2^{-n}))} \\
&\qquad \lesssim 2^{n\mathrm{tr}(\vec{A})\bcdot\vec{p}^{-1}}\normB{ \sum_{m,l}\norm{s_{m,l}}_{X}1_{Q^{\vec{A}}_{m,l}(3)}}_{L_{\vec{p},\mathpzc{d}}},
\end{align*}
where the last sum is taken over all $(m,l)$ such that $Q^{\vec{A}}_{m,l}(3)$ intersects $B^{\vec{A}}(x,R2^{-n}))$ and $m \geq n$.
From this it follows that
\begin{align}
& 2^{n\mathrm{tr}(\vec{A})\bcdot\vec{p}^{-1}}\normB{z \mapsto \sum_{m=n}^{\infty}\norm{L_{cz}\Delta_{z}^{M}h_{m}}_{X}}_{L_{\vec{p},\mathpzc{d}}(B^{\vec{A}}(0,2^{-n}))} \nonumber \\
&\qquad \lesssim \sum_{k \in \Z^{d}}2^{n\mathrm{tr}(\vec{A})\bcdot\vec{p}^{-1}}\normB{ \sum_{m,l}\norm{s_{m,l}}_{X}1_{Q^{\vec{A}}_{m,l}(3)}}_{L_{\vec{p},\mathpzc{d}}}1_{Q^{\vec{A}}_{n,k}(3R)},
\label{IR:it:thm:HN_T1.1.14_YL;proof;(iv)->(ii);2}
\end{align}
where the sum is taken over all $(m,l)$ such that $Q^{\vec{A}}_{m,l}(3) \subset Q^{\vec{A}}_{n,k}(3R)$ and $m \geq n$.

In order to estimate the first term in \eqref{IR:it:thm:HN_T1.1.14_YL;proof;(iv)->(ii);1}, note that
\[
\Delta^{M}_{z}h_{m}(x) = \int_{[0,1]^{M}}D^{M}h_{m}(x+(t_{1}+\ldots+t_{M})z)(z,\ldots,z)\,d(t_{1},\ldots,t_{M})
\]
and thus that
\begin{align*}
\norm{\Delta^{M}_{z}h_{m}(x)}_{X}
&\leq \sup_{t \in [0,M]}\norm{D^{M}h_{m}(x+tz)(z,\ldots,z)}_{X} \\
&= \sup_{t \in [0,M]}\normb{D^{M}[h_{m} \circ \vec{A}_{2^{-m}}](\vec{A}_{2^{m}}x+t\vec{A}_{2^{m}}z)(\vec{A}_{2^{m}}z,\ldots,\vec{A}_{2^{m}}z)}_{X} \\
&\lesssim \sup_{t \in [0,M]}\sup_{|\alpha| \leq M}\normb{D^{\alpha}[h_{m} \circ \vec{A}_{2^{-m}}](\vec{A}_{2^{m}}x+t\vec{A}_{2^{m}}z)}_{X}|\vec{A}_{2^{m}}z|^{M},
\end{align*}
from which it follows that
\begin{align*}
\norm{L_{cz}\Delta^{M}_{z}h_{m}(x)}_{X}
&\lesssim \sup_{t \in [0,M]}\sup_{|\alpha| \leq M}\normb{D^{\alpha}[h_{m} \circ \vec{A}_{2^{-m}}](\vec{A}_{2^{m}}x+(c+t)\vec{A}_{2^{m}}z)}_{X}|\vec{A}_{2^{m}}z|^{M} \\
&\leq \sup_{y \in B^{\vec{A}}(0,R\rho_{\vec{A}}(z))}\sup_{|\alpha| \leq M}\normb{D^{\alpha}[h_{m} \circ \vec{A}_{2^{-m}}](\vec{A}_{2^{m}}[x+y])}_{X}|\vec{A}_{2^{m}}z|^{M}.
\end{align*}
Given $\varepsilon \in (0,\lambda^{\vec{A}}_{\min})$, for  $m \in \{0,\ldots,n-1\}$ and $z \in B^{\vec{A}}(0,2^{-n})$ this gives
\begin{align*}
\norm{L_{cz}\Delta^{M}_{z}h_{m}(x)}_{X}
&\lesssim_{\varepsilon} \sup_{y \in B^{\vec{A}}(0,R2^{-n})}\sup_{|\alpha| \leq M}\normb{D^{\alpha}[h_{m} \circ \vec{A}_{2^{-m}}](\vec{A}_{2^{m}}[x+y])}_{X}\rho_{\vec{A}}(\vec{A}_{2^{m}}z)^{(\lambda^{\vec{A}}_{\min}-\varepsilon)M} \\
&\lesssim \sup_{y \in B^{\vec{A}}(0,R2^{-n})}\sup_{|\alpha| \leq M}\normb{D^{\alpha}[h_{m} \circ \vec{A}_{2^{-m}}](\vec{A}_{2^{m}}[x+y])}_{X}2^{(\lambda^{\vec{A}}_{\min}-\varepsilon)M(m-n)}.
\end{align*}
Since
\begin{align*}
\normb{D^{\alpha}[h_{m} \circ \vec{A}_{2^{-m}}](\vec{A}_{2^{m}}[x+y])}_{X}
&\leq \sum_{l \in \Z^{d}}\norm{s_{m,l}}_{X}1_{[-1,2]^{d}+l}(\vec{A}_{2^{m}}[x+y]) \\
&\leq \sum_{l \in \Z^{d}}\norm{s_{m,l}}_{X}1_{Q^{\vec{A}}_{m,l}(3)}(x+y),
\end{align*}
it follows that
\begin{align*}
& 2^{n\mathrm{tr}(\vec{A})\bcdot\vec{p}^{-1}}\normB{z \mapsto \sum_{m=0}^{n-1}\norm{L_{cz}\Delta_{z}^{M}h_{m}(x)}_{X}}_{L_{\vec{p},\mathpzc{d}}(B^{\vec{A}}(0,2^{-n}))} \\
&\qquad\qquad \lesssim_{\varepsilon} \sum_{m=0}^{n-1}\sup_{z \in B^{\vec{A}}(0,2^{-n})}\norm{L_{cz}\Delta_{z}^{M}h_{m}(x)}_{X}2^{(\lambda^{\vec{A}}_{\min}-\varepsilon)M(m-n)} \\
&\qquad\qquad \lesssim \sum_{m,l}2^{(\lambda^{\vec{A}}_{\min}-\varepsilon)M(m-n)}\norm{s_{m,l}}_{X},
\end{align*}
where the last sum is taken over all $(m,l)$ such that $Q^{\vec{A}}_{m,l}(3)$ intersects $B^{\vec{A}}(x,R2^{-n}))$ and $m < n$.
From this it follows that
\begin{equation}\label{IR:it:thm:HN_T1.1.14_YL;proof;(iv)->(ii);3}
2^{n\mathrm{tr}(\vec{A})\bcdot\vec{p}^{-1}}\normB{z \mapsto \sum_{m=0}^{n-1}\norm{L_{cz}\Delta_{z}^{M}h_{m}(x)}_{X}}_{L_{\vec{p},\mathpzc{d}}(B^{\vec{A}}(0,2^{-n}))}
\lesssim \sum_{m,l}\norm{s_{m,l}}_{X},
\end{equation}
where the last sum is taken over all $(m,l)$ such that $Q^{\vec{A}}_{m,l}(3R) \supset Q^{\vec{A}}_{n,k}(3)$  and $m < n$.

A combination of \eqref{IR:it:thm:HN_T1.1.14_YL;proof;(iv)->(ii);1},
\eqref{IR:it:thm:HN_T1.1.14_YL;proof;(iv)->(ii);2}, Lemma~\ref{IR:lemma:HN_L1.2.4},
\eqref{IR:it:thm:HN_T1.1.14_YL;proof;(iv)->(ii);3} and
Lemma~\ref{IR:lemma:HN_L1.2.6} give the desired result.

\end{proof}

\begin{proof}[Proof of Theorem~\ref{IR:thm:HN_T1.1.14_YL_widetilde}]
The chain of implications \eqref{IR:it:thm:HN_T1.1.14_YL_widetilde;(i)} $\lra$ \eqref{IR:it:thm:HN_T1.1.14_YL_widetilde;(iv)} $\ra$ \eqref{IR:it:thm:HN_T1.1.14_YL_widetilde;(ii)} with corresponding estimates for $f \in L_{0}(S;L_{\vec{r},\mathpzc{d}}(\R^{d};X))$ can be obtained in the same way as Theorem~\ref{IR:thm:HN_T1.1.14_YL} with some natural modifications; in particular, Lemmas \ref{IR:lemma:HN_L1.2.4} and \ref{IR:lemma:HN_L1.2.6} need to be replaced with Lemmas \ref{IR:lemma:HN_L1.2.4;widetilde} and \ref{IR:lemma:HN_L1.2.6;widetilde}, respectively.
Furthermore, \eqref{IR:it:thm:HN_T1.1.14_YL_widetilde;(iv)} $\ra$ \eqref{IR:it:thm:HN_T1.1.14_YL_widetilde;(iii)} can be done in the same way as \cite[Theorem~1.1.14]{Hedberg&Netrusov2007}, similarly to the implication \eqref{IR:it:thm:HN_T1.1.14_YL_widetilde;(iv)} $\ra$ \eqref{IR:it:thm:HN_T1.1.14_YL_widetilde;(ii)} (see the proof of \eqref{IR:it:thm:HN_T1.1.14_YL;(iv)} $\ra$ \eqref{IR:it:thm:HN_T1.1.14_YL;(ii)} in Theorem~\ref{IR:thm:HN_T1.1.14_YL}).

Fix $q \in (0,\infty)$ with $q \leq \vec{r}_{\min} \wedge \vec{p}_{\min}$\eqref{IR:it:thm:HN_T1.1.14_YL_widetilde;(ii)}$^{*}_{q}$ and let \eqref{IR:it:thm:HN_T1.1.14_YL_widetilde;(iii)}$^{*}_{q}$ be the statements \eqref{IR:it:thm:HN_T1.1.14_YL_widetilde;(ii)} and \eqref{IR:it:thm:HN_T1.1.14_YL_widetilde;(iii)}, respectively, in which $\vec{p}$ gets replaced by $\vec{q}:=(q,\ldots,q) \in (0,\infty)^{\ell}$.
Then, clearly, \eqref{IR:it:thm:HN_T1.1.14_YL_widetilde;(ii)} $\ra$ \eqref{IR:it:thm:HN_T1.1.14_YL_widetilde;(ii)}$^{*}_{q}$ and \eqref{IR:it:thm:HN_T1.1.14_YL_widetilde;(iii)} $\ra$ \eqref{IR:it:thm:HN_T1.1.14_YL_widetilde;(iii)}$^{*}_{q}$.

To finish this proof, it suffices to establish the implication \eqref{IR:it:thm:HN_T1.1.14_YL_widetilde;(v)} $\ra$ \eqref{IR:it:thm:HN_T1.1.14_YL_widetilde;(iii)}$^{*}_{q}$ for $f \in L_{0}(S;L_{\vec{r},\mathpzc{d}}(\R^{d};X))$ and the implications
\eqref{IR:it:thm:HN_T1.1.14_YL_widetilde;(ii)}$^{*}_{q}$ $\ra$ \eqref{IR:it:thm:HN_T1.1.14_YL_widetilde;(v)} and
\eqref{IR:it:thm:HN_T1.1.14_YL_widetilde;(iii)}$^{*}_{q}$ $\ra$  \eqref{IR:it:thm:HN_T1.1.14_YL_widetilde;(iv)} for $f$ of the form $f=\sum_{i \in I}1_{S_{i}} \otimes f^{[i]}$ with $(S_{i})_{i \in I} \subset \mathscr{A}$ a countable family of mutually disjoint sets and $(f^{[i]})_{i \in I} \in L_{\vec{r},\mathpzc{d},\loc}(\R^{d};X)$.

\eqref{IR:it:thm:HN_T1.1.14_YL_widetilde;(v)} $\ra$ \eqref{IR:it:thm:HN_T1.1.14_YL_widetilde;(iii)}$^{*}_{q}$: For this implication we just observe that, for $x \in Q^{\vec{A}}_{n,k}$ and $n \geq 1$,
\begin{align*}
\mathcal{E}^{\vec{A},\vec{q}}_{M,x^{*},n}(f))(x)
&\lesssim \overline{\mathcal{E}}_{M}(\ip{f}{x^{*}},Q^{\vec{A}}_{n,k}(3),L_{q}) \lesssim M^{\vec{A}}_{\vec{q}}(g_{x^{*},n})(x) \leq M^{\vec{A}}_{\vec{r}}(g_{x^{*},n})(x).
\end{align*}

\emph{\eqref{IR:it:thm:HN_T1.1.14_YL_widetilde;(ii)}$^{*}_{q}$ $\ra$ \eqref{IR:it:thm:HN_T1.1.14_YL_widetilde;(v)} for $f$ of the form $f=\sum_{i \in I}1_{S_{i}} \otimes f^{[i]}$ with $(S_{i})_{i \in I} \subset \mathscr{A}$ a countable family of mutually disjoint sets and $(f^{[i]})_{i \in I} \in L_{\vec{r},\mathpzc{d},\loc}(\R^{d};X)$:}
By Lemma~\ref{IR:lemma:HN_L1.2.1}, for each $i \in I$ and $(x^{*},n,k) \in X^{*} \times \N_{\geq1} \times \Z^{d}$ there exists a $\pi^{[i]}_{x^{*},n,k} \in \mathcal{P}^{d}_{M-1}$ such that
\[
\big|\ip{f^{[i]}}{x^{*}}-\pi^{[i]}_{x^{*},n,k}\big|\,1_{Q^{\vec{A}}_{n,k}(3)}
\lesssim d^{\vec{A},\vec{q}}_{M,x^{*},n}(f^{[i]}) +
\left(\fint_{Q^{\vec{A}}_{n,k}(6)}d^{\vec{A},\vec{q}}_{M,x^{*},n}(f^{[i]})(y)^{q}\,dy\right)^{1/q}.
\]
Defining $\pi_{x^{*},n,k} \in L_{0}(S;\mathcal{P}^{d}_{M-1})$ by $\pi_{x^{*},n,k}:=\sum_{i \in I}1_{S_{i}} \otimes \pi^{[i]}_{x^{*},n,k}$, we obtain
\[
\big|\ip{f}{x^{*}}-\pi_{x^{*},n,k}\big|\,1_{Q^{\vec{A}}_{n,k}(3)}
\lesssim d^{\vec{A},\vec{q}}_{M,x^{*},n}(f) + M^{\vec{A}}_{\vec{q}}(d^{\vec{A},\vec{q}}_{M,x^{*},n}(f)) \leq 2M^{\vec{A}}_{\vec{r}}(d^{\vec{A},\vec{q}}_{M,x^{*},n}(f)).
\]
Since
\[
\#\big\{ k \in \Z^{d} : x \in Q^{\vec{A}}_{n,k}(3) \big\} \lesssim 1, \qquad x \in \R^{d}, n \in \N,
\]
it follows that
\begin{align*}
&\normb{\{g_{x^{*},n}\}_{(x^{*},n) \in X^{*}\times\N_{\geq1}}}_{\mathscr{F}_{\mathrm{M}}(X^{*};E(\N_{1}))} \\
&\qquad\qquad\lesssim \normb{\{M^{\vec{A}}_{\vec{r}}[d^{\vec{A},\vec{p}}_{M,x^{*},n}(f)]\}_{(x^{*},n) \in X^{*}\times\N_{\geq1}}}_{\mathscr{F}_{\mathrm{M}}(X^{*};E(\N_{1}))} \\
&\qquad\qquad\lesssim \normb{\{d^{\vec{A},\vec{p}}_{M,x^{*},n}(f)\}_{(x^{*},n) \in X^{*}\times\N_{\geq1}}}_{\mathscr{F}_{\mathrm{M}}(X^{*};E(\N_{1}))}.
\end{align*}

\emph{\eqref{IR:it:thm:HN_T1.1.14_YL_widetilde;(iii)}$^{*}_{q}$ $\ra$ \eqref{IR:it:thm:HN_T1.1.14_YL_widetilde;(iv)} for $f$ of the form $f=\sum_{i \in I}1_{S_{i}} \otimes f^{[i]}$ with $(S_{i})_{i \in I} \subset \mathscr{A}$ a countable family of mutually disjoint sets and $(f^{[i]})_{i \in I} \in L_{\vec{r},\mathpzc{d},\loc}(\R^{d};X)$:}
Let $\omega \in C^{\infty}_{c}([-1,2]^{d})$ be such that
\[
\sum_{k \in \Z^{d}}\omega(x-k) = 1, \qquad x \in \R^{d},
\]
and put $\omega_{n,k}:= \omega(\vec{A}_{2^{n}}\,\cdot\,-k)$ and $Q^{\omega}_{n,k}:=\vec{A}_{2^{-n}}([-1,2]^{d}+k)$ for $(n,k) \in \N \times \Z^{d}$;
so $\supp(\omega_{n,k}) \subset Q^{\omega}_{n,k}$.
Define
\[
I_{n,k} := \big\{ l \in \Z^{d} :  Q^{\omega}_{n,k} \cap Q^{\omega}_{n-1,l} \neq \emptyset \big\}, \qquad (n,k) \in \N_{1} \times \Z^{d}.
\]
Then $\#I_{n,k} \lesssim 1$ and there exists $b \in (1,\infty)$ such that
\begin{equation}\label{IR:it:thm:HN_T1.1.14_YL_widetilde;proof;(v)->(iv);1}
Q^{\omega}_{n,k} \subset Q^{\vec{A}}_{n,k}(b) \cap Q^{\vec{A}}_{n-1,l}(b),
\qquad l \in I_{n,k}, (n,k) \in \N_{1} \times \Z^{d}.
\end{equation}
Furthermore, there exists $n_{0} \in \N_{1}$ such that
\begin{equation}\label{IR:it:thm:HN_T1.1.14_YL_widetilde;proof;(v)->(iv);2}
Q^{\vec{A}}_{n,k}(b) \cup Q^{\vec{A}}_{n-1,l}(b) \subset B^{\vec{A}}(x,2^{-(n-n_{0})}),
\qquad x \in Q^{\omega}_{n,k}, (n,k) \in \N \times \Z^{d}.
\end{equation}

For each $i \in I$, let us pick $(\pi^{[i]}_{x^{*},n,k})_{(x^{*},n,k) \in X^{*} \times \N \times \Z^{d}} \subset \mathcal{P}^{d}_{M-1}$ with the property that
\begin{equation}\label{IR:it:thm:HN_T1.1.14_YL_widetilde;proof;(v)->(iv);3}
\norm{\ip{f^{[i]}}{x^{*}}-\pi_{x^{*},n,k}}_{L_{q}(Q^{\vec{A}}_{n,k}(b))} \leq 2\mathcal{E}_{M}(\ip{f^{[i]}}{x^{*}},Q^{\vec{A}}_{n,k}(b),L_{q})
\end{equation}
and put $\pi_{x^{*},n,k} := \sum_{i \in I}1_{S_{i}} \otimes \pi^{[i]}_{x^{*},n,k} \in L_{0}(S;\mathcal{P}^{d}_{M-1})$.
Define
\[
u_{x^{*},n,k} := \left\{\begin{array}{ll}
\omega_{n,k}\sum_{l \in \Z^{d}}\omega_{n-1,l}[\pi_{x^{*},n,k}-\pi_{x^{*},n-1,l}], & n > n_{0},\\
\omega_{n,k}\pi_{x^{*},n,k},& n=n_{0},\\
0, & n < n_{0}.\\
\end{array}\right.
\]

Let $x^{*} \in X^{*}$ and $(n,k) \in \N_{\geq n_{0}+1} \times \Z^{d}$. Let $l \in I_{n,k}$. For $x \in Q^{\omega}_{n,k}$ we can estimate
\begin{align*}
\norm{\pi_{x^{*},n,k}-\pi_{x^{*},n-1,l}}_{L_{q}(Q^{\omega}_{n,k})}
&\stackrel{\eqref{IR:it:thm:HN_T1.1.14_YL_widetilde;proof;(v)->(iv);1}}{\lesssim} \norm{\ip{f}{x^{*}}-\pi_{x^{*},n,k}}_{L_{q}(Q^{\vec{A}}_{n,k}(b))} \\
&\qquad +
 \norm{\ip{f}{x^{*}}-\pi_{x^{*},n-1,l}}_{L_{q}(Q^{\vec{A}}_{n-1,l}(b))} \\
&\stackrel{\eqref{IR:it:thm:HN_T1.1.14_YL_widetilde;proof;(v)->(iv);2},
\eqref{IR:it:thm:HN_T1.1.14_YL_widetilde;proof;(v)->(iv);3}}{\leq}
4\mathcal{E}_{M}(\ip{f}{x^{*}},B^{\vec{A}}(x,2^{-(n-n_{0})}),L_{q}), \\
\end{align*}
implying
\begin{align*}
& \norm{(\pi_{x^{*},n,k}-\pi_{x^{*},n-1,l})(\vec{A}_{2^{-n}}\,\cdot\,+k)}_{C^{M}_{b}([-1,2]^{M})} \\
&\qquad\qquad\qquad\qquad\lesssim 2^{n\mathrm{tr}(\vec{A}^{\oplus})/q}\mathcal{E}_{M}(\ip{f}{x^{*}},B^{\vec{A}}(x,2^{-(n-n_{0})}),L_{q})
\end{align*}
in view of Corollary~\ref{IR:cor:lemma:HN_L1.2.2;abstract_complemented subspace;top_isom;polynomials}.
Since $\#I_{n,k} \lesssim 1$, it follows that
\begin{align}
\norm{u_{x^{*},n,k}(\vec{A}_{2^{-n}}\,\cdot\,+k)}_{C^{M}_{b}([-1,2]^{M})}
&\lesssim \overline{\mathcal{E}}_{M}(\ip{f}{x^{*}},B^{\vec{A}}(x,2^{-(n-n_{0})}),L_{q}) \nonumber \\
&= \mathcal{E}^{\vec{A},\vec{q}}_{M,x^{*},n-n_{0}}(f)(x), \qquad x \in Q^{\omega}_{n,k}.
\label{IR:it:thm:HN_T1.1.14_YL_widetilde;proof;(v)->(iv);4}
\end{align}

For $n=n_{0}$ we similarly have
\begin{align}
\norm{u_{x^{*},n_{0},k}(\vec{A}_{2^{-n_{0}}}\,\cdot\,+k)}_{C^{M}_{b}([-1,2]^{M})}
&\lesssim  \norm{\ip{f}{x^{*}}}_{L_{\vec{q},\mathpzc{d}}(B^{\vec{A}}(x,1))} \nonumber \\
&\lesssim \norm{x^{*}}\,M^{\vec{A}}_{\vec{q}}(\norm{f}_{X})(x) \nonumber \\
&\leq \norm{x^{*}}\,M^{\vec{A}}_{\vec{r}}(\norm{f}_{X})(x), \qquad x \in Q^{\omega}_{n_{0},k}.
\label{IR:it:thm:HN_T1.1.14_YL_widetilde;proof;(v)->(iv);5}
\end{align}

Define $s_{x^{*},n,k} := \norm{u_{x^{*},n,k}(\vec{A}_{2^{-n}}\,\cdot\,+k)}_{C^{M}_{b}([-1,2]^{M})}$,
\[
a_{x^{*},n,k} := \left\{\begin{array}{ll}
\frac{u_{x^{*},n,k}}{s_{x^{*},n,k}},& s_{x^{*},n,k} \neq 0,\\
0,& s_{x^{*},n,k} = 0,\\
\end{array}\right.
\]
and $b_{x^{*},n,k}:=u_{x^{*},n,k}(\vec{A}_{2^{-n}}\,\cdot\,+k)$.
Then $b_{x^{*},n,k} \in C^{M}_{c}([-1,2]^{d})$ with $\norm{b_{x^{*},n,k}}_{C^{M}_{b}} \leq 1$ and
$(s_{x^{*},n,k})_{(x^{*},n,k)} \in \widetilde{y}^{\vec{A}}(E;X)$ with
\begin{align*}
\norm{(s_{x^{*}n,k})_{(x^{*},n,k)}}_{\widetilde{y}^{\vec{A}}(E;X)}
&\stackrel{\eqref{IR:it:thm:HN_T1.1.14_YL_widetilde;proof;(v)->(iv);4},
\eqref{IR:it:thm:HN_T1.1.14_YL_widetilde;proof;(v)->(iv);5}}{\lesssim}
 \norm{M^{\vec{A}}_{\vec{r}}(\norm{f}_{X})}_{E_{0}} \\
&\qquad +\: \normb{\{\mathcal{E}^{\vec{A},\vec{q}}_{M,x^{*},n-n_{0}}(f))\}_{(x^{*},n) \in X^{*} \times \N_{\geq n_{0}}}}_{\mathscr{F}_{\mathrm{M}}(X^{*};E(\N_{\geq n_{0}+1}))} \nonumber \\
&\lesssim \norm{f}_{E_{0}(X)} + 2^{\varepsilon_{-}n_{0}}
\normb{\{\mathcal{E}^{\vec{A},\vec{q}}_{M,x^{*},n}(f))\}_{(x^{*},n)}}_{\mathscr{F}_{\mathrm{M}}(X^{*};E(\N_{\geq1}))}.
\label{IR:it:thm:HN_T1.1.14_YL_widetilde;proof;(v)->(iv);5}
\end{align*}

Note that, for $n \geq n_{0}+1$,
\begin{align*}
\sum_{k \in \Z^{d}}s_{x^{*},n,k}a_{x^{*},n,k}
&= \sum_{k \in \Z^{d}}u_{x^{*},n,k} \\
&= \sum_{k \in \Z^{d}}\pi_{x^{*},n,k}\omega_{x^{*},n,k}\sum_{l \in \Z^{d}}\omega_{n-1,l}
 - \sum_{k \in \Z^{d}}\omega_{n,k}\sum_{l \in \Z^{d}}\pi_{x^{*},n-1,l}\omega_{n-1,l} \\
&=  \sum_{k \in \Z^{d}}\pi_{x^{*},n,k}\omega_{n,k} - \sum_{l \in \Z^{d}}\pi_{x^{*},n-1,l}\omega_{n-1,l}.
\end{align*}
In combination with Lemma~\ref{IR:lemma:HN_L1.2.3} and an alternating sum argument, this implies that
\[
\ip{f}{x^{*}} = \sum_{n=0}^{\infty}\sum_{k \in \Z^{d}}s_{x^{*},n,k}a_{x^{*},n,k} \qquad \text{in} \:\: L_{0}(S;L_{q,\loc}(\R^{d})).
\]
The required convergence finally follows from this with an argument as in (the last part of) the
proof of the implication \eqref{IR:it:thm:HN_T1.1.14_YL;(i)} $\ra$ \eqref{IR:it:thm:HN_T1.1.14_YL;(iv)} in
Theorem~\ref{IR:thm:HN_T1.1.14_YL}.
\end{proof}

\begin{proof}[Proof of Corollary~\ref{IR:cor:thm:HN_T1.1.14_YL(_widetilde);Y}]
This is an immediate consequence of Theorems \ref{IR:thm:incl_comparY&YL}, \ref{IR:thm:HN_T1.1.14_YL}, \ref{IR:thm:HN_T1.1.14_YL_widetilde} and the observation that
\[
\norm{(d^{\vec{A},\vec{p}}_{M,x^{*},n}(f))_{(x^{*},n)}}_{\mathscr{F}_{\mathrm{M}}(X^{*};E)}
\leq \norm{(d^{\vec{A},\vec{p}}_{M,n}(f))_{n \geq 1}}_{E(\N_{1})}. \qedhere
\]
\end{proof}

\begin{proof}[Proof of Theorem~\ref{IR:thm:difference_norm_'Ap-case'}]
The estimates
\[
\norm{f}_{Y^{\vec{A}}(E;X)} \eqsim \norm{f}_{YL^{\vec{A}}(E;X)} \eqsim \norm{f}_{\widetilde{YL}^{\vec{A}}(E;X)}
\]
follow from Theorem~\ref{IR:thm:incl_comparY&YL}.
Combining the inclusion
\[
YL^{\vec{A}}(E;X) \stackrel{\eqref{IR:eq:lemma:HN_L1.1.4;1}}{\hookrightarrow} E_{0}(X)
\]
with the estimate corresponding to the implication \eqref{IR:it:thm:HN_T1.1.14_YL;(i)}$\ra$\eqref{IR:it:thm:HN_T1.1.14_YL;(ii)} in Theorem~\ref{IR:thm:HN_T1.1.14_YL} gives
\[
\norm{f}_{E_{0}(X)} + \norm{(d^{\vec{A},\vec{p}}_{M,n}(f))_{n \geq 1}}_{E(\N_{1};X)}
\lesssim \norm{f}_{YL^{\vec{A}}(E;X)}.
\]
As it clearly holds that
\[
\norm{I^{\vec{A}}_{M,n}(f)}_{X} \leq d^{\vec{A},\vec{p}}_{M,n}(f), \qquad n \in \N,
\]
it remains to be shown that
\begin{equation}\label{IR:eq:thm:difference_norm_'Ap-case';proof;1}
\norm{f}_{Y^{\vec{A}}(E;X)} \lesssim \norm{f}_{E_{0}(X)} + \norm{(I^{\vec{A}}_{M,n}(f))_{n \geq 1}}_{E(\N_{1};X)}.
\end{equation}

Put $K:=1_{B^{\vec{A}}}(0,1)$ and $K^{\Delta^{M}} := \sum_{l=0}^{M-1}(-1)^{l}{M \choose l}\tilde{K}_{[M-l]^{-1}}$, where $\tilde{K}_{t} := t^{d}K(-t\,\cdot\,)$ for $t \in (0,\infty)$.
Furthermore, put
\[
K^{\vec{A}}_{M}(t,f) := t^{-\mathrm{tr}(\vec{A}^{\oplus})}K^{\Delta^{M}}(\vec{A}_{t^{-1}}\,\cdot\,)*f + (-1)^{M}\hat{K}(0)f, \qquad t \in (0,\infty).
\]
Note that
\begin{equation}\label{IR:eq:thm:difference_norm_'Ap-case';proof;2}
I^{\vec{A}}_{M,n}(f) = K^{\vec{A}}_{M}(2^{-n},f), \qquad n \in \N.
\end{equation}

As $\widehat{K^{\Delta^{M}}}(0)=\sum_{l=0}^{M-1}(-1)^{l}{M \choose l}\hat{K}(0)=(-1)^{M+1}\hat{K}(0) \neq 0$, we can pick $\epsilon,c \in (0,\infty)$ such that $K^{\Delta^{m}}$ fulfills the Tauberian condition
\[
|\mathscr{F}K^{\Delta^{m}}(\xi)| \geq c, \quad\quad \xi \in \R^{d}, \frac{\epsilon}{2} < \rho_{\vec{A}}(\xi) < 2\epsilon.
\]
So there exists $N \in \N$ such that $k:=  2^{N\mathrm{tr}(\vec{A}^{\oplus})}K^{\Delta^{m}}(\vec{A}_{2^{N}}\,\cdot\,) - K^{\Delta^{m}} \in L_{1,\mathrm{c}}(\R^{d})$ satisfies
\[
|\hat{k}(\xi)| \geq \frac{c}{2} > 0, \quad\quad \xi \in \R^{d}, \frac{\delta}{2} < \rho_{\vec{A}}(\xi) < 2\delta,
\]
for $\delta:= 2^{N}\epsilon > 0$.
Let $\varphi = (\varphi_{n})_{n \in \N} \in \Phi^{\vec{A}}(\R^{d})$ be such that $\supp \hat{\varphi}_{1} \subset  \{ \xi : 2\epsilon \leq \rho_{\vec{A}}(\xi)\}$ (see Definition~\ref{IR:def:LP-sequences}).
Let $(k_{n})_{n \in \N}$ be defined by $k_{n}:=2^{n\mathrm{tr}(\vec{A}^{\oplus})}k(\vec{A}_{2^{n}}\,\cdot\,)$.
Then, by construction,
\begin{align*}
k_{n}*f
= K^{\vec{A}}_{M}(2^{-(n+N)},f) - K^{\vec{A}}_{M}(2^{-n},f)
\stackrel{\eqref{IR:eq:thm:difference_norm_'Ap-case';proof;2}}{=} I^{\vec{A}}_{M,n+N}(f) - I^{\vec{A}}_{M,n}(f), \qquad n \in \N.
\end{align*}
An application of Lemma~\ref{IR:lemma:one-sided_estimate_local_means} thus yields that
\begin{align}
\norm{(\varphi_{n}*f)_{n \geq 1}}_{E(\N_{1};X)}
&\lesssim \norm{(k_{n}*f)_{n \geq 1}}_{E(\N_{1};X)} \nonumber \\
&\lesssim \norm{(I^{\vec{A}}_{M,n+N}(f))_{n \geq 1}}_{E(\N_{1};X)} + \norm{(I^{\vec{A}}_{M,n}(f))_{n \geq 1}}_{E(\N_{1};X)} \nonumber \\
&\lesssim (2^{-\varepsilon_{+}N}+1)\norm{(I^{\vec{A}}_{M,n}(f))_{n \geq 1}}_{E(\N_{1};X)}. \label{IR:eq:thm:difference_norm_'Ap-case';proof;3}
\end{align}
As $\norm{\varphi_{0}*f}_{X} \lesssim M^{\vec{A}}(\norm{f}_{X})$, it furthermore holds that
\begin{equation}\label{IR:eq:thm:difference_norm_'Ap-case';proof;4}
\norm{\varphi_{0}*f}_{E_{0}(X)} \lesssim \norm{f}_{E_{0}(X)}.
\end{equation}
A combination of Proposition~\ref{IR:prop:LP-decomp_characterization},
\eqref{IR:eq:thm:difference_norm_'Ap-case';proof;3} and
\eqref{IR:eq:thm:difference_norm_'Ap-case';proof;4} finally gives
\eqref{IR:eq:thm:difference_norm_'Ap-case';proof;1}.
\end{proof}

\begin{proof}[Proof of Proposition~\ref{IR:prop:HN_T1.1.14;difference+shift_one-sided}]
Using the estimate corresponding to the implication \eqref{IR:it:thm:HN_T1.1.14_YL;(i)} $\ra$ \eqref{IR:it:thm:HN_T1.1.14_YL;(iv)} in Theorem~\ref{IR:thm:HN_T1.1.14_YL}, the first estimate can be obtained as in the proof of the implication \eqref{IR:it:thm:HN_T1.1.14_YL;(iv)} $\ra$ \eqref{IR:it:thm:HN_T1.1.14_YL;(ii)} in Theorem~\ref{IR:thm:HN_T1.1.14_YL}.
The second estimate can be obained similarly, replacing Theorem~\ref{IR:thm:HN_T1.1.14_YL} by Theorem~\ref{IR:thm:HN_T1.1.14_YL_widetilde}.
\end{proof}

\section{An Intersection Representation}\label{IR:sec:IR}

In this section we come to the main results of this paper, namely, intersection representations. In particular, these include Theorem~\ref{IR:intro:thm} from the introduction of this paper as a special case. 
Before we can state the results, we need to introduce some notation.

Let $E \in \mathcal{S}(\varepsilon_{+},\varepsilon_{-},\vec{A},\vec{r},(S,\mathscr{A},\mu))$ with $\varepsilon_{+},\varepsilon_{-} > 0$. Let $J$ be a nonempty subset of $\{1,\ldots,\ell\}$, say $J=\{j_{1},\ldots,j_{k}\}$ with $1 \leq j_{1} \leq \ldots \leq j_{k} \leq \ell$. Put $\mathpzc{d}_{J}=(\mathpzc{d}_{j_{1}},\ldots,\mathpzc{d}_{j_{k}})$, $d_{J}:=|\mathpzc{d}_{J}|_{1}$ $\vec{A}_{J}:=(A_{j_{1}},\ldots,A_{j_{k}})$, $\vec{r}_{J}:=(r_{j_{1}},\ldots,r_{j_{k}})$ and
\[
(S_{J},\mathscr{A}_{J},\mu_{J}) :=   (\R^{d-d_{J}},\mathcal{B}(\R^{d-d_{J}}),\lambda^{d-d_{J}}) \otimes (S,\mathscr{A},\mu)
\]
Furthermore, define $E_{[\mathpzc{d};J]}$ as the quasi-Banach space $E$ viewed as quasi-Banach function space on the measure space $\R^{d_{J}} \times \N \times S_{J}$.
Then
\[
E_{[\mathpzc{d};J]} \in \mathcal{S}(\varepsilon_{+},\varepsilon_{-},\vec{A}_{J},\vec{r}_{J},(S_{J},\mathscr{A}_{J},\mu_{J}))
\]
By Remark~\ref{IR:rmk:indep_r_YL},
\[
\widetilde{YL}^{\vec{A}}(E;X) \hookrightarrow E^{\vec{A}}_{\otimes}(B^{1,w_{\vec{A},\vec{r}}}_{\vec{A}}(X)) \hookrightarrow L_{0}(S;L_{\vec{r},\mathpzc{d},\loc}(\R^{d};X)).
\]
In the same way,
\[
\widetilde{YL}^{\vec{A}_{J}}(E_{[\mathpzc{d};J]};X) \hookrightarrow E^{\vec{A}}_{\otimes}(B^{1,w_{\vec{A},\vec{r}}}_{\vec{A}}(X)) \hookrightarrow L_{0}(S;L_{\vec{r},\mathpzc{d},\loc}(\R^{d};X)),
\]
In particular, it makes sense to compare $\widetilde{YL}^{\vec{A}_{J}}(E_{[\mathpzc{d};J]};X)$ with $\widetilde{YL}^{\vec{A}}(E;X)$.

\begin{theorem}\label{IR:thm:IR}
Let $E \in \mathcal{S}(\varepsilon_{+},\varepsilon_{-},\vec{A},\vec{r},(S,\mathscr{A},\mu))$ with $\varepsilon_{+},\varepsilon_{-} > 0$.
Let $\{J_{1},\ldots,J_{L}\}$ be a partition of $\{1,\ldots,\ell\}$.
\begin{enumerate}[(i)]
\item\label{IR:it:thm:IR;1} There is the estimate
\[
\norm{f}_{\widetilde{YL}^{\vec{A}_{J_{l}}}(E_{[\mathpzc{d};J_{l}]};X)} \leq \norm{f}_{\widetilde{YL}^{\vec{A}}(E;X)}, \qquad l \in \{1,\ldots,L\},
\]
for all $f \in L_{0}(S;L_{\vec{r},\mathpzc{d},\loc}(\R^{d};X))$.
\item\label{IR:it:thm:IR;2} There is the estimate
\[
\norm{f}_{\widetilde{YL}^{\vec{A}}(E;X)} \lesssim \sum_{l=1}^{L}\norm{f}_{\widetilde{YL}^{\vec{A}_{J_{l}}}(E_{[\mathpzc{d};J_{l}]};X)}
\]
for all $f \in L_{0}(S;L_{\vec{r},\mathpzc{d},\loc}(\R^{d};X))$ of the form $f=\sum_{i \in I}1_{S_{i}} \otimes f^{[i]}$ with $(S_{i})_{i \in I} \subset \mathscr{A}$ a countable family of mutually disjoint sets and $(f^{[i]})_{i \in I} \in L_{\vec{r},\mathpzc{d},\loc}(\R^{d};X)$.
\end{enumerate}
In particular, in case $(S,\mathscr{A},\mu)$ is atomic,
\[
\widetilde{YL}^{\vec{A}}(E;X) =\bigcap_{l=1}^{L}\widetilde{YL}^{\vec{A}_{J_{l}}}(E_{[\mathpzc{d};J_{l}]};X)
\]
with an equivalence of quasi-norms.
\end{theorem}

\begin{remark}\label{IR:rmk:thm:IR}
The analogous estimate in Theorem~\ref{IR:thm:IR}.\eqref{IR:it:thm:IR;1} for $YL^{\vec{A}}(E;X)$ holds as well, with a slightly modified proof that actually is a little bit easier. However, we are not able to obtain a version of Theorem~\ref{IR:thm:IR}.\eqref{IR:it:thm:IR;2} for $YL^{\vec{A}}(E;X)$ due to the unavailability of the crucial implication \eqref{IR:it:thm:HN_T1.1.14_YL;(ii)} $\ra$ \eqref{IR:it:thm:HN_T1.1.14_YL;(i)} (plus a corresponding estimate of the involved quasi-norm) in Theorem~\ref{IR:thm:HN_T1.1.14_YL}, see Remark~\ref{IR:rmk:thm:HN_T1.1.14_YL}.
\end{remark}

\begin{proof}[Proof of Theorem~\ref{IR:thm:IR}]
Let us start with \eqref{IR:it:thm:IR;1}. Fix $l \in \{1,\ldots,L\}$ and write $J:=J_{l}$.
Let $f \in \widetilde{YL}^{\vec{A}}(E;X)$. Let $\epsilon > 0$.
Choose $(g_{n})_{n}$ and $(f_{x^{*},n})_{(x^{*},n)}$ as in Definition~\ref{IR:def:YL_widetilde} with $\norm{(g_{n})_{n}}_{E} \leq (1+\epsilon)\norm{f}_{\widetilde{YL}^{\vec{A}}(E;X)}$.
As $f_{x^{*},n} \in L_{0}(S;\mathcal{S}'(\R^{d}))$ with $\supp\hat{f}_{x^{*},n} \subset B^{\vec{A}}(0,2^{n+1})$, we can naturally view $f_{x^{*},n}$ as an element of $L_{0}(S_{J};\mathcal{S}'(\R^{d-d_{J}}))$ with $\supp\hat{f}_{x^{*},n} \subset B^{\vec{A}_{J}}(0,2^{n+1})$.
Since
\[
L_{0}(S;L_{\vec{r},\mathpzc{d},\loc}(\R^{d})) \hookrightarrow L_{0}(S_{J};L_{\vec{r}_{J},\mathpzc{d}_{J},\loc}(\R^{d_{J}})),
\]
it follows that $f \in \widetilde{YL}^{\vec{A}_{J}}(E_{[\mathpzc{d};J]};X)$ with
\[
\norm{f}_{\widetilde{YL}^{\vec{A}_{J}}(E_{[\mathpzc{d};J]};X)} \lesssim \norm{(g_{n})_{n}}_{E_{[\mathpzc{d};J]}} = \norm{(g_{n})_{n}}_{E} \leq (1+\epsilon)\norm{f}_{\widetilde{YL}^{\vec{A}}(E;X)}.
\]

Let us next treat \eqref{IR:it:thm:IR;2}. We may without loss of generality assume that $L=\ell$ and that $J_{l}=\{l\}$ for each $l \in \{1,\ldots,\ell\}$. We will write $E_{[\mathpzc{d};j]} = E_{[\mathpzc{d};\{j\}]}$.

Let $f \in \bigcap_{j=1}^{\ell}\widetilde{YL}^{A_{j}}(E_{[\mathpzc{d};j]};X)$ be of the form $f=\sum_{i \in I}1_{S_{i}} \otimes f^{[i]}$ with $(S_{i})_{i \in I} \subset \mathscr{A}$ a countable family of mutually disjoint sets and $(f^{[i]})_{i \in I} \in L_{\vec{r},\mathpzc{d},\loc}(\R^{d};X)$.
In order to establish the desired inequality, we will combine the estimate corresponding to the implication \eqref{IR:it:thm:HN_T1.1.14_YL_widetilde;(ii)} $\ra$ \eqref{IR:it:thm:HN_T1.1.14_YL_widetilde;(i)} from Theorem~\ref{IR:thm:HN_T1.1.14_YL_widetilde} for the space $\widetilde{YL}^{\vec{A}}(E;X)$ with the estimates from Proposition~\ref{IR:prop:HN_T1.1.14;difference+shift_one-sided} for each of the spaces $\widetilde{YL}^{A_{j}}(E_{[\mathpzc{d};j]};X)$.
To this end, pick $M \in \N$ with $M\lambda^{\vec{A}}_{\min} > \varepsilon_{-}$.
Now, let us define $(g_{x^{*},n})_{(x^{*},n) \in X^{*} \times \N}$ and $(g_{c,x^{*},n,j})_{(x^{*},n) \in X^{*} \times \N}$, with $j \in \{1,\ldots,\ell\}$ and $c \in \R$, by
\[
g_{x^{*},n} := \left\{\begin{array}{ll}
d^{\vec{A},\vec{r}}_{0,x^{*},0}(f),& n=0,\\
d^{\vec{A},\vec{r}}_{\ell M,x^{*},n}(f),& n \geq 1,
\end{array}\right.
\]
and
\[
g_{c,x^{*},n,j} := \left\{\begin{array}{ll}
d^{[\mathpzc{d};j],A_{j},r_{j}}_{0,x^{*},0}(f),& n=0,\\
d^{[\mathpzc{d};j],A_{j},r_{j}}_{M,c,x^{*},n}(f),& n \geq 1,
\end{array}\right.
\]
where the notation is as in Theorem~\ref{IR:thm:HN_T1.1.14_YL_widetilde} and
Proposition~\ref{IR:prop:HN_T1.1.14;difference+shift_one-sided}.

For $n=0$ we have
\begin{align}
g_{x^{*},0} = d^{\vec{A},\vec{r}}_{0,x^{*},0}(f) &\lesssim \big[\bigcirc_{i=2}^{\ell}M^{[\mathpzc{d};i],A_{i}}_{r_{i}}\big]
\big(d^{[\mathpzc{d};1],A_{1},r_{1}}_{0,x^{*},0}(f)\big)  \nonumber\\
&\leq M^{\vec{A}}_{\vec{r}}\big[d^{[\mathpzc{d};1],A_{1},r_{1}}_{0,x^{*},0}(f)\big] = M^{\vec{A}}_{\vec{r}}\big[g_{c,x^{*},0,1}\big], \qquad c \in \R, \label{IR:eq:thm:IR;proof;1}
\end{align}
where $\bigcirc_{i=2}^{\ell}M^{[\mathpzc{d};i],A_{i}}_{r_{i}}$ stands for the composition $M^{[\mathpzc{d};\ell],A_{\ell}}_{r_{\ell}} \circ \ldots \circ  M^{[\mathpzc{d};2],A_{2}}_{r_{2}}$.

Now let $n \geq 1$. We will use the following elementary fact (cf.\ \cite[4.16]{Triebel2001_SF}): there exist $C \in (0,\infty)$, $K \in \N$ and $\{c^{[k]}_{j}\}_{j=1,\ldots,\ell;k=0,\ldots,K} \subset \R$ such that
\[
|\Delta^{\ell M}_{z}h(x)| \leq C\sum_{k=0}^{K}\sum_{j=1}^{\ell}\Big| \Delta^{M}_{\iota_{[\mathpzc{d};j]}z_{j}}h(x+\sum_{i=1}^{\ell}c^{[k]}_{i}\iota_{[\mathpzc{d};i]}z_{i}) \Big|
\]
for all $h \in L_{0}(\R^{d})$. Applying this pointwise in $S$ to $\ip{f}{x^{*}}$, we find that
\begin{align}
g_{x^{*},n} &= d^{\vec{A},\vec{r}}_{\ell M,x^{*},n}(f) = 2^{n\mathrm{tr}(\vec{A})\bcdot\vec{r}^{-1}}\normb{z \mapsto \Delta_{z}^{\ell M}\ip{f}{x^{*}}}_{L_{\vec{r},\mathpzc{d}}(B^{\vec{A}}(0,2^{-n}))} \nonumber\\
&\lesssim \sum_{k=0}^{K}\sum_{j=1}^{\ell}2^{n\mathrm{tr}(\vec{A})\bcdot\vec{r}^{-1}}\normB{z \mapsto \Big[\prod_{i=1}^{\ell}L_{c^{[k]}_{i}\iota_{[\mathpzc{d};i]}z_{i}}\Big]
\Delta^{M}_{\iota_{[\mathpzc{d};j]}z_{j}}\ip{f}{x^{*}}}_{L_{\vec{r},\mathpzc{d}}(B^{\vec{A}}(0,2^{-n}))} \nonumber\\
&\lesssim \sum_{k=0}^{K}\sum_{j=1}^{\ell}2^{n\mathrm{tr}(A_{j})/r_{j}}\big[\bigcirc_{i \neq j}M^{[\mathpzc{d};i],A_{i}}_{r_{i}}\big]\left[\normB{z_{j} \mapsto L_{c^{[k]}_{j}\iota_{[\mathpzc{d};j]}z_{j}}
\Delta^{M}_{\iota_{[\mathpzc{d};j]}z_{j}}\ip{f}{x^{*}}}_{L_{r_{j}}(B^{A_{j}}(0,2^{-n}))}\right] \nonumber\\
&\leq \sum_{k=0}^{K}\sum_{j=1}^{\ell}M^{\vec{A}}_{\vec{r}}\left[2^{n\mathrm{tr}(A_{j})/r_{j}}\normB{z_{j} \mapsto L_{c^{[k]}_{j}\iota_{[\mathpzc{d};j]}z_{j}}
\Delta^{M}_{\iota_{[\mathpzc{d};j]}z_{j}}\ip{f}{x^{*}}}_{L_{r_{j}}(B^{A_{j}}(0,2^{-n}))}\right] \nonumber\\
&= \sum_{k=0}^{K}\sum_{j=1}^{\ell}M^{\vec{A}}_{\vec{r}}\left[
d^{[\mathpzc{d};j],A_{j},r_{j}}_{M,c^{[k]}_{j},x^{*},n}(f)\right] = \sum_{k=0}^{K}\sum_{j=1}^{\ell}M^{\vec{A}}_{\vec{r}}\left[g_{c^{[k]}_{j},x^{*},n,j}\right]. \label{IR:eq:thm:IR;proof;2}
\end{align}

A combination of \eqref{IR:eq:thm:IR;proof;1} and \eqref{IR:eq:thm:IR;proof;2} gives
\[
g_{x^{*},n} \lesssim \sum_{k=0}^{K}\sum_{j=1}^{\ell}M^{\vec{A}}_{\vec{r}}\left[
d^{[\mathpzc{d};j],A_{j},r_{j}}_{M,c^{[k]}_{j},x^{*},n}(f)\right] = \sum_{k=0}^{K}\sum_{j=1}^{\ell}M^{\vec{A}}_{\vec{r}}\left[g_{c^{[k]}_{j},x^{*},n,j}\right]
\]
for all $(x^{*},n) \in X^{*} \times \N$. Therefore,
\begin{align*}
\normb{\{g_{x^{*},n}\}_{(x^{*},n)}}_{\mathscr{F}_{\mathrm{M}}(X^{*};E)}
&\lesssim \sum_{k=0}^{K}\sum_{j=1}^{\ell} \normB{\big\{M^{\vec{A}}_{\vec{r}}\big[g_{c^{[k]}_{j},x^{*},n,j}\big]\big\}_{(x^{*},n)}}_{\mathscr{F}_{\mathrm{M}}(X^{*};E)} \\
&\lesssim \sum_{k=0}^{K}\sum_{j=1}^{\ell} \normb{\{g_{c^{[k]}_{j},x^{*},n,j}\}_{(x^{*},n)}}_{\mathscr{F}_{\mathrm{M}}(X^{*};E)} \\
&= \sum_{k=0}^{K}\sum_{j=1}^{\ell} \normb{\{g_{c^{[k]}_{j},x^{*},n,j}\}_{(x^{*},n)}}_{\mathscr{F}_{\mathrm{M}}(X^{*};E_{[\mathpzc{d};j]})}.
\end{align*}
The desired result now follows from a combination of Theorem~\ref{IR:thm:HN_T1.1.14_YL_widetilde} and Proposition~\ref{IR:prop:HN_T1.1.14;difference+shift_one-sided}.
\end{proof}

As an immediate corollary to Theorems \ref{IR:thm:incl_comparY&YL} and \ref{IR:thm:IR} we have:

\begin{corollary}\label{IR:cor:thm:IR;comparY&YL}
Let $E \in \mathcal{S}(\varepsilon_{+},\varepsilon_{-},\vec{A},\vec{r},(S,\mathscr{A},\mu))$ with $\varepsilon_{+},\varepsilon_{-} > 0$ and $(S,\mathscr{A},\mu)$ atomic.
Let $\{J_{1},\ldots,J_{L}\}$ be a partition of $\{1,\ldots,\ell\}$.
If $\varepsilon_{+} > \mathrm{tr}(\vec{A})\bcdot(\vec{r}^{-1}-\vec{1})_{+}$, where  $\mathrm{tr}(\vec{A})=(\mathrm{tr}(A_1),\ldots,\mathrm{tr}(A_\ell))$, then
\begin{align*}
Y^{\vec{A}}(E;X) &= YL^{\vec{A}}(E;X) = \widetilde{YL}^{\vec{A}}(E;X)
= \bigcap_{l=1}^{L}\widetilde{YL}^{\vec{A}_{J_{l}}}(E_{[\mathpzc{d};J_{l}]};X) \\
&= \bigcap_{l=1}^{L}YL^{\vec{A}_{J_{l}}}(E_{[\mathpzc{d};J_{l}]};X) =
\bigcap_{l=1}^{L}Y^{\vec{A}_{J_{l}}}(E_{[\mathpzc{d};J_{l}]};X)
\end{align*}
with an equivalence of quasi-norms.
\end{corollary}

In the case that $\vec{r}=1$, the above intersection representation simplifies a bit thanks to the corresponding simplification in the crucial estimate involving differences, also see
Remark~\ref{IR:rmk:ex:prop:LP-decomp_characterization}. In particular, we can drop the assumption of $(S,\mathscr{A},\mu)$ being atomic.

\begin{theorem}\label{IR:thm:IR_Ap-case}
Let $E \in \mathcal{S}(\varepsilon_{+},\varepsilon_{-},\vec{A},\vec{1},(S,\mathscr{A},\mu))$ with $\varepsilon_{+},\varepsilon_{-} > 0$.
Let $\{J_{1},\ldots,J_{L}\}$ be a partition of $\{1,\ldots,\ell\}$.
Then
\begin{align*}
Y^{\vec{A}}(E;X) &= YL^{\vec{A}}(E;X) = \widetilde{YL}^{\vec{A}}(E;X)
= \bigcap_{l=1}^{L}\widetilde{YL}^{\vec{A}_{J_{l}}}(E_{[\mathpzc{d};J_{l}]};X) \\
&= \bigcap_{l=1}^{L}YL^{\vec{A}_{J_{l}}}(E_{[\mathpzc{d};J_{l}]};X) =
\bigcap_{l=1}^{L}Y^{\vec{A}_{J_{l}}}(E_{[\mathpzc{d};J_{l}]};X)
\end{align*}
with an equivalence of quasi-norms.
\end{theorem}
\begin{proof}
In view of Theorem~\ref{IR:thm:incl_comparY&YL}, this can be proved in exactly the same way as Theorem~\ref{IR:thm:IR}, using  Theorem~\ref{IR:thm:difference_norm_'Ap-case'} instead of Theorem~\ref{IR:thm:HN_T1.1.14_YL_widetilde}.
\end{proof}

\begin{remark}\label{IR:ex:cor:thm:IR;comparY&YL}
In light of Example~\ref{IR:ex:prop:LP-decomp_characterization}, the intersection representation
\begin{equation}\label{IR:eq:ex:cor:thm:IR;comparY&YL;IR}
Y^{\vec{A}}(E;X) = \bigcap_{l=1}^{L}Y^{\vec{A}_{J_{l}}}(E_{[\mathpzc{d};J_{l}]};X)
\end{equation}
from Corollary~\ref{IR:cor:thm:IR;comparY&YL} and Theorem~\ref{IR:thm:IR_Ap-case} extends the well-known Fubini property for the classical Lizorkin-Triebel spaces $F^{s}_{p,q}(\R^{d})$ (see \cite[Section~4]{Triebel2001_SF} and the references given therein). It also covers Theorem~\ref{IR:intro:thm} and thereby \eqref{IR:intro:eq:intersection_rep;Denk&Kaip}, the intersection representation from \cite[Proposition~3.23]{Denk&Kaip2013}.
The intersection representation \cite[Proposition~5.2.38]{Lindemulder_master-thesis} for anisotropic weighted mixed-norm Lizorkin-Triebel is a special case as well.
Furthermore, it suggests an operator sum theorem for generalized Lizorkin-Triebel spaces in the sense of \cite{Kunstmann&Ullmann2014}.
\end{remark}

\begin{example}\label{IR:ex:ex:cor:thm:IR;comparY&YL;concrete_examples}
Let us state the intersection representation \eqref{IR:eq:ex:cor:thm:IR;comparY&YL;IR}
from Corollary~\ref{IR:cor:thm:IR;comparY&YL} and Theorem~\ref{IR:thm:IR_Ap-case} for some concrete choices of $E$ (see Examples \ref{IR:ex:def:S} and \ref{IR:ex:prop:LP-decomp_characterization}) for the case that $\ell=2$ with partition $\{\{1\},\{2\}\}$ of $\{1,2\}$.
\begin{enumerate}[(I)]
\item\label{IR:it:ex:ex:cor:thm:IR;comparY&YL;concrete_examples;I} Let $\vec{p} \in (0,\infty)^{2}$, $q \in (0,\infty]$, $\vec{w} \in A_{\infty}(\R^{\mathpzc{d}_{1}},A_{1}) \times A_{\infty}(\R^{\mathpzc{d}_{2}},A_{2})$ and $s \in \R$. Pick $\vec{r} \in (0,\infty)^{2}$ such that $r_{1} < p_{1} \wedge q$, $r_{2} < p_{1} \wedge p_{2} \wedge q$ and $\vec{w} \in A_{p_1/r_1}(\R^{\mathpzc{d}_{1}},A_{1}) \times A_{p_2/r_2}(\R^{\mathpzc{d}_{2}},A_{2})$. If $s > \mathrm{tr}(\vec{A})\bcdot(\vec{r}^{-1}-\vec{1})_{+}$, then
    \begin{align*}
    F^{s,\vec{A}}_{\vec{p},q}(\R^{d},\vec{w};X) &= \F^{s,A_{2}}_{p_{2},q}(\R^{\mathpzc{d}_{2}},w_{2};L_{p_{1}}(\R^{\mathpzc{d}_{1}},w_{1});X) \\
    &\qquad\qquad \cap
    L_{p_{2}}(\R^{\mathpzc{d}_{2}},w_{2};F^{s,A_{1}}_{p_{1},q}(\R^{\mathpzc{d}_{1}},w_{1};X)).
    \end{align*}
\item\label{IR:it:ex:ex:cor:thm:IR;comparY&YL;concrete_examples;II} Let $\vec{p} \in (0,\infty)^{2}$, $q \in (0,\infty]$, $\vec{w} \in A_{\infty}(\R^{\mathpzc{d}_{1}},A_{1}) \times A_{\infty}(\R^{\mathpzc{d}_{2}},A_{2})$ and $s \in \R$. Pick $\vec{r} \in (0,\infty)^{2}$ such that $r_{1} < p_{1}$, $r_{2} < p_{1} \wedge p_{2} \wedge q$ and $\vec{w} \in A_{p_1/r_1}(\R^{\mathpzc{d}_{1}},A_{1}) \times A_{p_2/r_2}(\R^{\mathpzc{d}_{2}},A_{2})$. If $s > \mathrm{tr}(\vec{A})\bcdot(\vec{r}^{-1}-\vec{1})_{+}$, then
    \begin{align*}
    &Y^{\vec{A}}\left(L_{p_{2}}(\R^{\mathpzc{d}_{2}},w_{2})[[\ell_{q}^{s}(\N)]
    L_{p_{1}}(\R^{\mathpzc{d}_{1}},w_{1})];X\right)  \\ &\qquad= F^{s,A_{2}}_{p_{2},q}(\R^{\mathpzc{d}_{2}},w_{2};L_{p_{1}}(\R^{\mathpzc{d}_{1}},w_{1};X)) \cap
    L_{p_{2}}(\R^{\mathpzc{d}_{2}},w_{2};B^{s,A_{1}}_{p_{1},q}(\R^{\mathpzc{d}_{1}},w_{1};X)).
    \end{align*}
\end{enumerate}
\end{example}

To finish this section, let us finally state the Fubini property variants of the two examples from Example~\ref{IR:ex:ex:cor:thm:IR;comparY&YL;concrete_examples} (cf.\ Remark~\ref{IR:ex:cor:thm:IR;comparY&YL}).
\begin{example}\label{IR:ex:ex:cor:thm:IR;comparY&YL;concrete_examples;Fubini}
Taking $\vec{p}=(p,q)$ in \eqref{IR:it:ex:ex:cor:thm:IR;comparY&YL;concrete_examples;I} and \eqref{IR:it:ex:ex:cor:thm:IR;comparY&YL;concrete_examples;II} of Example~\ref{IR:ex:ex:cor:thm:IR;comparY&YL;concrete_examples}, an application of Fubini's theorem yields the following.
\begin{enumerate}[(I)]
\item Let $p,q \in (0,\infty)$, $\vec{w} \in A_{\infty}(\R^{\mathpzc{d}_{1}},A_{1}) \times A_{\infty}(\R^{\mathpzc{d}_{2}},A_{2})$ and $s \in \R$. Pick $\vec{r} \in (0,\infty)^{2}$ such that $r_{1},r_{2} < p \wedge q$ and $\vec{w} \in A_{p/r_1}(\R^{\mathpzc{d}_{1}},A_{1}) \times A_{q/r_2}(\R^{\mathpzc{d}_{2}},A_{2})$. If $s > \mathrm{tr}(\vec{A})\bcdot(\vec{r}^{-1}-\vec{1})_{+}$, then
    \begin{align*}
    F^{s,\vec{A}}_{(p,q),p}(\R^{d},\vec{w};X) &= F^{s,A_{2}}_{q,p}(\R^{\mathpzc{d}_{2}},w_{2};L_{p}(\R^{\mathpzc{d}_{1}},w_{1};X)) \\
    &\qquad\qquad \cap
    L_{q}(\R^{\mathpzc{d}_{2}},w_{2};B^{s,A_{1}}_{p,p}(\R^{\mathpzc{d}_{1}},w_{1};X)).
    \end{align*}
\item Let $p \in (0,\infty)$, $q \in (0,\infty]$, $\vec{w} \in A_{\infty}(\R^{\mathpzc{d}_{1}},A_{1}) \times A_{\infty}(\R^{\mathpzc{d}_{2}},A_{2})$ and $s \in \R$. Pick $\vec{r} \in (0,\infty)^{2}$ such that $r_{1} < p$, $r_{2} < p \wedge q$ and $\vec{w} \in A_{p/r_1}(\R^{\mathpzc{d}_{1}},A_{1}) \times A_{q/r_2}(\R^{\mathpzc{d}_{2}},A_{2})$. If $s > \mathrm{tr}(\vec{A})\bcdot(\vec{r}^{-1}-\vec{1})_{+}$, then
    \begin{align*}
    B^{s,\vec{A}}_{(p,q),q}(\R^{d},\vec{w};X) &=
    B^{s,A_{2}}_{q,q}(\R^{\mathpzc{d}_{2}},w_{2};L_{p}(\R^{\mathpzc{d}_{1}},w_{1};X)) \\
    &\qquad\qquad \cap
    L_{q}(\R^{\mathpzc{d}_{2}},w_{2};B^{s,A_{1}}_{p,q}(\R^{\mathpzc{d}_{1}},w_{1};X)).
    \end{align*}
\end{enumerate}
\end{example}

In applications to parabolic partial differential equations, one uses anisotropies of the form $\vec{A}=(a_1I_{\mathpzc{d}_1},a_2I_{\mathpzc{d}_2})$ with $a_1=2m$, $a_2=1$, $\mathpzc{d}_1 \in \{n-1,n\}$ and $\mathpzc{d}_2 = 1$,  where $2m$ is the order of the elliptic operator under consideration and $n$ is the dimension of the spatial domain (see e.g.\ \cite{Lindemulder2018_DSOP,lindemulder2017maximal}).
So let us for convenience of reference state Examples \ref{IR:ex:ex:cor:thm:IR;comparY&YL;concrete_examples} and \ref{IR:ex:ex:cor:thm:IR;comparY&YL;concrete_examples;Fubini} for such anisotropies.

In view of Example~\ref{IR:ex:prop:scaling anisotropy} and the fact that $A_p(\R^n,\lambda A) = A_p(\R^n,A)$ for every $\lambda \in (0,\infty)$, the following two examples are obtained as special cases of Examples \ref{IR:ex:ex:cor:thm:IR;comparY&YL;concrete_examples} and \ref{IR:ex:ex:cor:thm:IR;comparY&YL;concrete_examples;Fubini}.

\begin{example}\label{IR:ex:ex:cor:thm:IR;comparY&YL;concrete_examples;scalar_anisotropies} Let $\mathpzc{d} \in (\N_{1})^2$ and $\vec{a} \in (0,\infty)^2$.
\begin{enumerate}[(I)]
\item\label{IR:it:ex:ex:cor:thm:IR;comparY&YL;concrete_examples;I;scalar_anisotropies} Let $\vec{p} \in (0,\infty)^{2}$, $q \in (0,\infty]$, $\vec{w} \in A_{\infty}(\R^{\mathpzc{d}_{1}}) \times A_{\infty}(\R^{\mathpzc{d}_{2}})$ and $s \in \R$. Pick $\vec{r} \in (0,\infty)^{2}$ such that $r_{1} < p_{1} \wedge q$, $r_{2} < p_{1} \wedge p_{2} \wedge q$ and $\vec{w} \in A_{p_1/r_1}(\R^{\mathpzc{d}_{1}}) \times A_{p_2/r_2}(\R^{\mathpzc{d}_{2}})$. If $s > a_1\mathpzc{d}_1(r_1^{-1}-1)_{+} + a_2\mathpzc{d}_2(r_2^{-1}-1)_{+}$, then
    \begin{align*}
    F^{s,(\vec{a};\mathpzc{d})}_{\vec{p},q}(\R^{d},\vec{w};X) &= \F^{s/a_2}_{p_{2},q}(\R^{\mathpzc{d}_{2}},w_{2};L_{p_{1}}(\R^{\mathpzc{d}_{1}},w_{1});X) \\
    &\qquad\qquad \cap
    L_{p_{2}}(\R^{\mathpzc{d}_{2}},w_{2};F^{s/a_1}_{p_{1},q}(\R^{\mathpzc{d}_{1}},w_{1};X)).
    \end{align*}
\item Let $\vec{p} \in (0,\infty)^{2}$, $q \in (0,\infty]$, $\vec{w} \in A_{\infty}(\R^{\mathpzc{d}_{1}}) \times A_{\infty}(\R^{\mathpzc{d}_{2}})$ and $s \in \R$. Pick $\vec{r} \in (0,\infty)^{2}$ such that $r_{1} < p_{1}$, $r_{2} < p_{1} \wedge p_{2} \wedge q$ and $\vec{w} \in A_{p_1/r_1}(\R^{\mathpzc{d}_{1}}) \times A_{p_2/r_2}(\R^{\mathpzc{d}_{2}})$. If $s > a_1\mathpzc{d}_1(r_1^{-1}-1)_{+} + a_2\mathpzc{d}_2(r_2^{-1}-1)_{+}$, then
    \begin{align*}
    &Y^{(\vec{a};\mathpzc{d})}\left(L_{p_{2}}(\R^{\mathpzc{d}_{2}},w_{2})[[\ell_{q}^{s}(\N)]
    L_{p_{1}}(\R^{\mathpzc{d}_{1}},w_{1})];X\right)  \\ &\qquad= F^{s/a_2}_{p_{2},q}(\R^{\mathpzc{d}_{2}},w_{2};L_{p_{1}}(\R^{\mathpzc{d}_{1}},w_{1};X)) \cap
    L_{p_{2}}(\R^{\mathpzc{d}_{2}},w_{2};B^{s/a_1}_{p_{1},q}(\R^{\mathpzc{d}_{1}},w_{1};X)).
    \end{align*}
\end{enumerate}
\end{example}

\begin{example}\label{IR:ex:ex:cor:thm:IR;comparY&YL;concrete_examples;Fubini;scalar_anisotropies}
Let $\mathpzc{d} \in (\N_{1})^2$ and $\vec{a} \in (0,\infty)^2$.
\begin{enumerate}[(I)]
\item Let $p,q \in (0,\infty)$, $\vec{w} \in A_{\infty}(\R^{\mathpzc{d}_{1}}) \times A_{\infty}(\R^{\mathpzc{d}_{2}})$ and $s \in \R$. Pick $\vec{r} \in (0,\infty)^{2}$ such that $r_{1},r_{2} < p \wedge q$ and $\vec{w} \in A_{p/r_1}(\R^{\mathpzc{d}_{1}}) \times A_{q/r_2}(\R^{\mathpzc{d}_{2}})$. If $s > a_1\mathpzc{d}_1(r_1^{-1}-1)_{+} + a_2\mathpzc{d}_2(r_2^{-1}-1)_{+}$, then
    \begin{align*}
    F^{s,(\vec{a};\mathpzc{d})}_{(p,q),p}(\R^{d},\vec{w};X) &= F^{s/a_2}_{q,p}(\R^{\mathpzc{d}_{2}},w_{2};L_{p}(\R^{\mathpzc{d}_{1}},w_{1};X)) \\
    &\qquad\qquad \cap
    L_{q}(\R^{\mathpzc{d}_{2}},w_{2};B^{s/a_1}_{p,p}(\R^{\mathpzc{d}_{1}},w_{1};X)).
    \end{align*}
\item Let $p \in (0,\infty)$, $q \in (0,\infty]$, $\vec{w} \in A_{\infty}(\R^{\mathpzc{d}_{1}}) \times A_{\infty}(\R^{\mathpzc{d}_{2}})$ and $s \in \R$. Pick $\vec{r} \in (0,\infty)^{2}$ such that $r_{1} < p$, $r_{2} < p \wedge q$ and $\vec{w} \in A_{p/r_1}(\R^{\mathpzc{d}_{1}}) \times A_{q/r_2}(\R^{\mathpzc{d}_{2}})$. If $s > a_1\mathpzc{d}_1(r_1^{-1}-1)_{+} + a_2\mathpzc{d}_2(r_2^{-1}-1)_{+}$, then
    \begin{align*}
    B^{s,(\vec{a};\mathpzc{d})}_{(p,q),q}(\R^{d},\vec{w};X) &=
    B^{s/a_2}_{q,q}(\R^{\mathpzc{d}_{2}},w_{2};L_{p}(\R^{\mathpzc{d}_{1}},w_{1};X)) \\
    &\qquad\qquad \cap
    L_{q}(\R^{\mathpzc{d}_{2}},w_{2};B^{s/a_1}_{p,q}(\R^{\mathpzc{d}_{1}},w_{1};X)).
    \end{align*}
\end{enumerate}
\end{example}

Combining Example~\ref{IR:ex:ex:cor:thm:IR;comparY&YL;concrete_examples;scalar_anisotropies}.\eqref{IR:it:ex:ex:cor:thm:IR;comparY&YL;concrete_examples;I;scalar_anisotropies} together with a randomized Littlewood-Paley decomposition for UMD Banach space-valued Bessel potential spaces and type and cotype considerations (we refer the reader to \cite{Hytonen&Neerven&Veraar&Weis2016_Analyis_in_Banach_Spaces_II} for the notions of type and cotype), we find the following embedding.

\begin{example}\label{IR:ex:ex:cor:thm:IR;comparY&YL;concrete_examples;scalar_anisotropies;(co)type} 
Let $X$ be a UMD Banach space with type $\rho_0 \in [1,2]$ and cotype $\rho_1 \in [2,\infty]$.
Let $\mathpzc{d} \in (\N_{1})^2$, $\vec{a} \in (0,\infty)^2$, $\vec{p} \in (1,\infty)^{2}$, $q \in [\rho_0,\rho_1]$, $\vec{w} \in A_{\infty}(\R^{\mathpzc{d}_{1}}) \times A_{p_2}(\R^{\mathpzc{d}_{2}})$ and $s \in \R$. Pick $r \in (0,\infty)$ such that $r < p_{1} \wedge q$ and $w_1 \in A_{p_1/r}(\R^{\mathpzc{d}_{1}})$. If $s > a_1\mathpzc{d}_1(r^{-1}-1)_{+}$, then
    \begin{align*}
    F^{s,(\vec{a};\mathpzc{d})}_{\vec{p},\rho_0}(\R^{d},\vec{w};X) &\hookrightarrow H^{s/a_2}_{p_{2}}(\R^{\mathpzc{d}_{2}},w_{2};L_{p_{1}}(\R^{\mathpzc{d}_{1}},w_{1};X)) \\
    &\qquad\qquad \cap
    L_{p_{2}}(\R^{\mathpzc{d}_{2}},w_{2};F^{s/a_1}_{p_{1},q}(\R^{\mathpzc{d}_{1}},w_{1};X)) \\
    &\qquad\qquad\qquad \hookrightarrow F^{s,(\vec{a};\mathpzc{d})}_{\vec{p},\rho_1}(\R^{d},\vec{w};X).
    \end{align*}
\end{example}
\begin{proof}
By \cite[Proposition~3.2]{Meyries&Veraar2015_pointwise_multiplication} and the fact that $L_{p_{1}}(\R^{\mathpzc{d}_{1}},w_{1};X)$ is a UMD Banach space (see e.g.\ \cite[Proposition~4.2.15]{Hytonen&Neerven&Veraar&Weis2016_Analyis_in_Banach_Spaces_I}), 
\begin{equation}\label{IR:eq:ex:ex:cor:thm:IR;comparY&YL;concrete_examples;scalar_anisotropies;(co)type;I}
H^{s/a_2}_{p_{2}}(\R^{\mathpzc{d}_{2}},w_{2};L_{p_{1}}(\R^{\mathpzc{d}_{1}},w_{1};X)) = 
F^{s/a_2}_{p_{2},\mathrm{rad}}(\R^{\mathpzc{d}_{2}},w_{2};L_{p_{1}}(\R^{\mathpzc{d}_{1}},w_{1};X)).    
\end{equation}

Let $(\Omega,\mathcal{F},\Prob)$ be a probability space and $(\epsilon_{k})_{k \in \N}$ a Rademacher sequence on $(\Omega,\mathcal{F},\Prob)$.
The space $\mathrm{Rad}_{p}(\N;X)$, where $p \in [1,\infty)$, is defined as the Banach space of sequences $(x_{k})_{k \in \N}$ for which  there is convergence of $\sum_{k=0}^{\infty}\epsilon_{k}x_{k}$ in $L_{p}(\Omega;X)$, endowed with the norm
\[
\norm{(x_{k})_{k \in \N}}_{\mathrm{Rad}_{p}(\N;X)} := \norm{\sum_{k=0}^{\infty}\epsilon_{k}x_{k}}_{L_{p}(\Omega;X)} = \sup_{K \geq 0}\norm{\sum_{k=0}^{K}\epsilon_{k}x_{k}}_{L_{p}(\Omega;X)}.
\]
As a consequence of the Kahane-Khintchine inequalities (see e.g.\ \cite[Proposition~6.3.1]{Hytonen&Neerven&Veraar&Weis2016_Analyis_in_Banach_Spaces_II}), $\mathrm{Rad}_{p}(\N;X)=\mathrm{Rad}_{\tilde{p}}(\N;X)$ with an equivalence of norms for any $p,\tilde{p} \in [1,\infty)$. We put $\mathrm{Rad}(\N;X) = \mathrm{Rad}_{2}(\N;X)$.

With the just introduced notation, the type and cotype assumptions on $X$ can be reformulated as
$$
\ell_{\rho_0}(\N;X) \hookrightarrow \mathrm{Rad}(\N;X) \hookrightarrow \ell_{\rho_1}(\N;X).
$$
Combining this with the identity 
$$
\mathrm{Rad}(\N;L_{p_{1}}(\R^{\mathpzc{d}_{1}},w_{1};X)) = L_{p_{1}}(\R^{\mathpzc{d}_{1}},w_{1};\mathrm{Rad}(\N;X))
$$
obtained from Fubine's theorem and the Kahane-Khintchine inequalities, we find
\begin{equation}\label{IR:eq:ex:ex:cor:thm:IR;comparY&YL;concrete_examples;scalar_anisotropies;(co)type;II}
L_{p_{1}}(\R^{\mathpzc{d}_{1}},w_{1};\ell_{\rho_0}(\N;X)) \hookrightarrow \mathrm{Rad}(\N;L_{p_{1}}(\R^{\mathpzc{d}_{1}},w_{1};X)) \hookrightarrow L_{p_{1}}(\R^{\mathpzc{d}_{1}},w_{1};\ell_{\rho_1}(\N;X)).     
\end{equation}

A combination of \eqref{IR:eq:ex:ex:cor:thm:IR;comparY&YL;concrete_examples;scalar_anisotropies;(co)type;I} and \eqref{IR:eq:ex:ex:cor:thm:IR;comparY&YL;concrete_examples;scalar_anisotropies;(co)type;II} yields
\begin{align*}
\F^{s/a_2}_{p_{2},\rho_1}(\R^{\mathpzc{d}_{2}},w_{2};L_{p_{1}}(\R^{\mathpzc{d}_{1}},w_{1});X) 
&\hookrightarrow H^{s/a_2}_{p_{2}}(\R^{\mathpzc{d}_{2}},w_{2};L_{p_{1}}(\R^{\mathpzc{d}_{1}},w_{1};X))   \\
&\hookrightarrow \F^{s/a_2}_{p_{2},\rho_2}(\R^{\mathpzc{d}_{2}},w_{2};L_{p_{1}}(\R^{\mathpzc{d}_{1}},w_{1});X).
\end{align*}
The desired result now follows from Example~\ref{IR:ex:ex:cor:thm:IR;comparY&YL;concrete_examples;scalar_anisotropies}.\eqref{IR:it:ex:ex:cor:thm:IR;comparY&YL;concrete_examples;I;scalar_anisotropies} and 'monotonicity' of Lizorkin-Triebel spaces in the microsopic parameter.
\end{proof}

\section{Duality}\label{IR:sec:duality}

\begin{definitie}
Let $E \in \mathcal{S}(\varepsilon_{+},\varepsilon_{-},\vec{A},\vec{r},(S,\mathscr{A},\mu))$.
We define $Y^{\vec{A}}(E;X^{*},\sigma(X^{*},X))$ as the space of all $f \in \mathcal{S}'(\R^{d};L_{0}(S;X^{*},\sigma(X^{*},X)))$ which have a representation
\[
f = \sum_{n=0}^{\infty}f_{n} \quad \text{in} \quad \mathcal{S}'(\R^{d};L_{0}(S;X^{*},\sigma(X^{*},X)))
\]
with $(f_{n})_{n} \subset \mathcal{S}'(\R^{d};L_{0}(S;X^{*},\sigma(X^{*},X)))$ satisfying the spectrum condition
\begin{align*}
\supp \hat{f}_{0} & \subset \bar{B}^{\vec{A}}(0,2) \\
\supp \hat{f}_{n} & \subset \bar{B}^{\vec{A}}(0,2^{n+1}) \setminus B^{\vec{A}}(0,2^{n-1}), \qquad n \in \N,
\end{align*}
and $(f_{n})_{n} \in E(X)$. We equip $Y^{\vec{A}}(E;X^{*},\sigma(X^{*},X))$  with the quasinorm
\[
\norm{f}_{Y^{\vec{A}}(E;X^{*},\sigma(X^{*},X))} := \inf \norm{(f_{n})}_{E(X^{*},\sigma(X^{*},X))},
\]
where the infimum is taken over all representations as above.
\end{definitie}

Similarly to Proposition~\ref{IR:prop:LP-decomp_characterization} we have the following Littlewood-Paley decomposition description for $Y^{\vec{A}}(E;X^{*},\sigma(X^{*},X))$:
\begin{prop}\label{IR:prop:thm:duality}
Let $E \in \mathcal{S}(\varepsilon_{+},\varepsilon_{-},\vec{A},\vec{r},(S,\mathscr{A},\mu))$. Let $\varphi=(\varphi_{n})_{n \in \N} \in \Phi^{\vec{A}}(\R^{d})$ with associated sequence of convolution operators $(S_{n})_{n \in \N}$.
Then
\begin{align*}
&Y^{\vec{A}}(E;X^{*},\sigma(X^{*},X)) \\
&\qquad = \left\{ f \in \mathcal{S}'(\R^{d};L_{0}(S;X^{*},\sigma(X^{*},X))) : (S_{n}f)_{n \in \N} \in E(X^{*},\sigma(X^{*},X)) \right\}
\end{align*}
with
\begin{equation}\label{IR:eq:prop:thm:duality;equiv_norms}
\norm{f}_{Y^{\vec{A}}(E;X^{*},\sigma(X^{*},X))} \eqsim \norm{(S_{n}f)_{n \in \N}}_{E(X^{*},\sigma(X^{*},X))}.
\end{equation}
\end{prop}

Using the description from the above proposition it is easy to see that
\[
Y^{\vec{A}}(E;X^{*}) = Y^{\vec{A}}(E;X^{*},\sigma(X^{*},X)) \cap \mathcal{S}'(\R^{d};L_{0}(S;X)).
\]
with an equivalence of quasinorms.

\begin{thm}\label{IR:thm:duality}
Let $E \in \mathcal{S}(\varepsilon_{+},\varepsilon_{-},\vec{A},\vec{r},(S,\mathscr{A},\mu))$ be a Banach function space with an order continuous norm and a weak order unit such that $E^{\times} \in \mathcal{S}(-\varepsilon_{-},-\varepsilon_{+},\vec{A},\vec{1},(S,\mathscr{A},\mu))$.
Assume that there exists a Banach function space $F$ on $S$ with an order continuous norm and a weak order unit such that $\mathcal{S}(\R^{d};F(X)) \stackrel{d}{\hookrightarrow} Y^{\vec{A}}(E;X)$.
Viewing
\begin{align*}
[Y^{\vec{A}}(E;X)]^{*}
&\hookrightarrow  \mathcal{S}'(\R^{d};[F(X)]^{*}) = \mathcal{S}'(\R^{d};F^{\times}(X^{*},\sigma(X^{*},X))) \\
&\qquad\qquad\hookrightarrow \mathcal{S}'(\R^{d};L_{0}(S;X^{*},\sigma(X^{*},X)))
\end{align*}
via the natural pairing, we have
\[
[Y^{\vec{A}}(E;X)]^{*}  = Y^{\vec{A}}(E^{\times};X^{*},\sigma(X^{*},X)).
\]
Consequently, if $X^{*}$ has the Radon-Nikod\'ym property with respect to $\mu$, then
\[
Y^{\vec{A}}(E^{\times};X^{*}) = [Y^{\vec{A}}(E;X)]^{*} \hookrightarrow \mathcal{S}'(\R^{d};F^{\times}(X^{*}))
\hookrightarrow \mathcal{S}'(\R^{d};L_{0}(S;X^{*})).
\]
\end{thm}

\begin{example}\label{IR:ex:thm:duality}
Let us consider the notation introduced in Example~\ref{IR:ex:prop:LP-decomp_characterization}. For a weight vector $\vec{w}$ and $\vec{p} \in (1,\infty)^\ell$ we define the $\vec{p}$-dual weight of $\vec{w}$ by $\vec{w}'_{\vec{p}} := (w_1^{-\frac{1}{p_1-1}},\ldots,w_\ell^{-\frac{1}{p_\ell-1}})$ and we write $\vec{p}'$ for the H\"older conjugate vector of $\vec{p}$.
\begin{enumerate}[(i)]
    \item Let $\vec{p} \in (1,\infty)^\ell$, $q \in [1,\infty)$, $\vec{w} \in \prod_{j=1}^{\ell}A_{p_j}(\R^{\mathpzc{d}_j},A_j)$ and $s \in \R$. Then
    $$
    [F^{s,\vec{A}}_{\vec{p},q}(\R^d,\vec{w};X)]^{*} = F^{-s,\vec{A}}_{\vec{p}',q'}(\R^d,\vec{w}'_{\vec{p}};X^{*}).
    $$
    \item Let $\vec{p} \in (1,\infty)^\ell$, $q \in [1,\infty)$, $\vec{w} \in \prod_{j=1}^{\ell}A_{p_j}(\R^{\mathpzc{d}_j},A_j)$ and $s \in \R$. Then
    $$
    [B^{s,\vec{A}}_{\vec{p},q}(\R^d,\vec{w};X)]^{*} = B^{-s,\vec{A}}_{\vec{p}',q'}(\R^d,\vec{w}'_{\vec{p}};X^{*}).
    $$
    \item Let $F$ be a UMD Banach function space, $\vec{p} \in (1,\infty)^\ell$, $q \in [1,\infty)$, $\vec{w} \in \prod_{j=1}^{\ell}A_{p_j}(\R^{\mathpzc{d}_j},A_j)$ and $s \in \R$. If $X^{*}$ has the Radon-Nikod\'ym property with respect to $\mu$, then
    $$
    [\F^{s,\vec{A}}_{\vec{p},q}(\R^d,\vec{w};F;X)]^{*} = \F^{-s,\vec{A}}_{\vec{p}',q'}(\R^d,\vec{w}'_{\vec{p}};F^{\times};X^{*}).
    $$
\end{enumerate}
\end{example}

Let $E \in \mathcal{S}(\varepsilon_{+},\varepsilon_{-},\vec{A},\vec{1},(S,\mathscr{A},\mu))$ be a Banach function space.
By Remark~\ref{IR:rmk:lemma:HN_L1.1.4} we then have
\[
E_{i} \hookrightarrow E^{\vec{A}}_{\otimes}(B^{1,w_{\vec{A},\vec{1}}}_{\vec{A}}) \hookrightarrow E^{\vec{A}}_{\otimes}[B^{1,w_{\vec{A},\vec{1}}}_{\vec{A}}],
\]
from which it follows that
\begin{align*}
E_{i}(X^{*},\sigma(X^{*},X)) &\hookrightarrow E^{\vec{A}}_{\otimes}[B^{1,w_{\vec{A},\vec{1}}}_{\vec{A}}](X^{*},\sigma(X^{*},X))
\hookrightarrow \mathcal{S}'(\R^{d};E^{\vec{A}}_{\otimes}(X^{*},\sigma(X^{*},X))) \\
&\qquad\qquad \hookrightarrow \mathcal{S}'(\R^{d};L_{0}(S;X^{*},\sigma(X^{*},X))).
\end{align*}

\begin{lemma}\label{IR:lemma:thm:duality}
Let $E \in \mathcal{S}(\varepsilon_{+},\varepsilon_{-},\vec{A},\vec{1},(S,\mathscr{A},\mu))$ be a Banach function space and let $Z$ be a Banach space with $Z \hookrightarrow L_{0}(S;X^{*},\sigma(X^{*},X))$.
Let $\varphi=(\varphi_{n})_{n \in \N} \in \Phi^{\vec{A}}(\R^{d})$ with associated sequence of convolution operators $(S_{n})_{n \in \N}$ be such that
\begin{equation}\label{IR:eq:prop:thm:duality:Fourier_support}
\left\{\begin{array}{ll}
\supp \hat{\varphi}_{0} & \subset \bar{B}^{\vec{A}}(0,2) \\
\supp \hat{\varphi}_{n} & \subset \bar{B}^{\vec{A}}(0,2^{n+1}) \setminus B^{\vec{A}}(0,2^{n-1}), \qquad n \in \N.
\end{array}\right.
\end{equation}
Then
\begin{align*}
&Y^{\vec{A}}(E;X^{*},\sigma(X^{*},X)) \cap \Schw'(\R^{d};Z) \\
&\quad = \left\{ f \in \mathcal{S}'(\R^{d};Z) : \exists (f_{k})_{k} \in E(X^{*},\sigma(X^{*},X)), f = \sum_{k=0}^{\infty}S_{k}f_{k} \:\text{in}\: \mathcal{S}'(\R^{d};Z) \right\}
\end{align*}
with
\[
\norm{f}_{Y^{\vec{A}}(E;X^{*},\sigma(X^{*},X))} \eqsim \inf\norm{(f_{k})_{k}}_{E(X^{*},\sigma(X^{*},X))}.
\]
\end{lemma}
\begin{proof}
Given $f \in Y^{\vec{A}}(E;X^{*},\sigma(X^{*},X)) \cap \Schw'(\R^{d};Z)$, let $f_{k} := T_{k}f$, where $T_{k}:=S_{k-1}+S_{k}+S_{k+1}$. Then $S_{k}f_{k}=S_{k}f$, so $f = \sum_{k=0}^{\infty}S_{k}f_{k}$ in $\Schw'(\R^{d};Z)$.
From
\[
|\ip{f_{k}}{x}| = |T_{k}\ip{S_{k}f}{x}| \lesssim M^{\vec{A}}\ip{S_{k}f}{x} \leq M^{\vec{A}}\vartheta(S_{k}f), \qquad x \in B_{X},
\]
it follows that $\vartheta(f_{k}) \lesssim M^{\vec{A}}\vartheta(S_{k}f)$.
Using that $M^{\vec{A}}$ is bounded on $E$, we find
\[
\norm{ (f_{k})_{k} }_{E(X^{*},\sigma(X^{*},X))} \lesssim \norm{ (S_{k}f)_{k} }_{E(X^{*},\sigma(X^{*},X))} \stackrel{\eqref{IR:eq:prop:thm:duality;equiv_norms}}{\eqsim} \norm{f}_{Y^{\vec{A}}(E;X^{*},\sigma(X^{*},X))}.
\]

For the converse, let $f = \sum_{k=0}^{\infty}S_{k}f_{k}$ in $\mathcal{S}'(\R^{d};Z)$ with $(f_{k})_{k} \in E(X^{*},\sigma(X^{*},X))$.
Then
\[
|\ip{S_{k}f_{k}}{x}| = |S_{k}\ip{f_{k}}{x}| \lesssim M^{\vec{A}}\ip{f_{k}}{x} \leq M^{\vec{A}}\vartheta(f_{k}), \qquad x \in B_{X},
\]
so that $\vartheta(S_{k}f) \lesssim M^{\vec{A}}\vartheta(f_{k})$.
In view of
\[
f = \sum_{k=0}^{\infty}S_{k}f_{k} \quad \text \quad \Schw'(\R^{d};Z) \hookrightarrow \mathcal{S}'(\R^{d};L_{0}(S;X^{*},\sigma(X^{*},X))),
\]
\eqref{IR:eq:prop:thm:duality:Fourier_support} and the boundedness of $M^{\vec{A}}$ on $E$, it follows that $f \in Y^{\vec{A}}(E;X^{*},\sigma(X^{*},X))$ with
\[
\norm{f}_{Y^{\vec{A}}(E;X^{*},\sigma(X^{*},X))} \lesssim \norm{ (S_{k}f_{k})_{k} }_{E(X^{*},\sigma(X^{*},X))} \lesssim \norm{ (f_{k})_{k} }_{E(X^{*},\sigma(X^{*},X))}. \qedhere
\]
\end{proof}

\begin{proof}[Proof of Theorem~\ref{IR:thm:duality}]
By assumption and Proposition~\ref{IR:prop:incl_S'},
\[
\Schw(\R^{d};F(X)) \hookrightarrow Y^{\vec{A}}(E;X) \hookrightarrow \Schw'(\R^{d};E^{\vec{A}}_{\otimes}(X)),
\]
from which it follows that $F \hookrightarrow E^{\vec{A}}_{\otimes}$, implying in turn that $[E^{\vec{A}}_{\otimes}]^{\times} \hookrightarrow F^{\times}$. On the other hand it holds that $[E^{\times}]^{\vec{A}}_{\otimes} \hookrightarrow [E^{\vec{A}}_{\otimes}]^{\times}$.
Therefore, $[E^{\times}]^{\vec{A}}_{\otimes} \hookrightarrow F^{\times}$.
By (a variant of) Proposition~\ref{IR:prop:incl_S'} we thus obtain
\begin{equation}\label{IR:eq:thm:duality;incl}
Y^{\vec{A}}(E^{\times};X^{*},\sigma(X^{*},X)) \hookrightarrow \Schw'(\R^{d};[E^{\times}]^{\vec{A}}_{\otimes}(X^{*},\sigma(X^{*},X))) \hookrightarrow \Schw'(\R^{d};F^{\times}(X^{*},\sigma(X^{*},X))).
\end{equation}
So we can use Lemma~\ref{IR:lemma:thm:duality} with $Z= F^{\times}(X^{*},\sigma(X^{*},X))$ to describe $Y^{\vec{A}}(E^{\times};X^{*},\sigma(X^{*},X))$.

Let $(S_k)_{k \in \N}$ be as in Lemma~\ref{IR:lemma:thm:duality} and equip $Y^{\vec{A}}(E;X)$ with the corresponding equivalent norm from Proposition~\ref{IR:prop:LP-decomp_characterization}.
Then
\[
\iota : Y^{\vec{A}}(E;X) \longra E(X),\, f \mapsto (S_{k}f)_{k}
\]
defines an isometric linear mapping.
By order continuity of $E$ and $F$, there are the natural identifications
\[
[E(X)]^{*} = E^{\times}(X^{*},\sigma(X^{*},X)) \qquad \text{and} \qquad [F(X)]^{*} = F^{\times}(X^{*},\sigma(X^{*},X)).
\]
As $\mathcal{S}(\R^{d};F(X)) \stackrel{d}{\hookrightarrow} Y^{\vec{A}}(E;X)$, we may thus view
\[
[Y^{\vec{A}}(E;X)]^{*} \hookrightarrow \mathcal{S}'(\R^{d};[F(X)]^{*}) = \mathcal{S}'(\R^{d};F^{\times}(X^{*},\sigma(X^{*},X))).
\]
Denoting the adjoint of $\iota$ by $j$, we thus obtain the following commutative diagram:
\begin{equation*}
\begin{tikzcd}
E^{\times}(X^{*},\sigma(X^{*},X)) \arrow[r, "T"] \arrow[d] \arrow[dr,"j"]
& \mathcal{S}'(\R^{d};F^{\times}(X^{*},\sigma(X^{*},X))) \arrow[d,hookleftarrow] \\
\faktor{E^{\times}(X^{*},\sigma(X^{*},X))}{\ker j} \arrow[r,"\tilde{j}","\simeq"']
&  \left[Y^{\vec{A}}(E;X)\right]^{*} 
\end{tikzcd}.    
\end{equation*}
Here $T$ is explicitly given by
\[
T(f_{k})_{k} = \sum_{k=0}^{\infty}S_{k}f_{k} \quad \text{in} \quad \mathcal{S}'(\R^{d};F^{\times}(X^{*},\sigma(X^{*},X))),
\]
which can be seen by testing against $\phi \in \mathcal{S}(\R^{d};F(X))$:
\begin{align*}
\ip{T(f_{k})_{k}}{\phi} = \ip{(f_{k})_{k}}{\iota\phi} = \ip{(f_{k})_{k}}{(S_{k}\phi)_{k}} = \sum_{k=0}^{\infty}\ip{f_{k}}{S_{k}\phi} = \sum_{k=0}^{\infty}\ip{S_{k}f_{k}}{\phi}.
\end{align*}
The desired result follows by an application of Lemma~\ref{IR:lemma:thm:duality} with $Z= F^{\times}(X^{*},\sigma(X^{*},X))$ (recall \eqref{IR:eq:thm:duality;incl}).
\end{proof}

\section{A Sum Representation}\label{IR:sec:sum_repr}

In this section we combine the intersection representation for $Y^{\vec{A}}(E;X)$ from Theorem~\ref{IR:thm:IR_Ap-case} and the duality result Theorem~\ref{IR:thm:duality} with the following fact on duality for intersection spaces: given an interpolation couple of Banach spaces $(Y,Z)$ for which $Y \cap Z$ is dense in both $Y$ and $Z$, it holds that $(X^{*},Y^{*})$ is an interpolation couple of Banach space and
\begin{equation}\label{IR:eq:dual_of_intersection}
[Y \cap Z]^{*} = Y^{*}+Z^{*}, \qquad [X+Y]^{*} = X^{*} \cap Y^{*},
\end{equation}
hold isometrically under the natural identifications (see \cite[Theorem~I.3.1]{Krein&Petun&Semenov1982}).

We let the notation be as in Section~\ref{IR:sec:IR}.

\begin{cor}\label{IR:cor:thm:IR_Ap-case;SR}
Let $E \in \mathcal{S}(\varepsilon_{+},\varepsilon_{-},\vec{A},\vec{1},(S,\mathscr{A},\mu))$ be a Banach function space such that $E^{\times}$ has an order continuous norm and a weak order unit and $E^{\times} \in \mathcal{S}(-\varepsilon_{-},-\varepsilon_{+},\vec{A},\vec{r},(S,\mathscr{A},\mu))$ with $\varepsilon_{+},\varepsilon_{-} < 0$. Suppose that $X$ is reflexive.
Let $F$ Banach function space on $S$ with an order continuous norm such that $\mathcal{S}(\R^{d};F(X)) \stackrel{d}{\hookrightarrow} Y^{\vec{A}}(E^{\times};X)$.
Let $\{J_{1},\ldots,J_{L}\}$ be a partition of $\{1,\ldots,\ell\}$ and, for each $l \in \{1,\ldots,L\}$, let $F_{l}$ be a Banach function space on $S_{J_{l}}$ with an order continuous norm and a weak order unit such that 
$$
\mathcal{S}(\R^{d};F(X)) \stackrel{d}{\hookrightarrow} \mathcal{S}(\R^{d-d_{J_{l}}};F_{l}(X)) \stackrel{d}{\hookrightarrow} Y^{\vec{A}_{J_{l}}}(E^{\times}_{[\mathpzc{d};J_{l}]};X).
$$
Then
\[
Y^{\vec{A}}(E;X) = \foo_{l=1}^{L}Y^{\vec{A}_{J_{l}}}(E_{[\mathpzc{d};J_{l}]};X)
\]
with an equivalence of norms.
\end{cor}
\begin{proof}
As $E$ has the Fatou property, $E=(E^{\times})^{\times}$.
The desired result thus follows from a combination of Theorem~\ref{IR:thm:IR_Ap-case}, Theorem~\ref{IR:thm:duality}, \eqref{IR:eq:dual_of_intersection}, the reflexivity of $X$ and the fact that the Radon--Nikod\'ym property is implied by reflexivity (see \cite[Theorem~1.3.21]{Hytonen&Neerven&Veraar&Weis2016_Analyis_in_Banach_Spaces_II}).
\end{proof}

\appendix

\section{Some Maximal Function Inequalities}

Suppose that $\R^{d}$ is $\mathpzc{d}$-decomposed with $\mathpzc{d} \in (\N_{1})^{\ell}$ and let $\vec{A}=(A_{1},\ldots,A_{\ell})$ be a $\mathpzc{d}$-anisotropy.

\begin{lemma}[Anisotropic Peetre's inequality]\label{IR:appendix:lemma:Peetre-Feffeman-Stein}
Let $X$ be a Banach space, $\vec{r} \in (0,\infty)^{\ell}$, $K \subset \R^{d}$ a compact set and $N \in \N$. For all $\alpha \in \N^{n}$ with $|\alpha| \leq N$ and $f \in \mathcal{S}'(\R^{d};X)$ with $\supp(\hat{f}) \subset K$, there is the pointwise estimate
\begin{align*}
\sup_{z \in\R^{d}} \frac{\norm{D^{\alpha}f(x+z)}_{X}}
{\prod_{j=1}^{\ell}(1+\rho_{A_{j}}(z_{j}))^{\mathrm{tr}(A_{j})/r_{j}}}
&\lesssim
\sup_{z \in\R^{d}} \frac{\norm{f(x+z)}_{X}}{\prod_{j=1}^{\ell}(1+\rho_{A_{j}}(z_{j}))^{\mathrm{tr}(A_{j})/r_{j}}} \\
&\lesssim \left[M^{\vec{A}}_{\vec{r}}(\norm{f}_{X})\right](x),\qquad x \in \R^{d}.
\end{align*}
\end{lemma}
\begin{proof}
This can be obtained by combining the proof of \cite[Proposition~3.11]{JS_traces} for the case $\mathpzc{d}=\vec{1}$ with the proof of \cite[Lemma~3.4]{Bownik&Ho2006} for the case $\ell=1$.
Although it get quite technical, we have decided to not provide the details.
\end{proof}

For $f \in \mathcal{F}^{-1}\mathcal{E}'(\R^{d};X)$, $\vec{r} \in (0,\infty)^{\ell}$, $\vec{R} \in (0,\infty)^{\ell}$ we define the \emph{maximal fuction of Peetre-Fefferman-Stein type} $f^{*}(\vec{A},\vec{r},\vec{R};\,\cdot\,)$ by
\[
f^{*}(\vec{A},\vec{r},\vec{R};x) :=
\sup_{z \in\R^{d}} \frac{\norm{f(x+z)}_{X}}
{\prod_{j=1}^{\ell}(1+R_{j}\rho_{A_{j}}(z_{j}))^{\mathrm{tr}(A_{j})/r_{j}}}.
\]
\begin{corollary}\label{IR:appendix:cor:lemma:Peetre-Feffeman-Stein}
Let $X$ be a Banach space and $\vec{r} \in (0,\infty)^{\ell}$. For all $f \in \mathcal{S}'(\R^{d};X)$ and $\vec{R} \in (0,\infty)^{\ell}$ with $\supp(\hat{f}) \subset B^{\vec{A}}(0,\vec{R})$, there is the pointwise estimate
\[
f^{*}(\vec{A},\vec{r},\vec{R};x) \lesssim_{\vec{A},\vec{r}} \left[M^{\vec{A}}_{\vec{r}}(\norm{f}_{X})\right](x),\qquad x \in \R^{d}.
\]
\end{corollary}
\begin{proof}
By a dilation argument it suffices to consider the case $\vec{R}=\vec{1}$, which is contained in Lemma~\ref{IR:appendix:lemma:Peetre-Feffeman-Stein}.
\end{proof}

\begin{lemma}\label{IR:appendix:lemma:master_thesis_Prop.3.4.8_pointwise_ineq}
Let $X$ and $Y$ be Banach spaces. For all $(M_{n})_{n \in \N} \subset \mathcal{F}L^{1}(\R^{d};\mathcal{B}(X,Y))$, $(\vec{R}^{(n)})_{n \in \N} \subset (0,\infty)^{\ell}$, $\vec{\lambda} \in (0,\infty)^{\ell}$, $c \in [1,\infty)$ and $(f_{n})_{n \in \N} \subset \mathcal{F}^{-1}\mathcal{E}'(\R^{d};X)$, there is the pointwise estimate
\begin{align*}
&\normb{[\mathcal{F}(M_{n}\hat{f}_{n})](x)}_{Y} \\
&\qquad\lesssim c^{\sum_{j=1}^{\ell}\lambda_{j}}\sup_{k \in \N}\int_{\R^{d}}\norm{\check{M}_{n}(\vec{A}_{\vec{R^{(n)}}}y)}_{\mathcal{B}(X,Y)}
\prod_{j=1}^{\ell}(1+\rho_{A_{j}}(y_{j}))^{\lambda_{j}}\,dy \\
&\qquad\qquad\cdot\quad \sup_{z \in \R^{d}}\frac{\norm{f_{n}(x+z)}_{X}}{\prod_{j=1}^{\ell}(1+cR^{(n)}_{j}\rho_{A_{j}}(y_{j}))^{\lambda_{j}}}.
\end{align*}
\end{lemma}
\begin{proof}
This can be shown as the pointwise estimate in the proof of \cite[Proposition~3.4.8]{Lindemulder_master-thesis}, which was in turn based on \cite[Proposition~2.4]{Meyries&Veraar2012_sharp_embedding}.
\end{proof}

The following proposition is an extension of \cite[Proposition~3.13]{JS_traces} to our setting, which is in turn a version of the pointwise estimate of pseudo-differential operators due to Marschall~\cite{Marschall1996}. In order to state it, we first need to introduce the anisotropic mixed-norm homogeneous Besov space $\dot{B}^{s,\vec{A}}_{\vec{p},q}(\R^{d};Z)$.

Let $Z$ be a Banach space, $\vec{p} \in (1,\infty)^{\ell}$, $q \in (0,\infty]$ and $s \in \R$.
Fix $(\phi_{k})_{k \in \Z} \subset \mathcal{S}(\R^{d})$ that satisfies $\hat{\phi}_{k} = \hat{\psi}(\vec{A}_{2^{-k}}\,\cdot)-\hat{\psi}(\vec{A}_{2^{-(k+)}}\,\cdot\,)$ for some $\psi \in \mathcal{F}C^{\infty}_{c}(\R^{d})$ with $\hat{\psi} \equiv 1$ on a neighbourhood of $0$.
Then $\dot{B}^{s,\vec{A}}_{\vec{p},q}(\R^{d};Z)$ is defined as the space of all $f \in [\mathcal{S}'/\mathcal{P}](\R^{d};Z)$ for which
\[
\norm{f}_{\dot{B}^{s,\vec{A}}_{\vec{p},q}(\R^{d};Z)} := \normb{(2^{sk}\phi_{k}*f)_{k \in \Z}}_{\ell_{q}(\Z)[L_{\vec{p},\mathpzc{d}}(\R^{d})](Z)} < \infty.
\]

\begin{proposition}\label{IR:appendix:prop:Marschall}
Let $X$ and $Y$ be Banach spaces and $\vec{r} \in (0,1]^{\ell}$. Put $\tau := \vec{r}_{\min} \in (0,1]$.
For all $b \in \mathcal{S}(\R^{d};\mathcal{B}(X,Y))$, $u \in \mathcal{S}'(\R^{d};X)$, $c \in (0,\infty)$ and $R \in [1,\infty)$ with $\supp(b) \subset B^{\vec{A}}(0,c)$ and $\supp(\hat{u}) \subset
B^{\vec{A}}(0,cR)$, there is the pointwise estimate
\[
\norm{b(D)u(x)}_{Y} \lesssim_{\vec{A},\vec{r}} (cR)^{\mathrm{tr}(\vec{A})\bcdot(\vec{r}^{-1}-\vec{1})}
\norm{b}_{\dot{B}^{\mathrm{tr}(\vec{A})\bcdot\vec{r}^{-1},\vec{A}}_{1,\tau}(\R^{d};\mathcal{B}(X,Y))}\,
\left[M^{\vec{A}}_{\vec{r}}(\norm{u}_{X})\right](x)
\]
for each $x \in \R^{d}$.
\end{proposition}

In the proof of Proposition~\ref{IR:appendix:prop:Marschall} we will use the following lemma.
\begin{lemma}\label{IR:appendix:lemma:prop:Marschall}
Let $X$ be a Banach space and $\vec{p},\vec{q} \in (0,\infty)^{\ell}$ with $\vec{p} \leq \vec{q}$.
For every $f \in \mathcal{S}'(\R^{d};X)$ and $\vec{R} \in (0,\infty)^{\ell}$ with $\supp(\hat{f}) \subset B^{\vec{A}}(0,\vec{R})$,
\[
\norm{f}_{L_{\vec{q},\mathpzc{d}}(\R^{d};X)} \lesssim_{\vec{p},\vec{q},\mathpzc{d}} \prod_{j=1}^{\ell}R_{j}^{\mathrm{tr}(A_{j})(\frac{1}{p_{j}}-\frac{1}{q_{j}})}
\norm{f}_{L_{\vec{p},\mathpzc{d}}(\R^{d};X)}
\]
\end{lemma}
\begin{proof}
By a scaling argument we may restrict ourselves to the case $\vec{R}=\vec{1}$. Now pick $\phi \in \mathcal{S}(\R^{d})$ with $\hat{\phi} \equiv 1$ on $B^{\vec{A}}(0,\vec{1})$. Then $f=\phi*f$ and the desired inequality follows from an iterated use of Young's inequality for convolutions.
\end{proof}

\begin{proof}[Proof of Proposition~\ref{IR:appendix:prop:Marschall}]
It holds that
\[
b(D)u(x) = \int_{\R^{d}}\check{b}(y)u(x-y)\,dy, \qquad x \in \R^{d}.
\]
For fixed $x \in \R^{d}$, by the quasi-triangle inequality for $\rho_{\vec{A}}$ (with constant $c_{\vec{A}}$),
\[
\supp(\mathcal{F}[y \mapsto \check{b}(y)u(x-y)]) \subset B_{\vec{A}}(0,c) + B_{\vec{A}}(0,cR) \subset
B_{\vec{A}}(0,c_{\vec{A}}(R+1)c).
\]
Therefore,
\begin{align}
\norm{b(D)u(x)}_{Y}
& \leq \norm{y \mapsto \check{b}(y)u(x-y) }_{L_{1}(\R^{d})} \nonumber \\
&\lesssim (c_{\vec{A}}(R+1)c)^{\sum_{j=1}^{\ell}\mathrm{tr}(A_{j})(\frac{1}{r_{j}}-1)}
\norm{y \mapsto \check{b}(y)u(x-y) }_{L_{\vec{r},\mathpzc{d}}(\R^{d})} \nonumber \\
&\lesssim (Rc)^{\sum_{j=1}^{\ell}\mathrm{tr}(A_{j})(\frac{1}{r_{j}}-1)}
\norm{y \mapsto \check{b}(y)u(x-y) }_{L_{\vec{r},\mathpzc{d}}(\R^{d})}, \label{IR:appendix:eq:prop:Marschall;proof;1}
\end{align}
where we used Lemma~\ref{IR:appendix:lemma:prop:Marschall} for the second estimate.

Let $(\phi_{k})_{k \in \Z}$ be as in the definition of the anisotropic homogeneous Besov space $\dot{B}^{s,\vec{A}}_{\vec{p},q}$ as given preceding the proposition. Then $\sum_{k=-\infty}^{\infty}\hat{\phi}_{k}(-\,\cdot\,) = 1$ on $\R^{d} \setminus \{0\}$, so that
\begin{equation}\label{IR:appendix:eq:prop:Marschall;proof;2}
\norm{\check{b}\,u(x-\,\cdot\,) }_{L_{\vec{r},\mathpzc{d}}(\R^{d})} \leq
\left( \sum_{k \in \Z}\norm{\hat{\phi}_{k}(-\,\cdot\,)\,\check{b}\,u(x-\,\cdot\,)}_{L_{\vec{r},\mathpzc{d}}(\R^{d})}^{\tau} \right)^{1/\tau}.
\end{equation}

Since
\begin{align*}
\sup_{y \in \R^{d}}\norm{\hat{\phi}_{k}(-y)\check{b}(y)}_{\mathcal{B}(X,Y)}
&\leq\norm{\mathcal{F}^{-1}[\hat{\phi}_{k}(-\,\cdot\,)\,\check{b}]}_{L_{1}(\R^{d};\mathcal{B}(X,Y))} \\
&= (2\pi)^{-d}\norm{\mathcal{F}^{-1}[\hat{\phi}_{k}\hat{b}]}_{L_{1}(\R^{d};\mathcal{B}(X,Y))}
\end{align*}
and $\supp(\hat{\phi}_{k}) \subset B^{\vec{A}}(0,2^{k+1})$, it follows from a combination of \eqref{IR:appendix:eq:prop:Marschall;proof;1} and \eqref{IR:appendix:eq:prop:Marschall;proof;2} that
\begin{align*}
\norm{b(D)u(x)}_{Y}
&\lesssim (Rc)^{\sum_{j=1}^{\ell}\mathrm{tr}(A_{j})(\frac{1}{r_{j}}-1)}\left( \sum_{k \in \Z}\norm{\hat{\phi}_{k}(-\,\cdot\,)\,\check{b}\,u(x-\,\cdot\,)}_{L_{\vec{r},\mathpzc{d}}(\R^{d})}^{\tau} \right)^{1/\tau} \\
&\lesssim (Rc)^{\sum_{j=1}^{\ell}\mathrm{tr}(A_{j})(\frac{1}{r_{j}}-1)}\left(
\sum_{k \in \Z}\left[2^{k\sum_{j=1}^{\ell}\mathrm{tr}(A_{j})\frac{1}{r_{j}}}
\norm{\mathcal{F}^{-1}[\hat{\phi}_{k}\hat{b}]}_{L_{1}}\right]^{\tau}\right)^{1/\tau} \\
&\qquad\qquad \sup_{k \in \Z}2^{-(k+1)\mathrm{tr}(A_{j})\frac{1}{r_{j}}}
\norm{1_{B^{\vec{A}}(0,2^{k+1})}u(x-\,\cdot\,)}_{L_{\vec{r},\mathpzc{d}}(\R^{d})} \\
&\leq (Rc)^{\sum_{j=1}^{\ell}\mathrm{tr}(A_{j})(\frac{1}{r_{j}}-1)}
\norm{b}_{\dot{B}^{\sum_{j=1}^{\ell}\mathrm{tr}(A_{j})\frac{1}{r_{j}},\vec{A}}_{1,\tau}(\R^{d};\mathcal{B}(X,Y))}\,
\left[M^{\vec{A}}_{\vec{r}}(\norm{u}_{X})\right](x).
\end{align*}
\end{proof}

\begin{corollary}\label{IR:appendix:cor:prop:Marschall}
Let $X$ and $Y$ be Banach spaces, $\vec{r} \in (0,1]^{\ell}$ and $\psi \in C^{\infty}_{c}(\R^{d};\mathcal{B}(X,Y))$.
Put $\psi_{k} := \psi(\vec{A}_{2^{-k}}\,\cdot\,)$ for each $k \in \N$.
Then, for all $(f_{k})_{k \in \N} \subset \mathcal{S}'(\R^{d};X)$ with $\supp\hat{f}_{k} \subset B^{\vec{A}}(0,r2^{k})$ for some $r \in [1,\infty)$, there is the pointwise estimate
\[
\norm{\psi_{k}(D)f_{k}(x)}_{Y} \lesssim r^{\mathrm{tr}(\vec{A})\bcdot(\vec{r}^{-1}-\vec{1})}
\left[M^{\vec{A}}_{\vec{r}}(\norm{f_{k}}_{X})\right](x), \qquad x \in \R^{d}.
\]
\end{corollary}
\begin{proof}
Set $\sigma = \mathrm{tr}(\vec{A}) \bcdot \vec{r}^{-1} =  \sum_{j=1}^{\ell}\mathrm{tr}(A_{j})\frac{1}{r_{j}}$.
Let $c \in [1,\infty)$ be such that $\supp(\psi) \subset B^{\vec{A}}(0,c)$.
Applying Proposition~\ref{IR:appendix:prop:Marschall} to $b=\psi_{k}$, $u=f_{k}$ and $R=r2^{k}$,
we find that
\[
\norm{\psi_{k}(D)f_{k}(x)}_{Y} \lesssim (cr2^{k})^{\mathrm{tr}(\vec{A})\bcdot(\vec{r}^{-1}-\vec{1})}
\norm{\psi_{k}}_{\dot{B}^{\sigma,\vec{A}}_{1,\tau}(\R^{d})}\,
\left[M^{\vec{A}}_{\vec{r}}(\norm{f_{k}}_{X})\right](x).
\]
Observing that
\[
\norm{\psi_{k}}_{\dot{B}^{\sigma,\vec{A}}_{1,\tau}(\R^{d})}
= 2^{-k\mathrm{tr}(\vec{A})\bcdot(\vec{r}^{-1}-\vec{1})}
\norm{\psi}_{\dot{B}^{\sigma,\vec{A}}_{1,\tau}(\R^{d})},
\]
we obtain the desired estimate.
\end{proof}

\section*{Acknowledgements.} The author would like to thank the anonymous referees for their valuable feedback.

\def\cprime{$'$} \def\cprime{$'$} \def\cprime{$'$} \def\cprime{$'$}


\begin{thebibliography}{10}

\bibitem{Amann09}
H.~Amann.
\newblock {\em Anisotropic function spaces and maximal regularity for parabolic
  problems. {P}art 1}, volume~6 of {\em Jind\v rich Ne\v cas Center for
  Mathematical Modeling Lecture Notes}.
\newblock Matfyzpress, Prague, 2009.
\newblock Function spaces.

\bibitem{Amann2003_Vector-valued_distributions}
Herbert Amann.
\newblock {\em Linear and quasilinear parabolic problems. {V}ol. {II}}, volume
  106 of {\em Monographs in Mathematics}.
\newblock Birkh\"{a}user/Springer, Cham, 2019.
\newblock Function spaces.

\bibitem{Aoki1942}
T.~Aoki.
\newblock Locally bounded linear topological spaces.
\newblock {\em Proc. Imp. Acad. Tokyo}, 18:588--594, 1942.

\bibitem{Barrios&Betancor2011}
B.~Barrios and J.~J. Betancor.
\newblock Characterizations of anisotropic {B}esov spaces.
\newblock {\em Math. Nachr.}, 284(14-15):1796--1819, 2011.

\bibitem{Berkolaiko1985}
M~Z. Berkola\u\i~ko.
\newblock Theorems on traces on coordinate subspaces for some spaces of
  differentiable functions with anisotropic mixed norm.
\newblock {\em Dokl. Akad. Nauk SSSR}, 282(5):1042--1046, 1985.

\bibitem{Berkolaiko1987}
M.~Z. Berkola\u\i~ko.
\newblock Traces of functions in generalized {S}obolev spaces with a mixed norm
  on an arbitrary coordinate subspace. {II}.
\newblock {\em Trudy Inst. Mat. (Novosibirsk)}, 9(Issled. Geom. ``v tselom'' i
  Mat. Anal.):34--41, 206, 1987.

\bibitem{Beurling1964}
A.~Beurling.
\newblock Construction and analysis of some convolution algebras.
\newblock {\em Ann. Inst. Fourier (Grenoble)}, 14(fasc. 2):1--32, 1964.

\bibitem{Bourgain1984}
J.~Bourgain.
\newblock Extension of a result of {B}enedek, {C}alder\'{o}n and {P}anzone.
\newblock {\em Ark. Mat.}, 22(1):91--95, 1984.

\bibitem{Bownik2003}
M.~Bownik.
\newblock Anisotropic {H}ardy spaces and wavelets.
\newblock {\em Mem. Amer. Math. Soc.}, 164(781):vi+122, 2003.

\bibitem{Bownik&Ho2006}
M.~Bownik and K.-P. Ho.
\newblock Atomic and molecular decompositions of anisotropic
  {T}riebel-{L}izorkin spaces.
\newblock {\em Trans. Amer. Math. Soc.}, 358(4):1469--1510, 2006.

\bibitem{Bui1982}
H.-Q. Bui.
\newblock Weighted {B}esov and {T}riebel spaces: interpolation by the real
  method.
\newblock {\em Hiroshima Math. J.}, 12(3):581--605, 1982.

\bibitem{Bui1994}
H.-Q. Bui.
\newblock Remark on the characterization of weighted {B}esov spaces via
  temperatures.
\newblock {\em Hiroshima Math. J.}, 24(3):647--655, 1994.

\bibitem{Bui&Paluszynski&Taibleson1996}
H.-Q. Bui, M.~Paluszy\'{n}ski, and M.~H. Taibleson.
\newblock A maximal function characterization of weighted {B}esov-{L}ipschitz
  and {T}riebel-{L}izorkin spaces.
\newblock {\em Studia Math.}, 119(3):219--246, 1996.

\bibitem{Dappa&Trebels1989}
H.~Dappa and W.~Trebels.
\newblock On anisotropic {B}esov and {B}essel potential spaces.
\newblock In {\em Approximation and function spaces ({W}arsaw, 1986)},
  volume~22 of {\em Banach Center Publ.}, pages 69--87. PWN, Warsaw, 1989.

\bibitem{Denk&Hieber&Pruess2007}
R.~Denk, M.~Hieber, and J.~Pr\"uss.
\newblock Optimal {$L^p$}-{$L^q$}-estimates for parabolic boundary value
  problems with inhomogeneous data.
\newblock {\em Math. Z.}, 257(1):193--224, 2007.

\bibitem{Denk&Kaip2013}
R.~Denk and M.~Kaip.
\newblock {\em General parabolic mixed order systems in {${L_p}$} and
  applications}, volume 239 of {\em Operator Theory: Advances and
  Applications}.
\newblock Birkh\"auser/Springer, Cham, 2013.

\bibitem{Dintelmann1995}
P.~Dintelmann.
\newblock Classes of {F}ourier multipliers and {B}esov-{N}ikolskij spaces.
\newblock {\em Math. Nachr.}, 173:115--130, 1995.

\bibitem{Garcia-Cuerva&Macias&Torrea1993}
J.~Garc\'{i}a-Cuerva, R.~Mac\'{i}as, and J.~L. Torrea.
\newblock The {H}ardy-{L}ittlewood property of {B}anach lattices.
\newblock {\em Israel J. Math.}, 83(1-2):177--201, 1993.

\bibitem{Hanninen&Lorist2019}
T.S. H\"{a}nninen and E.~Lorist.
\newblock Sparse domination for the lattice {H}ardy--{L}ittlewood maximal
  operator.
\newblock {\em Proc. Amer. Math. Soc.}, 147(1):271--284, 2019.

\bibitem{Haroske&Piotrowska2008}
D.D. Haroske and I.~Piotrowska.
\newblock Atomic decompositions of function spaces with {M}uckenhoupt weights,
  and some relation to fractal analysis.
\newblock {\em Math. Nachr.}, 281(10):1476--1494, 2008.

\bibitem{Haroske&Skrzypczak2008_EntropyI}
D.D. Haroske and L.~Skrzypczak.
\newblock Entropy and approximation numbers of embeddings of function spaces
  with {M}uckenhoupt weights. {I}.
\newblock {\em Rev. Mat. Complut.}, 21(1):135--177, 2008.

\bibitem{Haroske&Skrzypczak2011_EntropyII}
D.D. Haroske and L.~Skrzypczak.
\newblock Entropy and approximation numbers of embeddings of function spaces
  with {M}uckenhoupt weights, {II}. {G}eneral weights.
\newblock {\em Ann. Acad. Sci. Fenn. Math.}, 36(1):111--138, 2011.

\bibitem{Haroske&Skrzypczak2011_EntropyIII}
D.D. Haroske and L.~Skrzypczak.
\newblock Entropy numbers of embeddings of function spaces with {M}uckenhoupt
  weights, {III}. {S}ome limiting cases.
\newblock {\em J. Funct. Spaces Appl.}, 9(2):129--178, 2011.

\bibitem{Hedberg&Netrusov2007}
L.S. Hedberg and Y.~Netrusov.
\newblock An axiomatic approach to function spaces, spectral synthesis, and
  {L}uzin approximation.
\newblock {\em Mem. Amer. Math. Soc.}, 188(882):vi+97, 2007.

\bibitem{Hytonen&Neerven&Veraar&Weis2016_Analyis_in_Banach_Spaces_I}
T.P. Hyt\"onen, J.M.A.M.~van Neerven, M.C. Veraar, and L.~Weis.
\newblock {\em Analysis in {B}anach spaces. {V}ol. {I}. {M}artingales and
  {L}ittlewood-{P}aley theory}, volume~63 of {\em Ergebnisse der Mathematik und
  ihrer Grenzgebiete. 3. Folge.}
\newblock Springer, 2016.

\bibitem{Hytonen&Neerven&Veraar&Weis2016_Analyis_in_Banach_Spaces_II}
T.P. Hyt\"onen, J.M.A.M.~van Neerven, M.C. Veraar, and L.~Weis.
\newblock {\em Analysis in {B}anach spaces. {V}ol. {II}. {P}robabilistic
  {M}ethods and {O}perator {T}heory.}, volume~67 of {\em Ergebnisse der
  Mathematik und ihrer Grenzgebiete. 3. Folge.}
\newblock Springer, 2017.

\bibitem{Johnsen&Mucn_Hansen&Sickel2015}
J.~Johnsen, S.~Munch~Hansen, and W.~Sickel.
\newblock Anisotropic {L}izorkin-{T}riebel spaces with mixed norms---traces on
  smooth boundaries.
\newblock {\em Math. Nachr.}, 288(11-12):1327--1359, 2015.

\bibitem{JS_traces}
J.~Johnsen and W.~Sickel.
\newblock On the trace problem for {L}izorkin-{T}riebel spaces with mixed
  norms.
\newblock {\em Math. Nachr.}, 281(5):669--696, 2008.

\bibitem{Kalton2003}
N.~Kalton.
\newblock Quasi-{B}anach spaces.
\newblock In {\em Handbook of the geometry of {B}anach spaces, {V}ol. 2}, pages
  1099--1130. North-Holland, Amsterdam, 2003.

\bibitem{Kalton&Peck&James1984}
N.~J. Kalton, N.~T. Peck, and James~W. Roberts.
\newblock {\em An {$F$}-space sampler}, volume~89 of {\em London Mathematical
  Society Lecture Note Series}.
\newblock Cambridge University Press, Cambridge, 1984.

\bibitem{Krein&Petun&Semenov1982}
S.~Kre\u{\i}n, Y.~Petun\={\i}n, and E.~Sem\"{e}nov.
\newblock {\em Interpolation of linear operators}, volume~54 of {\em
  Translations of Mathematical Monographs}.
\newblock American Mathematical Society, Providence, R.I., 1982.
\newblock Translated from the Russian.

\bibitem{Kunstmann&Ullmann2014}
P.C. Kunstmann and A.~Ullmann.
\newblock $\mathcal{R}_{s}$-sectorial operators and generalized
  {T}riebel-{L}izorkin spaces.
\newblock {\em J. Fourier Anal. Appl.}, 20(1):135--185, 2014.

\bibitem{Lindemulder_master-thesis}
N.~Lindemulder.
\newblock Parabolic {I}nitial-{B}oundary {V}alue {P}roblems with
  {I}nhomoegeneous {D}ata: A weighted maximal regularity approach.
\newblock Master's thesis, Utrecht University, 2014.

\bibitem{Lindemulder2016_JFA}
N.~Lindemulder.
\newblock Difference norms for vector-valued {B}essel potential spaces with an
  application to pointwise multipliers.
\newblock {\em J. Funct. Anal.}, 272(4):1435--1476, 2017.

\bibitem{Lindemulder2018_DSOP}
N.~{Lindemulder}.
\newblock {Second Order Operators Subject to Dirichlet Boundary Conditions in
  Weighted Triebel-Lizorkin Spaces: Parabolic Problems}.
\newblock {\em ArXiv e-prints (arXiv:1812.05462)}, December 2018.

\bibitem{lindemulder2017maximal}
N.~Lindemulder.
\newblock Maximal regularity with weights for parabolic problems with
  inhomogeneous boundary conditions.
\newblock {\em J. Evol. Equ.}, 20(1):59--108, 2020.

\bibitem{Lindemulder&Veraar2018}
N.~Lindemulder and M.C. Veraar.
\newblock The heat equation with rough boundary conditions and holomorphic
  functional calculus.
\newblock {\em J. Differential Equations}, 269(7):5832--5899, 2020.

\bibitem{Marschall1996}
J.~Marschall.
\newblock Nonregular pseudo-differential operators.
\newblock {\em Z. Anal. Anwendungen}, 15(1):109--148, 1996.

\bibitem{Meyries&Veraar2012_sharp_embedding}
M.~Meyries and M.C. Veraar.
\newblock Sharp embedding results for spaces of smooth functions with power
  weights.
\newblock {\em Studia Math.}, 208(3):257--293, 2012.

\bibitem{Meyries&Veraar2014_traces}
M.~Meyries and M.C. Veraar.
\newblock Traces and embeddings of anisotropic function spaces.
\newblock {\em Math. Ann.}, 360(3-4):571--606, 2014.

\bibitem{Meyries&Veraar2015_pointwise_multiplication}
M.~Meyries and M.C. Veraar.
\newblock Pointwise multiplication on vector-valued function spaces with power
  weights.
\newblock {\em J. Fourier Anal. Appl.}, 21(1):95--136, 2015.

\bibitem{Netrusov1989b}
Y.V. Netrusov.
\newblock Metric estimates for the capacities of sets in {B}esov spaces.
\newblock volume 190, pages 159--185. 1989.
\newblock Translated in Proc. Steklov Inst. Math. {{\bf{1}}992}, no. 1,
  167--192, Theory of functions (Russian) (Amberd, 1987).

\bibitem{Netrusov1989}
Y.V. Netrusov.
\newblock Sets of singularities of functions in spaces of {B}esov and
  {L}izorkin-{T}riebel type.
\newblock volume 187, pages 162--177. 1989.
\newblock Translated in Proc. Steklov Inst. Math. {{\bf{1}}990}, no. 3,
  185--203, Studies in the theory of differentiable functions of several
  variables and its applications, 13 (Russian).

\bibitem{Nowak2000}
M.~Nowak.
\newblock Strong topologies on vector-valued function spaces.
\newblock {\em Czechoslovak Math. J.}, 50(125)(2):401--414, 2000.

\bibitem{Rafeiro&Samento&Samko2013}
H.~Rafeiro, N.~Samko, and S.~Samko.
\newblock Morrey-{C}ampanato spaces: an overview.
\newblock In {\em Operator theory, pseudo-differential equations, and
  mathematical physics}, volume 228 of {\em Oper. Theory Adv. Appl.}, pages
  293--323. Birkh\"{a}user/Springer Basel AG, Basel, 2013.

\bibitem{Rolewicz1957}
S.~Rolewicz.
\newblock On a certain class of linear metric spaces.
\newblock {\em Bull. Acad. Polon. Sci. Cl. III.}, 5:471--473, XL, 1957.

\bibitem{Rubio_de_Francia1986}
J.L. Rubio~de Francia.
\newblock Martingale and integral transforms of {B}anach space valued
  functions.
\newblock In {\em Probability and {B}anach spaces ({Z}aragoza, 1985)}, volume
  1221 of {\em Lecture Notes in Math.}, pages 195--222. Springer, Berlin, 1986.

\bibitem{Ru91}
W.~Rudin.
\newblock {\em Functional analysis}.
\newblock International Series in Pure and Applied Mathematics. McGraw-Hill
  Inc., New York, second edition, 1991.

\bibitem{Runst&Sickel1996}
Thomas Runst and Winfried Sickel.
\newblock {\em Sobolev spaces of fractional order, {N}emytskij operators, and
  nonlinear partial differential equations}, volume~3 of {\em De Gruyter Series
  in Nonlinear Analysis and Applications}.
\newblock Walter de Gruyter \& Co., Berlin, 1996.

\bibitem{Scharf&Schmeisser&Sickel_Traces_vector-valued_Sobolev}
B.~Scharf, H-J. Schmei{\ss}er, and W.~Sickel.
\newblock Traces of vector-valued {S}obolev spaces.
\newblock {\em Math. Nachr.}, 285(8-9):1082--1106, 2012.

\bibitem{Schwartz1966}
L.~Schwartz.
\newblock {\em Th\'{e}orie des distributions}.
\newblock Publications de l'Institut de Math\'{e}matique de l'Universit\'{e} de
  Strasbourg, No. IX-X. Nouvelle \'{e}dition, enti\'{e}rement corrig\'{e}e,
  refondue et augment\'{e}e. Hermann, Paris, 1966.

\bibitem{Sickel&Skrzypczak&Vybiral2014}
W.~Sickel, L.~Skrzypczak, and J.~Vyb\'{i}ral.
\newblock Complex interpolation of weighted {B}esov and {L}izorkin-{T}riebel
  spaces.
\newblock {\em Acta Math. Sin. (Engl. Ser.)}, 30(8):1297--1323, 2014.

\bibitem{Strichartz1967}
R.S. Strichartz.
\newblock Multipliers on fractional {S}obolev spaces.
\newblock {\em J. Math. Mech.}, 16:1031--1060, 1967.

\bibitem{Tozoni1995}
S.A. Tozoni.
\newblock Weighted inequalities for vector operators on martingales.
\newblock {\em J. Math. Anal. Appl.}, 191(2):229--249, 1995.

\bibitem{Triebel1983_TFS_I}
H.~Triebel.
\newblock {\em Theory of function spaces}, volume~78 of {\em Monographs in
  Mathematics}.
\newblock Birkh\"auser Verlag, Basel, 1983.

\bibitem{Triebel1992_TFS_II}
H.~Triebel.
\newblock {\em Theory of function spaces. {II}}, volume~84 of {\em Monographs
  in Mathematics}.
\newblock Birkh\"auser Verlag, Basel, 1992.

\bibitem{Triebel2001_SF}
H.~Triebel.
\newblock {\em The structure of functions}.
\newblock Modern Birkh\"auser Classics. Birkh\"auser/Springer Basel AG, Basel,
  2001.
\newblock [2012 reprint of the 2001 original] [MR1851996].

\bibitem{Triebel2011_Fractals}
H.~Triebel.
\newblock {\em Fractals and spectra}.
\newblock Modern Birkh\"{a}user Classics. Birkh\"{a}user Verlag, Basel, 2011.
\newblock Related to Fourier analysis and function spaces.

\bibitem{Weidemaier2002}
P.~Weidemaier.
\newblock Maximal regularity for parabolic equations with inhomogeneous
  boundary conditions in {S}obolev spaces with mixed {$L_p$}-norm.
\newblock {\em Electron. Res. Announc. Amer. Math. Soc.}, 8:47--51, 2002.

\bibitem{Whitney1934}
H.~Whitney.
\newblock Derivatives, difference quotients, and {T}aylor's formula.
\newblock {\em Bull. Amer. Math. Soc.}, 40(2):89--94, 1934.

\bibitem{Whitney1948}
H.~Whitney.
\newblock On ideals of differentiable functions.
\newblock {\em Amer. J. Math.}, 70:635--658, 1948.

\end{thebibliography}
\end{document}